\documentclass[12pt]{article}
\usepackage[T1]{fontenc}
\usepackage{mathrsfs,amsthm,graphicx,color,verbatim,bbm,amsmath,amsfonts,amssymb,newclude,nicefrac,graphicx,enumerate,hyperref,bm,geometry,mathabx,stmaryrd}
\usepackage{todonotes}
\usepackage[capitalise]{cleveref} 

\theoremstyle{plain}
\newtheorem{theorem}{Theorem}[section]
\newtheorem{lemma}[theorem]{Lemma}
\newtheorem{prop}[theorem]{Proposition}
\newtheorem{cor}[theorem]{Corollary}

\theoremstyle{remark}

\theoremstyle{definition}
\newtheorem{definition}[theorem]{Definition}

\newcommand{\E}{\mathbb{E}}
\newcommand{\F}{\mathbb{F}}
\renewcommand{\P}{\mathbb{P}}

\newcommand{\R}{\mathbb{R}}
\newcommand{\N}{\mathbb{N}}
\newcommand{\Sym}{\mathbb{S}}
\newcommand{\V}{\mathbb{V}}

\newcommand{\Borel}{\mathcal{B}}
\newcommand{\cC}{\mathcal{C}}
\newcommand{\cF}{\mathcal{F}}
\newcommand{\cK}{\mathcal{K}}

\newcommand{\cO}{\mathcal{O}}

\newcommand{\smallsum}{\textstyle\sum}

\newcommand{\Exp}[1]{ \E \! \left[ #1 \right]}

\newcommand{\EXP}[1]{ \E  [ #1 ]}
\newcommand{\EXPP}[1]{ \E \big[ #1 \big]}

\newcommand{\norm}[1]{ \left\| #1 \right\| }
\newcommand{\Norm}[1]{ \| #1 \| }

\newcommand{\HSnorm}[1]{ \left\vvvert #1 \right\vvvert }

\newcommand{\qandq}{\qquad\text{and}\qquad}

\newcommand{\Exists}{\exists\,}
\newcommand{\Forall}{\forall\,}

\newcommand{\mc}{\mathcal}
\newcommand{\mf}{\mathfrak}
\newcommand{\is}{\leftarrow}
\newcommand{\Closure}[1]{\overline{#1}}
\newcommand{\JetClosure}{\mf P}

\makeatletter
\newcommand{\vast}{\bBigg@{3.5}}
\newcommand{\Vast}{\bBigg@{4}}
\makeatother

\newcounter{AuthorCount}
\stepcounter{AuthorCount}

\begin{document}

\title{On nonlinear Feynman--Kac formulas \\
	for viscosity solutions of semilinear \\
	parabolic partial differential equations}

\author{
	Christian Beck$^{\arabic{AuthorCount}}\stepcounter{AuthorCount}$, 
	Martin Hutzenthaler$^{\arabic{AuthorCount}\stepcounter{AuthorCount}}$, and 
	Arnulf Jentzen$^{\arabic{AuthorCount}\stepcounter{AuthorCount},\arabic{AuthorCount}\stepcounter{AuthorCount}}$
	\bigskip
	\setcounter{AuthorCount}{1}
	\\
	\small{$^{\arabic{AuthorCount}
			\stepcounter{AuthorCount}}$ 
		Department of Mathematics, 
		ETH Zurich, 
		Z\"urich,}\\
	\small{Switzerland, 
		e-mail: christian.beck@math.ethz.ch} 
	\\
	\small{$^{\arabic{AuthorCount}
			\stepcounter{AuthorCount}}$ 
		Faculty of Mathematics, University of Duisburg-Essen, } \\
	\small{Essen, Germany, e-mail: martin.hutzenthaler@uni-due.de}
	\\
	\small{$^{\arabic{AuthorCount}
			\stepcounter{AuthorCount}}$ 
		Department of Mathematics, 
		ETH Zurich, 
		Z\"urich,}\\
	\small{Switzerland, 
		e-mail: arnulf.jentzen@sam.math.ethz.ch} 
	\\
	\small{$^{\arabic{AuthorCount}\stepcounter{AuthorCount}}$ 
		Faculty of Mathematics and Computer Science, 
		University of M\"unster, }\\
	\small{M\"unster, Germany, e-mail: ajentzen@uni-muenster.de}		
}

\maketitle

\begin{abstract}
The classical Feynman--Kac identity builds a bridge between stochastic analysis and partial differential equations (PDEs) by providing stochastic representations for classical solutions of linear Kolmogorov PDEs. 
This opens the door for the derivation of sampling based Monte Carlo approximation methods, which can be meshfree and thereby stand a chance to approximate solutions of PDEs without suffering from the curse of dimensionality. 
In this article we extend the classical Feynman--Kac formula to certain semilinear Kolmogorov PDEs. 
More specifically, we identify suitable solutions of stochastic fixed point equations (SFPEs), which arise when the classical Feynman--Kac identity is formally applied to semilinear Kolmorogov PDEs, as viscosity solutions of the corresponding PDEs. 
This justifies, in particular, employing full-history recursive multilevel Picard (MLP) approximation algorithms, which have recently been shown to overcome the curse of dimensionality in the numerical approximation of solutions of SFPEs, in the numerical approximation of semilinear Kolmogorov PDEs.
\end{abstract}

\tableofcontents

\section{Introduction}
\label{sect:intro}

The classical Feynman--Kac identity (see, e.g., \cite{GiSk1972_SDEs,HaHuJe2017_LossOfRegularityKolmogorov,
	KaSh1991_BrownianMotionAndStochasticCalculus,
	PaRa2014_SDEsBackwardSDEsAndPDEs}) 
builds a bridge between stochastic analysis and partial differential equations (PDEs) by providing stochastic representations for classical solutions of linear Kolmogorov PDEs. 
The fact that certain solutions of linear Kolmogorov PDEs can be expressed as appropriate averages of It\^o processes associated with these PDEs opens the door for the derivation of sampling based Monte Carlo approximation methods, which can be meshfree and thereby stand a chance to approximate solutions of PDEs without suffering from the curse of dimensionality. 
Since PDEs in applications are not always linear, an extension of the classical Feynman--Kac formula to nonlinear PDEs is desirable. 
One approach to nonlinear Feynman--Kac type formulas passes through backward stochastic differential equations (BSDEs); see, e.g., \cite{Bismut1973,PardouxPeng1990} for references on BSDEs and see, e.g., \cite{Antonelli1993FBSDEs,
	BarlesBuckdahnPardoux1997BSDEsAndIntegralPartialDifferentialEquations,
	HuPeng1995SolutionOfFBSDEs,
	HuYong2000FBSDEsWithNonsmoothCoefficients,
	MaProtterYong1994FourStepScheme,
	MaZhang2002RepresentationTheoremsForBSDEs,
	Pardoux1998BSDEsAndViscositySolutionsOfSystems,
	PardouxPeng1992,
	PardouxPradeillesRao1997ProbabilisticInterpretation,
	PaRa2014_SDEsBackwardSDEsAndPDEs,
	PardouxTang1999,
	Peng1991ProbabilisticInterpretation,
	SowPardoux2004ProbabilisticInterpretation} for references on the connection between BSDEs and PDEs. 
The approach which is pursued in this article is to identify suitable solutions of stochastic fixed point equations (SFPEs), which arise when the classical Feynman--Kac identity is formally applied to semilinear Kolmorogov PDEs by treating the nonlinearity as mere inhomogeneity, as viscosity solutions of the corresponding PDEs; see, e.g., 
	\cite{CrandallEvansLions1984,UsersGuide,CrandallLions1983ViscosityHJB,HaHuJe2017_LossOfRegularityKolmogorov,ImbertSilvestre_FullyNonlinearParabolic} 
for references on viscosity solutions of PDEs.
More specifically, we establish in this article a one-to-one correspondence between viscosity solutions of certain semilinear Kolmogorov PDEs and solutions of the associated SFPEs (see \cref{thm:existence_of_fixpoint} in \cref{subsec:viscosity_semilinear} below).
This justifies, in particular, employing full-history recursive multilevel Picard (MLP) approximation algorithms (see \cite{beck2019overcoming,EHutzenthalerJentzenKruse2019MLP,GilesJentzenWelti2019GeneralisedMLP,HutzenthalerJentzenKruse2019GradientDependent,hutzenthaler2016multilevel,Overcoming,hutzenthaler2019overcoming,hutzenthaler2017multi} for references on MLP approximation algorithms), which have been shown to overcome the curse of dimensionality in the numerical approximation of solutions of SFPEs, in the numerical approximation of semilinear Kolmogorov PDEs. 
MLP approximation algorithms are the first and up to now only methods which have been shown to overcome the curse of dimensionality in the numerical approximation of solutions of semilinear Kolmogorov PDEs. 
To illustrate the findings of this article, we now present in \cref{thm:intro} below a special case of \cref{thm:existence_of_fixpoint} which is the main result of this article. 

\begin{theorem}\label{thm:intro}
	Let $ d \in \N $,
		$ L,T \in (0,\infty) $, 
	let $ \mu \colon \R^d \to \R^d $ and $ \sigma \colon \R^d \to \R^{d\times d} $ be locally Lipschitz continuous, 
	let	$ f \in C( \R^d \times \R , \R ) $, 
		$ g \in C( \R^d, \R ) $ 
	be at most polynomially growing, 
	let $ \norm{\cdot}\colon\R^d\to [0,\infty)$ be the standard Euclidean norm on $\R^d$, 
	let $ \langle\cdot,\cdot\rangle \colon \R^d\times\R^d\to\R $ be the standard Euclidean scalar product on $\R^d$, 
	assume for all 
		$ x,y \in \R^d $, 
		$ v,w \in \R $  
	that
		$ \langle x,\mu(x)\rangle \leq  L ( 1 + \norm{x}^2 ) $,
		$ \Norm{\sigma(x)y} \leq L ( 1+\Norm{x} ) \Norm{y} $, and 
		$ | f( x, v ) - f( x, w ) | \leq L | v - w | $,
	let $ ( \Omega, \mathcal{F}, \P, (\mathbb{F}_t)_{t\in [0,T]} ) $ be a stochastic basis\footnote{Note that we say that a filtered probability space $(\Omega,\mc F,\P,(\F_t)_{t\in [0,T]})$ is a stochastic basis if and only if we have for all 
		$ t \in [0,T) $ 
	that 
		$ \{ A \in \mc F\colon \P(A) = 0 \} \subseteq \F_t = (\cap_{s\in (t,T]} \F_s) $; cf., e.g., Liu \& R\"ockner~\cite[Definition 2.1.11]{LiuRoeckner2015SPDEs}.}, 
	and let $W \colon [0,T] \times \Omega \to \R^d$ be a standard $(\mathbb{F}_t)_{t\in [0,T]}$-Brownian motion. 
	Then 
		\begin{enumerate}[(i)]
		\item \label{intro:item2}
		there exists a unique at most polynomially growing viscosity solution $ u \in C([0,T] \times \R^d, \R) $ of 
			\begin{equation} \label{intro:pde}
			(\tfrac{\partial}{\partial t}u)(t,x) 
			+ \tfrac{1}{2} \operatorname{Trace}\!\left( \sigma(x)[\sigma(x)]^{\ast}(\operatorname{Hess}_x u )(t,x) \right) 
			+ \langle \mu(x), (\nabla_x u)(t,x)\rangle 
			+ f(x, u(t, x)) 
			= 0
			\end{equation}
		with $u(T,x) = g(x)$ for $(t,x) \in (0,T) \times \R^d$,	
		\item \label{intro:item1}
		for every 
			$t\in [0,T]$, 
			$x\in\R^d$ 
		there exists an up to indistinguishability unique $(\mathbb{F}_s)_{s\in [t,T]}$-adapted stochastic process
			$X^{t,x}=(X^{t,x}_s)_{s\in [t,T]}\colon [t,T]\times\Omega\to\R^d$
		with continuous sample paths satisfying that for all
			$s\in [t,T]$ 
		we have $\P$-a.s.~that 
			\begin{equation} \label{intro:sde}
			X^{t,x}_s
			= 
			x + \int_t^s \mu( X^{t,x}_r )\, dr + \int_t^s \sigma( X^{t,x}_r ) \, dW_r,
			\end{equation}
		\item \label{intro:item3}
		there exists a unique at most polynomially growing $ v \in C([0,T]\times\R^d, \R)$ 
		which satisfies for all 
			$t\in [0,T]$, 
			$x\in\R^d$ 
		that 
			$ \EXP{|g(X^{t,x}_T)| + \int_t^T |f(X^{t,x}_s,v(s,X^{t,x}_s))|\,ds} < \infty $
		and 
			\begin{equation} \label{intro:sfpe}
			v(t, x)
			=
			\Exp{
				g ( X^{t,x}_{T} )
				+
				\int_{t}^T
				f (X^{t,x}_s,  v(s, X^{t,x}_s) )
				\, ds
			}\!, 
			\end{equation}
		and 
		\item\label{intro:item4} 
		we have for all 
			$ t \in [0,T] $, 
			$ x \in \R^d $ 
		that 
			$ u(t,x) = v(t,x) $. 
	\end{enumerate}		
\end{theorem}
\cref{thm:intro} above is an immediate consequence of \cref{existence_of_fixpoint_polynomial_growth}. 
\cref{existence_of_fixpoint_polynomial_growth}, in turn, follows from \cref{cor:existence_of_fixpoint_product_lyapunov} which itself is a special case of \cref{thm:existence_of_fixpoint}, the main result of this article.  
Let us comment on some of the mathematical objects appearing in \cref{thm:intro}. 
The real number $ L \in (0,\infty) $ in \cref{thm:intro} above is used to formulate a growth condition, a coercivity type condition, and a Lipschitz continuity condition on the functions $ \mu \colon \R^d \to \R^d $, $ \sigma \colon \R^d \to \R^{d\times d} $, and $ f \colon \R^d\times\R \to \R $. 
The real number $ T \in (0,\infty) $ specifies the time horizon of the PDE in \eqref{intro:pde} in \cref{thm:intro} above. 
The functions $ \mu \colon \R^d \to \R^d $ and $ \sigma \colon \R^d \to \R^{d\times d} $ in \cref{thm:intro} above determine the random dynamics in \eqref{intro:sde} and specify the linear part of the PDE in \eqref{intro:pde}. 
The assumption that the functions $ \mu \colon \R^d \to \R^d $ and $ \sigma \colon \R^d \to \R^{d\times d} $ in \cref{thm:intro} are locally Lipschitz continuous, the assumption that $ \sigma $ is at most linearly growing, and the assumption that $ \mu $ satisfies a coercivity type condition, i.e., the assumption that for all $ x \in \R^d $ we have that $ \langle x,\mu(x) \rangle \leq L (1 + \norm{x}^2) $, roughly speaking, prevent local solutions of the stochastic differential equation (SDE) in \eqref{intro:sde} from blowing up (see, e.g., Gy\"ongy \& Krylov~\cite{GyoengyKrylov1995_existenceStrong}). 
The function $ f \in C(\R^d\times\R,\R) $ in \cref{thm:intro} represents the nonlinearity of the semilinear Kolmogorov PDE in \eqref{intro:pde}. 
The function $ g \in C(\R^d,\R) $ in \cref{thm:intro}, in turn, specifies the terminal condition of the semilinear Kolmogorov PDE in \eqref{intro:pde}.
\cref{thm:intro} proves, in particular, the unique existence of an at most polynomially growing viscosity solution $u$ of the PDE in  \eqref{intro:pde} and, moreover, shows that $u$ is the unique at most polynomially growing solution of the SFPE in \eqref{intro:sfpe}. 
Related results can be found, e.g., in 
	El Karoui et al.~\cite[Theorem 8.5]{ElKarouiKapoudjianPardouxPeng1997RBSDEsAndObstacleProblems}, 
	Kalinin~\cite[Theorem 2.3]{KalininSchied2016MildAndViscositySolutionsToSemilinearParbabolicPPDEs}, 
	Ma \& Zhang~\cite[Theorem 4.2]{MaZhang2002RepresentationTheoremsForBSDEs},
	Pardoux~\cite[Theorem 4.6]{Pardoux1998BSDEsAndViscositySolutionsOfSystems}, 
	Pardoux \& Peng~\cite[Theorem 4.3]{PardouxPeng1992}, 
	Pardoux \& Tang~\cite[Theorem 5.1]{PardouxTang1999}, 
	Pardoux et al.~\cite[Theorem 4.1]{PardouxPradeillesRao1997ProbabilisticInterpretation}, 
and
	Peng~\cite[Theorem 3.2]{Peng1991ProbabilisticInterpretation}. 
Note that, roughly speaking, these results are, on the one hand, more general than \cref{thm:intro} above with regard to the assumptions on the nonlinearity $f$ in \cref{thm:intro} and, on the other hand, less general than \cref{thm:intro} above with regard to the assumptions on the coefficients  $\mu$ and $\sigma$ of the SDE in \eqref{intro:sde} in \cref{thm:intro} above. 
In addition, observe that in general the viscosity solution $ u \in C([0,T]\times\R^d,\R) $ of the PDE in \eqref{intro:pde} fails to be a classical solution of the PDE in \eqref{intro:pde}. 
Indeed, Hutzenthaler et al.~\cite{HaHuJe2017_LossOfRegularityKolmogorov} implies that there exist admissible choices for the functions $ \mu $, $ \sigma $, and $ g $ in \cref{thm:intro} such that the unique at most polynomially growing viscosity solution of the PDE in \eqref{intro:pde} with $ f = (\R^d\times\R \ni (x,a) \mapsto 0 \in \R) $ is not locally H\"older continuous (cf., e.g., Elworthy~\cite{Elworthy1978StochasticDynamicalSystemsAndTheirFlows} and Li \& Scheutzow~\cite{LiScheutzow2011LackOfStrongCompleteness} for related results). 
Next we comment on the proof of \cref{thm:intro}. 
Item~\eqref{intro:item1} is well-known in the scientific literature (see, e.g., Gy\"ongy \& Krylov~\cite{GyoengyKrylov1995_existenceStrong}) and Item~\eqref{intro:item3} follows from \cite[Corollary 3.10]{StochasticFixedPointEquations}. 
In order to prove Items~\eqref{intro:item2} and \eqref{intro:item4} we first show in \cref{prop:viscosity_solution_lyapunov} in \cref{subsec:viscosity_inhomogeneous} (see also the proof of \cref{thm:existence_of_fixpoint} in \cref{subsec:viscosity_semilinear}) that the unique at most polynomially growing solution of the SFPE in \eqref{intro:sfpe} is a viscosity solution of the PDE in \eqref{intro:pde} and, thereafter, we show in \cref{prop:uniqueness_viscosity_semilinear} in \cref{subsec:uniqueness_viscosity_semilinear} that there is at most one at most polynomially growing viscosity solution of the terminal value problem in Item~\eqref{intro:item2}. 

The remainder of this article is organized as follows.  
\cref{sec:linear_inhomogeneous_kolmogorov_pdes} is concerned with a Feynman--Kac representation result for viscosity solutions of linear inhomogeneous Kolmogorov PDEs; see \cref{prop:viscosity_solution_lyapunov} in \cref{subsec:viscosity_inhomogeneous}.
 \cref{prop:viscosity_solution_lyapunov} is proved by combining a well-known approximation argument (see \cref{cor:stability_result_for_viscosity_solutions} in \cref{subsec:approximation_of_viscosity_solutions}) with a well-known result for Feynman--Kac representations of classical solutions of linear inhomogeneous Kolmogorov PDEs (see \cref{lem:smooth_solutions} in \cref{subsec:smooth_solutions}) and an essentially well-known approximation result for SDEs (see \cref{lem:stability_for_sdes} in \cref{subsec:approximation_sdes}). 
The notion of viscosity solutions as well as some basic properties of viscosity solutions are recalled in \cref{subsec:def_and_elem_prop_viscosity_solutions}. 
\cref{sec:semilinear_kolmogorov_pdes} deals with existence, uniqueness, and Feynman--Kac representation results for viscosity solutions of semilinear Kolmogorov PDEs. 
In \cref{subsec:uniqueness_viscosity_semilinear} we establish suitable uniqueness results for suitable viscosity solutions of semilinear Kolmogorov PDEs (see \cref{prop:uniqueness_viscosity_semilinear} in \cref{subsec:uniqueness_viscosity_semilinear}). 
In  \cref{subsec:sde_existence} we reprove an essentially well-known existence result for solutions of SDEs which is originally due to Gy\"ongy \& Krylov~\cite{GyoengyKrylov1995_existenceStrong}.  
Finally, in \cref{subsec:viscosity_semilinear}, we combine the existence and uniqueness result for SFPEs in \cite[Theorem 3.8]{StochasticFixedPointEquations}, the Feynman--Kac representation result for viscosity solutions of linear inhomogeneous Kolmogorov PDEs in \cref{prop:viscosity_solution_lyapunov}, and the uniqueness result in \cref{prop:uniqueness_viscosity_semilinear} to establish  \cref{thm:existence_of_fixpoint}, the main result of this article. 
We conclude this article by presenting in  \cref{existence_of_fixpoint_polynomial_growth} and \cref{cor:existence_bounded_heat_equation_type} in \cref{subsec:viscosity_semilinear} below a few illustrative applications of  \cref{thm:existence_of_fixpoint}. 

\section{Linear inhomogeneous Kolmogorov partial differential equations (PDEs)}
\label{sec:linear_inhomogeneous_kolmogorov_pdes}

In this section we recall the definitions of a viscosity subsolution (see \cref{def:viscosity_subsolution} in \cref{subsec:def_and_elem_prop_viscosity_solutions} below), of a viscosity supersolution (see \cref{def:viscosity_supersolution} in \cref{subsec:def_and_elem_prop_viscosity_solutions} below), and of a viscosity solution (see \cref{def:viscosity_solution} in \cref{subsec:def_and_elem_prop_viscosity_solutions} below) in the case of a suitable class of degenerate parabolic PDEs, which in particular includes linear inhomogeneous Kolmogorov PDEs as special cases, and we establish in \cref{prop:viscosity_solution_lyapunov} in \cref{subsec:viscosity_inhomogeneous} below a Feynman--Kac type representation result for viscosity solutions of such linear inhomogeneous Kolmogorov PDEs. 
The Feynman--Kac type representation result in \cref{prop:viscosity_solution_lyapunov} in \cref{subsec:viscosity_inhomogeneous} will be employed in our proof of  \cref{thm:existence_of_fixpoint} in \cref{subsec:viscosity_semilinear} below, the main result of this article.
Our proof of \cref{prop:viscosity_solution_lyapunov}, in turn, is based on the combination of the following three essentially well-known results: 
(i) the existence and Feynman--Kac type representation result for classical solutions of certain linear inhomogeneous Kolmogorov PDEs in \cref{lem:smooth_solutions} in \cref{subsec:smooth_solutions} below, (ii) the approximation result for viscosity solutions of degenerate parabolic PDEs in \cref{cor:stability_result_for_viscosity_solutions} in \cref{subsec:approximation_of_viscosity_solutions} below, and (iii) the approximation result for solutions of SDEs in \cref{lem:stability_for_sdes} in \cref{subsec:approximation_sdes} below.

In \cref{subsec:smooth_solutions} we establish in the essentially well-known result in \cref{lem:smooth_solutions} that a linear inhomogeneous Kolmogorov PDE with smooth and compactly supported drift and diffusion coefficients, with a smooth terminal condition, and with a smooth inhomogeneity admits a  classical solution. 
For the sake of completeness we also provide in \cref{subsec:smooth_solutions} a detailed proof for \cref{lem:smooth_solutions}.
In \cref{subsec:def_and_elem_prop_viscosity_solutions} we specify in \cref{degenerate_elliptic,def:viscosity_solution} below the well-known notions of a degenerate elliptic function and  of a viscosity solution (cf.~also, for example, Crandall et al.~\cite[Sections 2 and 8]{UsersGuide}, Hairer et al.~\cite[Section 4.1 and Definition 4.1]{HaHuJe2017_LossOfRegularityKolmogorov}, and Peng~\cite[Definition 1.2 in Appendix C]{peng2010nonlinear}) which are used in this article. 
In addition, in \cref{subsec:def_and_elem_prop_viscosity_solutions} we also briefly recall in \cref{lem:classical_solutions_are_viscosity_solutions}, \cref{lem:local_maximum_enough}, \cref{lem:C2_test_functions_suffice},  \cref{lem:equivalent_conditions_for_viscosity_solutions:left_to_right}, and 
\cref{lem:equivalent_conditions_for_viscosity_solutions:right_to_left} some elementary and well-known properties of viscosity solutions which are employed later on in this article. 
In particular, \cref{lem:classical_solutions_are_viscosity_solutions} recalls that every classical solution is also a viscosity solution, 
\cref{lem:local_maximum_enough} recalls an equivalent characterization for the notion of a viscosity subsolution, 
\cref{lem:C2_test_functions_suffice} proves, roughly speaking, that under suitable assumptions the notion of a viscosity subsolution in \cref{def:viscosity_subsolution} is consistent with the notion of a viscosity subsolution in Hairer et al.~\cite[Definition 4.1]{HaHuJe2017_LossOfRegularityKolmogorov}, and  \cref{lem:equivalent_conditions_for_viscosity_solutions:left_to_right} and \cref{lem:equivalent_conditions_for_viscosity_solutions:right_to_left} provide an equivalent characterization for the notion of a viscosity subsolution based on the notion of a parabolic superjet, which we briefly recall in \cref{def:parabolic_superjets} (cf.~also Crandall et al.~\cite[Section 8]{UsersGuide} and Peng \cite[Appendix C]{peng2010nonlinear}).  
In \cref{subsec:approximation_of_viscosity_solutions} we establish in the essentially well-known results in \cref{lem:stability_of_subsolutions}, \cref{cor:stability_of_supersolutions}, and \cref{cor:stability_result_for_viscosity_solutions} (cf., e.g., Crandall et al.~\cite[Lemma 6.1]{UsersGuide}, Hairer et al.~\cite[Lemma 4.8]{HaHuJe2017_LossOfRegularityKolmogorov}, and Imbert \& Silvestre~\cite[Proposition 2.3.11]{ImbertSilvestre_FullyNonlinearParabolic}) approximation results for viscosity subsolutions, viscosity supersolutions, and viscosity solutions. 
Our proof of \cref{lem:stability_of_subsolutions} is strongly inspired by Hairer et al.~\cite[Lemma 4.8]{HaHuJe2017_LossOfRegularityKolmogorov} (cf.\ also Barles \& Perthame~\cite[Theorem A.2]{BarlesPerthame}), \cref{cor:stability_of_supersolutions} is a rather direct consequence of \cref{lem:stability_of_subsolutions}, and \cref{cor:stability_result_for_viscosity_solutions} follows immediately from \cref{lem:stability_of_subsolutions} and \cref{cor:stability_of_supersolutions}. 
In \cref{subsec:approximation_sdes} we recall in \cref{lem:stability_for_sdes} an essentially well-known result on the continuous dependence of solutions of SDEs on their initial values. 
In \cref{subsec:viscosity_inhomogeneous} we establish in \cref{prop:viscosity_solution_lyapunov} an existence result for viscosity solutions of linear inhomogeneous Kolmogorov PDEs.   
Our proof of \cref{prop:viscosity_solution_lyapunov} employs \cref{lem:viscosity_solution_compact_lipschitz} together with the approximation result for viscosity solutions of degenerate parabolic PDEs in \cref{cor:stability_result_for_viscosity_solutions}. 
\cref{lem:viscosity_solution_compact_lipschitz}, in turn, uses the existence and Feynman--Kac type representation result for classical solutions of linear inhomogeneous Kolmogorov PDEs in \cref{lem:smooth_solutions}, the approximation result for viscosity solutions of degenerate parabolic PDEs in \cref{cor:stability_result_for_viscosity_solutions}, and the approximation result for solutions of SDEs in \cref{lem:stability_for_sdes}.  

\subsection{Existence results for classical solutions of linear inhomogeneous Kolmogorov PDEs}
\label{subsec:smooth_solutions}

\begin{lemma} 
	\label{lem:integrability} 
	Let $ d,m \in \N $, 
		$ L,T \in (0,\infty) $, 
		$ \xi \in \R^d $, 
	let $ \norm{\cdot} \colon \R^d\to [0,\infty) $ be the standard Euclidean norm on $\R^d$, 
	let $ \HSnorm{\cdot} \colon\R^{d\times m}\to [0,\infty) $ be the Frobenius norm on $\R^{d\times m}$,  
	let $ \mu \in C([0,T]\times\R^d,\R^d)$, 
		$ \sigma \in C([0,T]\times\R^d,\R^{d\times m}) $ 
	satisfy for all 
		$ t \in [0,T] $, 
		$ x,y \in \R^d $ 
	that 
		$ \norm{\mu(t,x)-\mu(t,y)} + \HSnorm{\sigma(t,x)-\sigma(t,y)} \leq L \norm{x-y} $, 
	let $ g \colon \R^d \to \R $ be $ \Borel(\R^d) $/$ \Borel(\R) $-measurable, 
	let $ h \colon [0,T]\times\R^d \to \R $ be $ \Borel([0,T]\times\R^d) $/$ \Borel(\R) $-measurable, 
	let $ \mc O \subseteq \R^d $ be an open set which satisfies $ (\operatorname{supp}(\mu)\cup\operatorname{supp}(\sigma)) \subseteq [0,T]\times\mc O $, 
	assume that 
		$ \sup ( \{ |g(x)| + |h(t,x)| \colon t \in [0,T], x \in \Closure{\mc O} \} \cup \{ 0 \} ) < \infty $, 
	assume for all 
		$ x \in \R^d $ 
	that 
		$ \int_0^T |h(t,x)| \,dt < \infty $, 
	let $ ( \Omega, \mc F, \P, (\F_t)_{t\in [0,T]}) $ be a stochastic basis, 
	let $ W \colon [0,T]\times\Omega \to \R^m $ be a standard $ (\F_t)_{t\in [0,T]} $-Brownian motion, 
	and let 
		$ X = (X_t)_{t \in [0,T]} \colon [0,T]\times\Omega \to \R^d $ 
	be an $ (\F_t)_{t\in [0,T]} $-adapted stochastic process with continuous sample paths satisfying that for all 
		$ t \in [0,T] $ 
	we have $ \P $-a.s.~that 
		\begin{equation} 
		X_t = \xi + \int_0^t \mu(s,X_s)\,ds + \int_0^t \sigma(s,X_s)\,dW_s. 
		\end{equation} 
	Then 
		\begin{equation} \label{integrability:claim}
		\Exp{|g(X_T)| + \int_0^T |h(t,X_t)|\,dt} < \infty.
		\end{equation} 
\end{lemma} 

\begin{proof}[Proof of \cref{lem:integrability}]
	To prove \eqref{integrability:claim} we distinguish between the case $ \xi \in \R^d \setminus \mc O $ 
	and the case $ \xi \in \mc O $. 
 	We first prove \eqref{integrability:claim} in the case $ \xi \in \R^d \setminus \mc O $. 
	Note that the assumption that $ ( \operatorname{supp}(\mu) \cup \operatorname{supp}(\sigma) ) \subseteq [0,T]\times\mc O $ ensures that 
		$ \P( \Forall t\in [0,T]\colon X_t = \xi ) = 1 $ (cf., e.g., \cite[Item~(i) in Lemma 3.4]{StochasticFixedPointEquations}).  
	Combining this with the assumption that for all 
		$ x \in \R^d $ 
	we have that 
		$ \int_0^T |h(t,x)|\,dt < \infty $ 
	shows that 
		\begin{equation}
		\label{integrability:eq01}
		\Exp{ | g( X_T ) | + \int_0^T | h( t, X_{t} ) | \,dt }
		= | g(\xi) | + \int_0^T | h(t,\xi) | \,dt 
		< \infty. 
		\end{equation}
	This establishes \eqref{integrability:claim} in the case $ \xi \in \R^d\setminus\mc O $. 
	Next we prove \eqref{integrability:claim} in the case $ \xi \in \mc O $. 
	Observe that the assumption that $ ( \operatorname{supp}(\mu) \cup \operatorname{supp}(\sigma) ) \subseteq [0,T]\times\mc O $ yields that 
		$ \P(\Forall t\in [0,T]\colon X_{t}\in \Closure{\mc O}) = 1 $ (cf., e.g., \cite[Item~(ii) in Lemma 3.4]{StochasticFixedPointEquations}). 
	Combining this with the assumption that 
		$ \sup(\{ |g(x)|+|h(t,x)|\colon t\in [0,T], x\in\Closure{\mc O} \} \cup \{ 0 \} ) < \infty $ 
	assures that we have that 
		\begin{equation}
		\Exp{ | g( X_{T} ) | + \int_0^T | h( t, X_t ) | \,dt }
		\leq 
		\left[ \sup_{x\in\Closure{\mc O}} | g(x) | \right]
		+ 
		T 
		\left[ \sup_{t\in [0,T]} \sup_{x\in \Closure{\mc O}} | h(t,x) | \right]
		< \infty .
		\end{equation}
	This establishes~\eqref{integrability:claim} in the case $ \xi \in \mc O $. 
	This completes the proof of \cref{lem:integrability}. 
\end{proof} 

\begin{lemma} 
	\label{lem:smooth_solutions}
	Let $ d,m \in \N $, 
		$ T \in (0,\infty) $,
	let $ \langle\cdot,\cdot\rangle\colon\R^d\times\R^d\to\R $ be the standard Euclidean scalar product on $\R^d$, 
	let $ \mu \colon [0,T]\times\R^d\to\R^d $ and $ \sigma \colon [0,T]\times\R^d\to\R^{d\times m} $ be infinitely often differentiable functions with compact support, 
	let $ g \colon \R^d\to \R $ and $ h \colon [0,T]\times \R^d\to\R $ be infinitely often differentiable functions, 
	let $ (\Omega,\mathcal{F},\P,(\mathbb{F}_t)_{t\in [0,T]}) $ be a stochastic basis, 
	let $ W \colon [0,T]\times\Omega\to\R^m $ be a standard $(\mathbb{F}_t)_{t\in [0,T]}$-Brownian motion,  
	for every 
		$ t \in [0,T] $, 
		$ x \in \R^d $
	let 
		$ X^{t,x} = (X^{t,x}_s)_{s\in [t,T]}\colon [t,T]\times\Omega\to\R^d $ 
	be an $(\mathbb{F}_s)_{s\in [t,T]}$-adapted stochastic process with continuous sample paths satisfying that for all 
		$s\in [t,T]$ 
	we have $\P$-a.s.~that 
		\begin{equation}
		\label{smooth_solutions:ass1}
		X^{t,x}_s 
		= 
		x 
		+ 
		\int_t^s \mu(r,X^{t,x}_r)\,dr 
		+ 
		\int_t^s \sigma(r,X^{t,x}_r) \,dW_r,  
		\end{equation}
	and let 
		$ u \colon [0,T]\times\R^d \to \R $
	satisfy for all 
		$ t \in [0,T] $, 
		$ x \in \R^d $ 
	that 
		\begin{equation}
		\label{smooth_solutions:ass2}
		u(t,x) = 
		\Exp{ g( X^{t,x}_{T} ) 
			+ 
			\int_t^T h( s, X^{t,x}_{s} ) \,ds }
		\end{equation}
	(cf.~\cref{lem:integrability}). 
	Then 
		\begin{enumerate}[(i)] 
		\item \label{smooth_solutions:item1}
		we have that 
			$u\in C^{1,2}([0,T]\times\R^d,\R)$ 
		and	
		\item \label{smooth_solutions:item2} 
		we have for all 
			$t\in [0,T]$, 
			$x\in\R^d$ 
		that 
			$u(T,x) = g(x)$
		and 
			\begin{equation}
			\label{smooth_solutions:semilinear_kolmogorov_equation}
			(\tfrac{\partial}{\partial t}u)(t,x) 
			+ \tfrac12 \operatorname{Trace}\!\left( \sigma(t,x)[\sigma(t,x)]^{*} (\operatorname{Hess}_x u)(t,x) \right) 
			+ \langle \mu(t,x),(\nabla_x u)(t,x) \rangle 
			+ h(t,x) = 0 .
			\end{equation}
		\end{enumerate}
\end{lemma}

\begin{proof}[Proof of \cref{lem:smooth_solutions}] 
	Throughout this proof let 
		$ o \in (0,\infty) $, 
		$ \mf o \in (o,\infty) $, 
	assume that 
		$ (\operatorname{supp}(\mu)\cup\operatorname{supp}(\sigma)) \subseteq [0,T] \times (-o,o)^d $, 
	let $ \langle\langle\cdot,\cdot\rangle\rangle \colon \R^{d+1}\times \R^{d+1} \to \R$ be the standard Euclidean scalar product on $ \R^{d+1} $,
	let $ \mathfrak{m}\colon\R^{d+1}\to\R^{d+1} $, 
		$ \mathfrak{s}\colon\R^{d+1}\to\R^{(d+1)\times m} $, 
		$ \mathfrak{g}\colon\R^{d+1}\to\R $, 
	and
		$ \mathfrak{h}\colon\R^{d+1}\to\R $
	be infinitely often differentiable functions with bounded derivatives which satisfy for all 
		$ t \in [0,T] $, 
		$ x \in \R^d $, 
		$ y \in [-\mf o,\mf o]^d $ 
	that
		\begin{equation}
		\label{smooth_solutions:autonomized_coeffs}
		\begin{gathered}
		\mathfrak{m}(t,x) 
		= \begin{pmatrix}
		1 \\ \mu(t,x)
		\end{pmatrix}\in\R^{d+1}, \qquad
		\mathfrak{s}(t,x) 
		= \begin{pmatrix}
		0 \\ \sigma(t,x)
		\end{pmatrix}
		\in\R^{(d+1)\times m}, \\  
		\mathfrak{g}(t,y) = g(y) \in \R, \qquad\text{and}\qquad
		\mathfrak{h}(t,y) = h(t,y) \in \R 
		\end{gathered} 
		\end{equation}
	(cf., for instance, Seeley~\cite{Seeley1964Extension}), 
	for every 
		$ s \in [0,T] $, 
		$ t \in \R $, 
		$ x \in \R^d $ 
	let $Y^{s,(t,x)} = (Y^{s,(t,x)}_r)_{r\in [s,T]}\colon [s,T]\times\Omega\to\R^{d+1}$ be an $(\mathbb{F}_r)_{r\in [s,T]}$-adapted stochastic process with continuous sample paths satisfying that for all 
		$r\in [s,T]$ 
	we have $\P$-a.s.~that 
		\begin{equation}
		\label{smooth_solutions:autonomous_SDE}
		Y^{s,(t,x)}_r
		= \begin{pmatrix}
		t \\ x
		\end{pmatrix}
		+ \int_{s}^r \mathfrak{m}(Y^{s,(t,x)}_q)\,dq
		+ \int_{s}^r \mathfrak{s}(Y^{s,(t,x)}_q)\,dW_q   
		\end{equation}
	(cf., e.g., Karatzas \& Shreve~\cite[Theorem 5.2.9]{KaSh1991_BrownianMotionAndStochasticCalculus}), 
	for every 
		$ t \in [0,T] $, 
		$ x \in \R^d $ 
	let $Z^{t,x} = (Z^{t,x}_s)_{s\in [t,T]}\colon [t,T]\times\Omega\to\R^{d+1}$ satisfy for all 
		$ s \in [t,T]$ 
	that 
		$ Z^{t,x}_s = (s,X^{t,x}_{s}) $, 
	and let $v\colon [0,T]\times\R^{d+1}\to\R$ and $w\colon [0,T]\times\R^{d+1}\to\R$ satisfy for all 
		$ s \in [0,T] $, 
		$ t \in \R $, 
		$ x \in \R^{d} $ 
	that 
	\begin{equation}
	\label{classical_solutions:definition_of_v_and_w}
	v(s,t,x) 
	= 
	\Exp{\mathfrak{g}\big(Y^{s,(t,x)}_T\big)}
	\qquad\text{and}\qquad
	w(s,t,x)
	=
	\Exp{\mathfrak{h}\big(Y^{s,(t,x)}_T\big)}\!. 
	\end{equation}	
	Note that the assumption that $ ( \operatorname{supp}(\mu)\cup\operatorname{supp}(\sigma)) \subseteq [0,T]\times (-o,o)^d $ ensures that for all 
		$ t \in [0,T] $, 
		$ x \in \R^d \setminus (-o,o)^d $ 
	we have that 
		$ \P(\Forall s \in [t,T]\colon X^{t,x}_s = x) = 1 $ (cf., e.g., \cite[Item~(i) in Lemma 3.4]{StochasticFixedPointEquations}). 
	This implies that for all 
		$ t \in [0,T] $, 
		$ x\in\R^d\setminus (-o,o)^d $
	we have that
		\begin{equation}\label{smooth_solutions:outside_U}
		u(t,x) 
		= \Exp{ g( X^{t,x}_{T} ) 
			+ 
			\int_t^T h( s, X^{t,x}_{s} ) \,ds }
		= g(x) + \int_t^T h(s,x)\,ds . 
		\end{equation}
	The assumption that $ g $ and $ h $ are infinitely often differentiable  
	and the fact that $ (\operatorname{supp}(\mu)\cup\operatorname{supp}(\sigma)) \subseteq [0,T]\times (-o,o)^d $ 
	therefore assure that for all 
		$ t \in [0,T] $, 
		$ x \in \R^d \setminus [-o,o]^d $ 
	we have that 
		$ u|_{[0,T]\times(\R^d \setminus [-o,o]^d)} \in C^{1,2}([0,T]\times(\R^d \setminus [-o,o]^d),\R) $ 
	and 
		\begin{equation}
		\label{smooth_solutions:part_outside}
		\begin{split}
		0 & = 
		(\tfrac{\partial}{\partial t}u)(t,x) 
		+ 
		h(t,x) 
		\\
		& = (\tfrac{\partial}{\partial t}u)(t,x) 
		+ 
		\tfrac12 
		\operatorname{Trace}\!\left( 
		\sigma(t,x)[\sigma(t,x)]^{*}(\operatorname{Hess}_x u)(t,x)
		\right) 
		+ 
		\langle 
		\mu(t,x)
		, 
		(\nabla_x u)(t,x) 
		\rangle
		+
		h(t,x) 
		. 
		\end{split}
		\end{equation}
	Next note that, e.g., Da Prato \& Zabczyk~\cite[Theorem 7.4.5 and Theorem 7.5.1]{DaPZab2002_SecondOrderPDEsInHilbertSpaces} (cf.~also, e.g., G\={i}hman \& Skorokhod~\cite[Theorem 2.8.1 and Corollary 2.8.1]{GiSk1972_SDEs} and Andersson et al.~\cite[Theorem 1.1]{andersson2018regularity}) and \eqref{classical_solutions:definition_of_v_and_w} guarantee that
	\begin{enumerate}[(I)] 
		\item\label{smooth_solutions:proof_item1} 
			we have that $ v,w \in C^{1,2}([0,T]\times\R^{d+1},\R) $, 
		\item\label{smooth_solutions:proof_item2} 
			we have for all 
				$ s \in [0,T] $, 
				$ t \in \R $, 
				$ x \in \R^d $ 
			that 
				\begin{equation}
				\label{smooth_solutions:v_pde}
				\begin{split}
				& 
				- (\tfrac{\partial}{\partial s}v)(s,t,x) 
				\\
				& = 
				\tfrac12 \operatorname{Trace}\!\left( \mathfrak{s}(t,x)[\mathfrak{s}(t,x)]^{*}(\operatorname{Hess}_{(t,x)} v)(s,t,x) 	\right)
				+ \langle\langle \mathfrak{m}(t,x),	(\nabla_{(t,x)} v)(s,t,x) \rangle\rangle
				\\
				& 
				= 
				\tfrac12 \operatorname{Trace}\!\left( \sigma(t,x)[\sigma(t,x)]^{*}(\operatorname{Hess}_x v)(s,t,x) \right) 
				+ (\tfrac{\partial}{\partial t}v)(s,t,x) 
				+ \langle \mu(t,x), (\nabla_x v)(s,t,x) \rangle, 
				\end{split}
				\end{equation}
			and 
			\item we have for all 
				$ s \in [0,T] $, 
				$ t \in \R $, 
				$ x \in \R^d $ 
			that 
				\begin{align} \label{smooth_solutions:w_pde}
				\nonumber
				& - (\tfrac{\partial}{\partial s}w)(s,t,x) 
				\\
				& = 
				\tfrac12
				\operatorname{Trace}\!\left(
				\mathfrak{s}(t,x)[\mathfrak{s}(t,x)]^{*}
				(\operatorname{Hess}_{(t,x)} w)(s,t,x)
				\right)
				+
				\langle\langle 
				\mathfrak{m}(t,x),
				(\nabla_{(t,x)} w)(s,t,x)
				\rangle\rangle
				\\ \nonumber
				& = 
				\tfrac12 \operatorname{Trace}\!\left( 
				\sigma(t,x)[\sigma(t,x)]^{*}
				(\operatorname{Hess}_x w)(s,t,x)
				\right) 
				+ 
				(\tfrac{\partial}{\partial t}w)(s,t,x) 
				+ 
				\langle 
				\mu(t,x),
				(\nabla_x w)(s,t,x)
				\rangle.
				\end{align} 
		\end{enumerate} 
	Moreover, observe that \eqref{smooth_solutions:ass1}, \eqref{smooth_solutions:autonomized_coeffs}, and the fact that for all 
		$ t \in [0,T] $, 
		$ s \in [t,T] $, 
		$ x \in \R^d $ 
	we have that 
		$ Z^{t,x}_s = ( s, X^{t,x}_s ) $ 
	ensure that for all 
		$ t \in [0,T] $, 
		$ s \in [t,T] $, 
		$ x \in \R^d $ 
	we have $\P$-a.s.~that 
		\begin{equation}
		\label{smooth_solutions:SDE_for_Z}
		\begin{split}
		Z^{t,x}_s & = \begin{pmatrix} 
		s \\ X^{t,x}_{s}
		\end{pmatrix}
		= \begin{pmatrix}
		t \\ x 
		\end{pmatrix}
		+ 
		\int_t^s 
		\begin{pmatrix} 
		1 \\ \mu( r, X^{t,x}_{r} )
		\end{pmatrix}
		\,dr
		+ 
		\int_t^s 
		\begin{pmatrix}
		0 \\
		\sigma( r, X^{t,x}_{r} )
		\end{pmatrix}
		\,dW_r
		\\
		& = 
		\begin{pmatrix}
		t \\ x
		\end{pmatrix}
		+ \int_t^s \mathfrak{m}(Z^{t,x}_r)\,dr
		+ \int_t^s \mathfrak{s}(Z^{t,x}_r)\,dW_r . 
		\end{split}
		\end{equation}
	Combining this with the fact that for all 
		$ t \in [0,T] $, 
		$ x \in \R^d $  
	we have that $Z^{t,x}$ is an $(\mathbb{F}_s)_{s\in [t,T]}$-adapted stochastic process with continuous sample paths, \eqref{smooth_solutions:autonomous_SDE}, e.g., 
	Karatzas \& Shreve~\cite[Theorem 5.2.5]{KaSh1991_BrownianMotionAndStochasticCalculus}, 
	and \eqref{smooth_solutions:SDE_for_Z} 
	demonstrates that for all 
		$ t \in [0,T] $, 
		$ x \in \R^d $ 
	we have that  
		\begin{equation} 
		\label{smooth_solutions:indistinguishability_Y_Z}
		\P\left(\Forall s\in [t,T]\colon Z^{t,x}_s=Y^{t,(t,x)}_s\right) = 1.
		\end{equation}
	The fact that for all 
		$ t \in [0,T] $, 
		$ x \in (-\mf o,\mf o)^d $ 
	we have that 
		$ \P(\Forall s\in [t,T]\colon X^{t,x}_s \in [-\mf o,\mf o]^d) = 1 $, 
	\eqref{smooth_solutions:autonomized_coeffs}, 
	and \eqref{classical_solutions:definition_of_v_and_w} 
	therefore yield that for all 
		$ t \in [0,T] $, 
		$ x \in (-\mf o,\mf o)^d $ 
	we have that 
		\begin{equation}
		\label{smooth_solutions:v_eqn}
		v(t,t,x)
		=
		\EXPP{\mathfrak{g}( Y^{t,(t,x)}_T )}
		=
		\Exp{\mathfrak{g}( Z^{t,x}_T )}
		= 
		\Exp{\mathfrak{g}( T , X^{t,x}_{T} )}
		= 
		\Exp{g( X^{t,x}_{T} )}\!. 
		\end{equation}
	Furthermore, note that 
	\eqref{smooth_solutions:autonomized_coeffs},  
	\eqref{classical_solutions:definition_of_v_and_w}, 
	\eqref{smooth_solutions:indistinguishability_Y_Z}, 
	the fact that for all 
		$ t\in [0,T] $, 
		$ x \in (-\mf o,\mf o)^d $ 
	we have that 
		$ \P(\Forall s\in [t,T]\colon X^{t,x}_s \in [-\mf o,\mf o]^d) = 1 $, 
	and the fact that for all 
		$ t \in [0,T] $, 
		$ s \in [t,T] $, 
		$ x \in \R^d $, 
	 	$ B \in \Borel(\R) $ 
	we have that 
		$ \P(Y^{s,(t,x)}_T\in B) = \P(Y^{t,(t,x)}_{T-s+t}\in B) $ (cf., e.g., Klenke~\cite[Theorem 26.8]{Klenke2006_English}) 
	demonstrate that for all 
		$ t \in [0,T] $, 
		$ s \in [t,T] $, 
		$ x \in (-\mf o,\mf o)^d $
	we have that 
		\begin{equation}
		\label{smooth_solutions:w_eqn}
		\begin{split}
		w(T-s+t,t,x)
		& =
		\EXPP{\mathfrak{h}(Y^{T-s+t,(t,x)}_T)}
		=
		\Exp{\mathfrak{h}(Y^{t,(t,x)}_{s})}
		\\
		& = 
		\Exp{\mathfrak{h}( Z^{t,x}_s )}
		= 
		\Exp{\mathfrak{h}( s, X^{t,x}_s ) }
		=
		\Exp{h ( s, X^{t,x}_{s} ) }\!. 
		\end{split}
		\end{equation}
	This, \eqref{smooth_solutions:ass2}, and \eqref{smooth_solutions:v_eqn} show for all 
		$ t \in [0,T] $, 
		$ x \in (-\mf o,\mf o)^d $ 
	that 
		\begin{equation}
		\label{smooth_solutions:u_through_v_w}
		u(t,x) = v(t,t,x) + \int_t^T w(T-s+t,t,x) \,ds . 
		\end{equation}
	Combining this with the fact that 
		$v,w\in C^{1,2}([0,T]\times\R^{d+1},\R)$ 
	and the chain rule ensures that for all 
		$ t \in [0,T] $, 
		$ x \in (-\mf o,\mf o)^d $ 
	we have that 
		$ u|_{[0,T]\times (-\mf o,\mf o)^d}\in C^{1,2}([0,T]\times (-\mf o,\mf o)^d,\R) $ 
	and 
		\begin{equation}
		\label{smooth_solutions:u_diff1}
		\begin{split}
		& (\tfrac{\partial}{\partial t}u)(t,x) 
		= 
		(\tfrac{\partial}{\partial s}v)(t,t,x) 
		+ 
		(\tfrac{\partial}{\partial t}v) 
		(t,t,x) 
		+ 
		(\tfrac{\partial}{\partial t})\left[ 
		\int_t^T w(T-s+t,t,x)\,ds
		\right]\!. 
		\end{split}
		\end{equation}
	Furthermore, note that \eqref{smooth_solutions:v_pde} and \eqref{smooth_solutions:w_pde} yield that for all
		$ t \in [0,T] $, 
		$ s \in [t,T] $, 
		$ x \in \R^d $ 
	we have that
		\begin{equation}
		\label{smooth_solutions:u_diff2}
		\begin{split}
		& (\tfrac{\partial}{\partial s}v)(t,t,x) 
		+ 
		(\tfrac{\partial}{\partial t}v) 
		(t,t,x) 
		\\
		& = 
		- \tfrac12
		\operatorname{Trace}\!\left( 
		\sigma(t,x)[\sigma(t,x)]^{*}
		(\operatorname{Hess}_x v)(t,t,x)
		\right) 
		- \langle \mu(t,x),(\nabla_x v)(t,t,x) \rangle
		\end{split}
		\end{equation}
	and 
		\begin{equation}
		\label{smooth_solutions:u_diff3}
		\begin{split}
		& (\tfrac{\partial}{\partial s}w)(s,t,x) 
		+ 
		(\tfrac{\partial}{\partial t}w) 
		(s,t,x) 
		\\
		& = 
		- \tfrac12
		\operatorname{Trace}\!\left( 
		\sigma(t,x)[\sigma(t,x)]^{*}
		(\operatorname{Hess}_x w)(s,t,x)
		\right) 
		- \langle \mu(t,x),(\nabla_x w)(s,t,x) \rangle.
		\end{split}
		\end{equation}
	This ensures for all 
		$ t \in [0,T] $, 
		$ x \in \R^d $ 
	that 
		\begin{equation}
		\label{smooth_solutions:u_diff4}
		\begin{split}
		& (\tfrac{\partial}{\partial t})\!
		\left[ 
		\int_t^T w(T-s+t,t,x)\,ds
		\right] 
		\\
		& 
		= 
		-w(T,t,x) + 
		\int_t^t \left( 
		(\tfrac{\partial}{\partial s}w)(T-s+t,t,x) 
		+ 
		(\tfrac{\partial}{\partial t}w)(T-s+t,t,x)
		\right)\,ds 
		\\
		& 
		= 
		- w(T,t,x) 
		- \int_t^T 
		\tfrac12
		\operatorname{Trace}\!\left( 
		\sigma(t,x)[\sigma(t,x)]^{*}
		(\operatorname{Hess}_x w)(T-s+t,t,x)
		\right) 
		\,ds 
		\\
		& \quad 
		- \int_t^T \langle \mu(t,x),(\nabla_x w)(T-s+t,t,x) \rangle
		\,ds .
		\end{split}
		\end{equation}
	Next observe that the fact that $w\in C^{1,2}([0,T]\times\R^{d+1},\R)$ proves that for all 
		$t\in [0,T]$, 
		$x\in\R^d$ 
	we have that
		\begin{equation}
		\label{smooth_solutions:u_diff5}
		\begin{split}
		& \int_t^T \tfrac12
		\operatorname{Trace}\!\left( 
		\sigma(t,x)[\sigma(t,x)]^{*}
		(\operatorname{Hess}_x w)(T-s+t,t,x)
		\right) 
		\,ds 
		\\
		& = 
		\tfrac12 
		\operatorname{Trace}\!
		\left(
		\sigma(t,x)[\sigma(t,x)]^{*}
		\!\left[\operatorname{Hess}_x\!\left( 
		\int_t^T w(T-s+t,t,x)\,ds 
		\right)\right]
		\right) 
		\end{split}
		\end{equation}
	and 
	\begin{equation}
	\label{smooth_solutions:u_diff6}
	\begin{split}
	& \int_t^T \langle \mu(t,x),(\nabla_x w)(T-s+t,t,x) \rangle
	\,ds 
	= 
	\left\langle 
	\mu(t,x), 
	\nabla_x \!\left( 
	\int_t^T w(T-s+t,t,x)\,ds
	\right)
	\right\rangle
	. 
	\end{split}
	\end{equation}
	Moreover, note that \eqref{smooth_solutions:ass1} and \eqref{smooth_solutions:w_eqn} ensure for all 
		$ t\in [0,T] $, 
		$ x \in (-\mf o,\mf o)^d $ 
	that
		\begin{equation}
		w(T,t,x) = \Exp{h( t,X^{t,x}_{t}(x) )} = h(t,x).
		\end{equation}
	Combining this with \eqref{smooth_solutions:u_through_v_w}--\eqref{smooth_solutions:u_diff6} shows for all 	
		$ t \in [0,T] $, 
		$ x \in (-\mf o,\mf o)^d $ 
	that 
		\begin{equation}
		\begin{split}
		(\tfrac{\partial}{\partial t}u)(t,x) 
		=
		- 
		\tfrac12 \operatorname{Trace}\! \left(
		\sigma(t,x)[\sigma(t,x)]^{*}
		(\operatorname{Hess}_x u)(t,x)
		\right)
		- 
		\langle \mu(t,x),(\nabla_x u)(t,x)\rangle
		- h(t,x) . 
		\end{split}
		\end{equation}
	This and \eqref{smooth_solutions:part_outside} demonstrate for all 
	$t\in [0,T]$, 
	$x\in\R^d$ 
	that 
	$u\in C^{1,2}([0,T]\times\R^d,\R)$ 
	and 
	\begin{equation} \label{smooth_solutions:pde_verified}
	(\tfrac{\partial}{\partial t}u)(t,x) 
	=
	- 
	\tfrac12 \operatorname{Trace}\!\left(
	\sigma(t,x)[\sigma(t,x)]^{*}
	(\operatorname{Hess}_x u)(t,x)
	\right)
	- 
	\langle \mu(t,x),(\nabla_x u)(t,x)\rangle
	- h(t,x) .   
	\end{equation}
	This establishes Item~\eqref{smooth_solutions:item1}. 
	Furthermore, observe that \eqref{smooth_solutions:ass1} and \eqref{smooth_solutions:ass2} demonstrate that for all 		
		$x\in\R^d$ 
	we have that $u(T,x) = g(x)$. 
	Combining this with \eqref{smooth_solutions:pde_verified} establishes Item~\eqref{smooth_solutions:item2}. 
	This completes the proof of \cref{lem:smooth_solutions}. 
\end{proof}

\subsection{Basic properties of viscosity solutions of suitable PDEs} 
\label{subsec:def_and_elem_prop_viscosity_solutions}

\begin{definition}[Symmetric matrices]
	\label{symmetric_matrices}
	Let $ d \in \N $. Then we denote by $ \Sym_d $ the set given by $ \Sym_d = \{ A \in \R^{d\times d} \colon A^{*} = A \} $. 
\end{definition}	
	
\begin{definition}[Degenerate elliptic functions]
	\label{degenerate_elliptic}
	Let $ d \in \N $, 
		$ T \in (0,\infty) $, 
	let $ \mc O \subseteq \R^d $ be a non-empty open set, 
	and 
	let $ \langle\cdot,\cdot\rangle\colon \R^d\times\R^d\to\R$ be the standard Euclidean scalar product on $\R^d$. 
	Then we say that $ G $ is degenerate elliptic on $ (0,T) \times \mc O \times \R \times \R^d \times \Sym_d $ (we say that $ G $ is degenerate elliptic) if and only if 
	\begin{enumerate}[(i)]
		\item we have that $ G\colon (0,T) \times \mc O \times \R \times \R^d \times \Sym_d \to \R $ is a function from $ (0,T) \times \mc O \times \R \times \R^d \times \Sym_d $ to $ \R $ 
		and 
		\item we have for all 
			$ t \in (0,T) $, 
			$ x \in \mc O $, 
			$ r \in \R $, 
			$ p \in \R^d $, 
			$ A, B \in \Sym_{d} $ 
		with 
			$ \Forall y \in \R^d \colon \langle Ay,y\rangle \leq \langle By,y\rangle $ 
		that 
			$ G(t,x,r,p,A) \leq G(t,x,r,p,B) $
	\end{enumerate}
	(cf.~\cref{symmetric_matrices}). 
\end{definition}

\begin{definition}[Viscosity subsolutions] \label{def:viscosity_subsolution}
	Let $ d \in \N $, 
		$ T \in (0,\infty) $, 
	let $ \mc O \subseteq \R^d $ be a non-empty open set, 
	and let $ G \colon (0,T)\times\mc O \times \R \times \R^d \times \Sym_d \to \R $ be degenerate elliptic (cf.~\cref{symmetric_matrices,degenerate_elliptic}). 
	Then we say that $ u $ is a viscosity solution of $ (\frac{\partial}{\partial t}u)(t,x) + G( t, x, u(t,x), (\nabla_x u)(t,x), (\operatorname{Hess}_x u)(t,x) ) \geq 0 $ for $ ( t, x ) \in (0,T)\times\mc O $ (we say that $ u $ is a viscosity subsolution of $ (\frac{\partial}{\partial t}u)(t,x)  + G(t,x,u(t,x),(\nabla_x u)(t,x), \allowbreak (\operatorname{Hess}_x u)(t,x))=0 $ for $ (t,x) \in (0,T)\times\mc O $) if and only if there exists a set $ A $ such that 
	\begin{enumerate}[(i)]
		\item we have that $ (0,T) \times \mc O \subseteq A $, 
		\item we have that $ u \colon A \to \R $ is an upper semi-continuous function from $ A $ to $ \R $, 
		and 
		\item we have for all 
			$ t \in (0,T) $, 
			$ x \in \mc O $,
			$ \phi \in C^{1,2}((0,T)\times \mc O,\R) $ 
		with
			$\phi(t,x)=u(t,x)$ and 
			$\phi \geq u$ 
		that 
			\begin{equation} \label{viscosity_subsolution:defining_property}
			(\tfrac{\partial}{\partial t}\phi)(t,x) 
			+ G(t,x,\phi(t,x),(\nabla_x\phi)(t,x),(\operatorname{Hess}_x\phi)(t,x)) \geq 0. 
			\end{equation}
	\end{enumerate}
\end{definition}

\begin{definition}[Viscosity supersolutions] \label{def:viscosity_supersolution}
	Let $ d \in \N $, 
		$ T \in (0,\infty) $, 
	let $ \mc O \subseteq \R^d $ be a non-empty open set, 
	and 
	let $ G \colon (0,T) \times \mc O \times \R \times \R^d \times \Sym_d \to \R $ be degenerate elliptic (cf.~\cref{symmetric_matrices,degenerate_elliptic}). 
	Then we say that $ u $ is a viscosity solution of $ (\frac{\partial}{\partial t}u)(t,x) + G(t,x,u(t,x),(\nabla_x u)(t,x),(\operatorname{Hess}_x u)(t,x)) \leq 0 $ for $ (t,x) \in (0,T)\times\mc O $ (we say that $ u $ is a viscosity supersolution of $ (\frac{\partial}{\partial t}u)(t,x) + G(t,x,u(t,x),(\nabla_x u)(t,x), \allowbreak (\operatorname{Hess}_x u)(t,x))=0 $ for $ (t,x) \in (0,T)\times\mc O $) if and only if there exists a set $ A  $ such that 
	\begin{enumerate}[(i)]
		\item we have that $ (0,T)\times\mc O \subseteq A $, 
		\item we have that $ u \colon A \to \R $ is a lower semi-continuous function from $ A $ to $ \R $, 
		and 
		\item we have for all 
			$ t \in (0,T) $, 
			$ x \in \mc O $, 
			$ \phi \in C^{1,2}((0,T) \times \mc O,\R)$ 
		with 
			$ \phi(t,x) = u(t,x) $ 
		and 
			$ \phi \leq u $
		that 
			\begin{equation} \label{viscosity_supersolution:defining_property}
			(\tfrac{\partial}{\partial t}\phi)(t,x) 
			+ 
			G(t,x,\phi(t,x),(\nabla_x\phi)(t,x),(\operatorname{Hess}_x\phi)(t,x)) 
			\leq 0. 
			\end{equation}
	\end{enumerate}
\end{definition}

\begin{definition}[Viscosity solutions] \label{def:viscosity_solution}
	Let $ d \in \N $, 
		$ T \in (0,\infty) $, 
	let $ \mc O \subseteq \R^d $ be a non-empty open set, 
	and let $ G\colon (0,T) \times \mc O \times \R \times \R^d \times \Sym_d \to \R $ be degenerate elliptic (cf.~\cref{symmetric_matrices,degenerate_elliptic}). 
	Then we say that $ u $ is a viscosity solution of $ (\frac{\partial}{\partial t}u)(t,x)  + G(t,x,u(t,x),(\nabla_x u)(t,x),(\operatorname{Hess}_x u)(t,x))=0 $ for $ (t,x) \in (0,T)\times\mc O $ if and only if 
	\begin{enumerate}[(i)] 
		\item we have that $ u $ is a viscosity subsolution of 
		$ (\frac{\partial}{\partial t}u)(t,x) + G(t,x,u(t,x),(\nabla_x u)(t,x), (\operatorname{Hess}_x u)(t, \allowbreak x)) = 0 $ for $ (t,x) \in (0,T)\times\mc O $ 
		and 
		\item we have that $ u $ is a viscosity supersolution of $ (\frac{\partial}{\partial t}u)(t,x)  + G(t,x,u(t,x),(\nabla_x u)(t,x), (\operatorname{Hess}_x u)(t, \allowbreak x))=0 $ for $ (t,x) \in (0,T)\times\mc O $ 
	\end{enumerate}
	(cf.~\cref{def:viscosity_subsolution,def:viscosity_supersolution}). 
\end{definition}

\begin{lemma} \label{lem:classical_solutions_are_viscosity_solutions} 
	Let $ d \in \N $, 
		$ T \in (0,\infty) $, 
	let $ \mc O \subseteq \R^d $ be a non-empty open set, 
	and let $ G \colon (0,T) \times \mc O \times \R \times \R^d \times \Sym_d \to \R$ be degenerate elliptic (cf.~\cref{symmetric_matrices,degenerate_elliptic}). 
	Then  
	\begin{enumerate} [(i)]
		\item \label{classical_solutions_are_viscosity_solutions:item1}
		we have for every $ u \in C^{1,2}( (0,T) \times \mc O,\R)$ with  
			$ \Forall t \in (0,T), x \in \mc O \colon 
			(\tfrac{\partial}{\partial t}u)(t,x) + G(t,x,u(t,x), \allowbreak (\nabla_x u)(t,x),(\operatorname{Hess}_x u)(t,x)) 
			\geq 0 $
		that $ u $ is a viscosity solution of 
			\begin{equation} 
			(\tfrac{\partial}{\partial t}u)(t,x)
			+ 
			G(t,x,u(t,x),(\nabla_x u)(t,x),
			(\operatorname{Hess}_x u)(t,x)) 
			\geq 0
			\end{equation} 
		for $ (t,x) \in (0,T)\times\mc O $, 
		\item \label{classical_solutions_are_viscosity_solutions:item2} 
		we have for every $ u \in C^{1,2}( (0,T) \times \mc O,\R)$ with 
			$ \Forall t \in (0,T), x \in \mc O \colon 
			(\tfrac{\partial}{\partial t}u)(t,x)
			+ G(t,x,u(t,x), \allowbreak (\nabla_x u)(t,x), (\operatorname{Hess}_x u)(t,x)) \leq 0 $ 
		that $ u $ is a viscosity solution of 
			\begin{equation} 
			(\tfrac{\partial}{\partial t}u)(t,x)
			+ 
			G(t,x,u(t,x),(\nabla_x u)(t,x),
			(\operatorname{Hess}_x u)(t,x)) 
			\leq 0
			\end{equation} 
		for $ (t,x) \in (0,T)\times\mc O $, and 
		\item \label{classical_solutions_are_viscosity_solutions:item3}
		we have for every $ u \in C^{1,2}( (0,T) \times \mc O,\R)$ with 
			$ \Forall t \in (0,T), x \in \mc O \colon 
			(\tfrac{\partial}{\partial t}u)(t,x)
			+ G(t,x,u(t,x), \allowbreak (\nabla_x u)(t,x), (\operatorname{Hess}_x u)(t,x)) = 0 $ 
		that $ u $ is a viscosity solution of 
			\begin{equation} 
			(\tfrac{\partial}{\partial t}u)(t,x)
			+ 
			G(t,x,u(t,x),(\nabla_x u)(t,x),
			(\operatorname{Hess}_x u)(t,x)) 
			= 0
			\end{equation} 
		for $ (t,x) \in (0,T)\times\mc O $ 
		\end{enumerate}
	(cf.~\cref{def:viscosity_subsolution,def:viscosity_supersolution,def:viscosity_solution}). 
\end{lemma}

\begin{proof}[Proof of \cref{lem:classical_solutions_are_viscosity_solutions}] 
	First, note that \eqref{viscosity_subsolution:defining_property} establishes Item~\eqref{classical_solutions_are_viscosity_solutions:item1}.
	Next observe that \eqref{viscosity_supersolution:defining_property} proves Item~\eqref{classical_solutions_are_viscosity_solutions:item2}. 
	Moreover, note that Item~\eqref{classical_solutions_are_viscosity_solutions:item1} and Item~\eqref{classical_solutions_are_viscosity_solutions:item2} 
	establish Item~\eqref{classical_solutions_are_viscosity_solutions:item3}. 
	This completes the proof of \cref{lem:classical_solutions_are_viscosity_solutions}. 	
\end{proof} 

\begin{lemma} \label{lem:local_maximum_enough}
	Let $ d \in \N $, 
		$ T \in (0,\infty) $, 
		$ \mf t \in (0,T) $, 
	let $ \mc O \subseteq \R^d $ be an open set, 
	let $ \mf x \in \mc O $, 
		$ \phi \in C^{1,2}( (0,T)\times\mc O, \R ) $, 
	let $ G \colon (0,T) \times \mc O \times \R \times \R^d \times \Sym_d \to \R $ be degenerate elliptic, 
	let $ u \colon (0,T) \times \mc O \to\R $ be a viscosity solution of 
		$ (\tfrac{\partial}{\partial t}u)(t,x) 
		+ 
		G(t,x,u(t,x),(\nabla_x u)(t,x),(\operatorname{Hess}_x u)(t,x)) \geq 0$ 
	for $ (t,x) \in (0,T)\times\mc O $, 
	and assume that $ u-\phi $ has a local maximum at $ (\mf t, \mf x) \in (0,T)\times\mc O $ 	(cf.~\cref{symmetric_matrices,degenerate_elliptic,def:viscosity_subsolution}). 
	Then 
		\begin{equation} 
		(\tfrac{\partial}{\partial t}\phi)(\mf t, \mf x) 
		+ 
		G( \mf t, \mf x,u(\mf t, \mf x),(\nabla_x\phi)(\mf t,\mf x),(\operatorname{Hess}_x\phi)(\mf t, \mf x)) 
		\geq 0. 
		\end{equation}
	\end{lemma}

\begin{proof}[Proof of \cref{lem:local_maximum_enough}] 
	First, observe that the fact that $ u $ is upper semi-continuous implies that there exist $ \psi \in C^{1,2}((0,T)\times\mc O,\R) $ and a non-empty open set $ \mc U \subseteq (0,T)\times\mc O $  which satisfy that 
	\begin{enumerate}[(i)]
		\item we have that $ (\mf t, \mf x) \in \mc U $, 
		\item we have for all 
			$ t \in (0,T) $, 
			$ x \in \mc O $ 
		that 
			$ u(\mf t, \mf x) - \psi( \mf t, \mf x ) \geq u( t, x ) - \psi( t, x ) $, 
		and 
		\item we have for all $ ( t, x ) \in \mc U $ that 
			$ \psi( t, x ) = \phi ( t, x ) $. 		
	\end{enumerate}
	Hence, we obtain that 
		\begin{equation} 
		\begin{split} 
		& 
		( \tfrac{\partial}{\partial t}\phi )( \mf t, \mf x ) + G( \mf t, \mf x, u( \mf t, \mf x ), ( \nabla_x \phi)( \mf t, \mf x ), (\operatorname{Hess}_x\phi)( \mf t, \mf x)) 
		\\
		& = 
		( \tfrac{\partial}{\partial t}\psi )( \mf t, \mf x ) + G( \mf t, \mf x, u( \mf t, \mf x ), ( \nabla_x \psi)( \mf t, \mf x ), (\operatorname{Hess}_x\psi)( \mf t, \mf x)) 
		\geq 0 . 
		\end{split} 
		\end{equation} 
	This completes the proof of  \cref{lem:local_maximum_enough}. 
\end{proof}

\begin{lemma} \label{lem:C2_test_functions_suffice} 
	Let $ d \in \N $, 
		$ T \in (0,\infty) $, 
	let $ \mc O \subseteq \R^d $ be a non-empty open set, 
	let $ G \colon (0,T) \times \mc O \times \R \times \R^d \times \Sym_d \to \R $ be degenerate elliptic and upper semi-continuous, 
	let $ u \colon (0,T) \times \mc O \to \R $ be upper semi-continuous, 
	and assume for all 
		$ t \in (0,T) $, 
		$ x \in \mc O $, 
		$ \phi \in \{ \psi \in C^2((0,T)\times\mc O,\R) \colon (u - \psi~\text{has a local maximum at}~(t,x)\in (0,T)\times\mc O) \} $ 
	that 
		\begin{equation} \label{C2_test_functions_suffice:ass}
		(\tfrac{\partial}{\partial t}\phi)(t,x) + G(t,x,u(t,x),(\nabla_x\phi)(t,x),(\operatorname{Hess}_x\phi)(t,x)) \geq 0  
		\end{equation} 
	(cf.~\cref{symmetric_matrices,degenerate_elliptic}). 
	Then $ u $ is a viscosity solution of 
		\begin{equation} \label{C2_test_functions_suffice:claim}
		(\tfrac{\partial}{\partial t}u)(t,x) + G(t,x,u(t,x),(\nabla_x u)(t,x),(\operatorname{Hess}_x u)(t,x)) \geq 0  
		\end{equation} 
	for $ (t,x) \in (0,T)\times\mc O $ (cf.~\cref{def:viscosity_subsolution}). 
\end{lemma} 

\begin{proof}[Proof of \cref{lem:C2_test_functions_suffice}]
	Throughout this proof let $ \norm{\cdot}\colon\R^d\to [0,\infty) $ be the standard Euclidean norm on $ \R^d $, 
	let $ \HSnorm{\cdot} \colon \R^{d\times d} \to [0,\infty) $ be the Frobenius norm on $ \R^{d\times d} $, 
	let 
		$ t_0 \in (0,T) $, 
		$ x_0 \in \mc O $, 
		$ \phi \in C^{1,2}((0,T)\times\mc O,\R) $ 
	satisfy for all 
		$ s \in (0,T) $, 
		$ y \in \mc O $ 
	that 
		$ \phi(s,y) \geq u(s,y) $ 
	and 
		$ \phi(t_0,x_0) = u(t_0,x_0) $, 
	let $ \psi_0 \in C^{1,2}((0,T)\times\mc O,\R) $ satisfy for all 
		$ s \in (0,T) $, 
		$ y \in \mc O $ 
	that 
		$ \psi_0(s,y) = \phi(s,y) + |s-t_0|^4 + \Norm{y-x_0}^4 $,  
	let 
		$ \eta \in (0,\infty) $ 
	satisfy that 
		$ \{ (s,y) \in \R\times\R^{d} \colon \max\{ |s-t_0|, \norm{y-x_0} \} \leq \eta \} \subseteq (0,T)\times\mc O $, 
	and let $ I_r \in [0,\infty) $, $ r \in (0,\eta] $, 
	satisfy for all 
		$ r \in (0,\eta] $ 
	that 
		$ I_r = \frac12 \inf \{ \psi_0(s,y) - u(s,y) \colon (s,y) \in (0,T)\times\mc O, r^2 \leq |s-t_0|^2+\Norm{y-x_0}^2 \leq \eta^2 \}  $.
	Observe that the fact that for all 
		$ (s,y) \in [ (0,T) \times\mc O ] \setminus \{(t_0,x_0)\} $ 
	we have that $ \psi_0(s,y) > \phi(s,y) $ ensures that for all 
		$ (s,y) \in [ (0,T) \times\mc O ] \setminus \{(t_0,x_0)\} $ 
	we have that 
		$ \psi_0(s,y) > u(s,y) $ 
	and 
		$ \psi_0(t_0,x_0) = u(t_0,x_0) $. 
	The assumption that $ u $ is upper semi-continuous hence guarantees that for all 
		$ r \in (0,\eta] $ 
	we have that 
		$ I_r \in (0,\infty) $.  
	Moreover, note that there exist 
		$ \psi_n \in C^2((0,T)\times\mc O,\R) $, $ n \in \N $, 
	which satisfy for all non-empty compact $ \mc K \subseteq (0,T)\times\mc O $ that 
		\begin{multline} \label{C2_test_functions_suffice:approximation_by_C2_functions}
		\limsup_{n\to\infty} 
		\Bigg[ 
		\sup_{(s,y)\in\mc K} \Big(  
			|(\tfrac{\partial}{\partial t}\psi_n)(s,y) - (\tfrac{\partial}{\partial t}\psi_0)(s,y)| 
			+ 
			| \psi_n(s,y) - \psi_0(s,y) | 
			\\
			+ 
			\norm{(\nabla_x\psi_n)(s,y) - (\nabla_x\psi_0)(s,y)} 
			+ 
			\HSnorm{(\operatorname{Hess}_x\psi_n)(s,y) - (\operatorname{Hess}_x\psi_0)(s,y)}
			\Big) 
		\Bigg] 
		= 0 . 
		\end{multline} 
	This implies that there exists 
		$ \mf n = ( \mf n_{ \varepsilon } )_{ \varepsilon \in (0,\infty) } \colon ( 0, \infty ) \to \N $ 
	which satisfies that for all 
		$ \varepsilon \in (0,\infty) $,  
		$ n \in \N \cap [\mf n_{\varepsilon},\infty) $ 
	we have that 
		\begin{equation} \label{C2_test_functions_suffice:defining_property_of_n_function}
		\sup\!\left\{ |\psi_n(s,y)-\psi_0(s,y)| \colon  (s,y)\in (0,T)\times\mc O, |s-t_0|^2+\norm{y-x_0}^2\leq \eta^2\right\} < \varepsilon.   
		\end{equation} 
	The fact that $ I_{\eta} \in (0,\infty) $ and \eqref{C2_test_functions_suffice:approximation_by_C2_functions} hence ensure that there exist
		$ (t_n, x_n) \in (0,T) \times \mc O $, $n \in \N $,  
	which satisfy for all 
		$ n \in \N \cap [\mf n_{I_{\eta}},\infty) $, 
		$ s \in (0,T) $, 
		$ y \in \mc O $ 
	with 
		$ | s - t_0 |^2 + \Norm{ y - x_0 }^2 \leq \eta^2 $  
	that 
		$ | t_n - t_0 |^2 + \Norm{ x_n - y_0 }^2 < \eta^2 $
	and 
		\begin{equation} \label{C2_test_functions_suffice:approximate_local_maximum}
		u(t_n,x_n) - \psi_n(t_n,x_n) 
		\geq 
		u(s,y) - \psi_n(s,y) .    
		\end{equation} 
 	Combining this with \eqref{C2_test_functions_suffice:ass} proves that for all 
 		$ n \in [\mf n_{I_{\eta}},\infty) $ 
 	we have that  
 		\begin{equation} \label{C2_test_functions_suffice:subsolution_inequality}
 		(\tfrac{\partial}{\partial t}\psi_n)(t_n,x_n) 
 		+ 
 		G(t_n,x_n,u(t_n,x_n),(\nabla_x \psi_n)(t_n,x_n),(\operatorname{Hess}_x \psi_n)(t_n,x_n)) \geq 0.  
 		\end{equation} 
 	Moreover, note that \eqref{C2_test_functions_suffice:defining_property_of_n_function} and  \eqref{C2_test_functions_suffice:approximate_local_maximum} imply that for all 
 		$ r \in (0,\eta) $, 
 		$ n \in [\mf n_{I_r},\infty) $ 
 	we have that 
 		$ |t_n-t_0|^2 + \norm{x_n-x_0}^2 \leq r^2 $. 
 	Therefore, we obtain that 
 		$ \limsup_{ n \to \infty} ( | t_n - t_0 |^2 + \Norm{x_n - x_0}^2 ) = 0 $. 
 	Combining this with \eqref{C2_test_functions_suffice:approximation_by_C2_functions}, \eqref{C2_test_functions_suffice:approximate_local_maximum},  	and the assumption that $ u $ is upper semi-continuous shows that 
 		\begin{equation} 
 		\begin{split} 
 		0 
 		& \geq \limsup_{n\to\infty} [u(t_n,x_n) - u(t_0,x_0)] 
 		\geq \liminf_{n\to\infty} [u(t_n,x_n) - u(t_0,x_0)] 
 		\\
 		& \geq \liminf_{n\to\infty} [\psi_n(t_n,x_n)-\psi_n(t_0,x_0)] 
 		= 0.  
 		\end{split} 
 		\end{equation}  
 	The fact that $ \limsup_{ n \to \infty} ( | t_n - t_0 |^2 + \Norm{x_n - x_0}^2 ) = 0 $, the fact that $ \psi_0 \in C^{1,2}((0,T)\times\mc O,\R) $, \eqref{C2_test_functions_suffice:approximation_by_C2_functions}, 
 	the assumption that $ G $ is upper semi-continuous, and \eqref{C2_test_functions_suffice:subsolution_inequality} hence demonstrate that 
 		\begin{equation}
 		\begin{split} 
 		& (\tfrac{\partial}{\partial t} \phi)(t_0,x_0) + G(t_0,x_0,u(t_0,x_0),(\nabla_x \phi)(t_0,x_0),(\operatorname{Hess}_x \phi)(t_0,x_0) 
 		\\
 		& = 
 		(\tfrac{\partial}{\partial t} \psi_0)(t_0,x_0) + G(t_0,x_0,u(t_0,x_0),(\nabla_x \psi_0)(t_0,x_0),(\operatorname{Hess}_x \psi_0)(t_0,x_0) \geq 0. 
		\end{split} 
 		\end{equation} 
 	This establishes \eqref{C2_test_functions_suffice:claim}. 
 	This completes the proof of \cref{lem:C2_test_functions_suffice}. 
\end{proof} 

\begin{definition}[Parabolic superjets] \label{def:parabolic_superjets}
	Let $ d \in \N $, 
		$ T \in (0,\infty) $, 
	let $ \mc O \subseteq \R^d $ be a non-empty open set, 
	let $ t \in ( 0, T ) $, 
		$ x \in \mc O $, 
	let $ \langle\cdot,\cdot\rangle \colon \R^d\times\R^d \to \R $ be the standard Euclidean scalar product on $ \R^d $, 
	let $ \norm{\cdot} \colon \R^d \to [0,\infty) $ be the standard Euclidean norm on $ \R^d $,  
	and let $ u \colon ( 0, T ) \times \mc O \to \R $ be a function. 
	Then we denote by $ (\mc P^{+}_{d,T,\mc O} u)(t,x) $ (we denote by $ (\mc P^{+}u)(t,x) $) the set given by 
		\begin{multline} \label{parabolic_superjets}
		(\mc P^{+}_{d,T,\mc O} u)(t,x) 
		= 
		(\mc P^{+} u)(t,x) 
		= 
		\bigg\{ 
		(b,p,A) \in \R \times \R^d \times \Sym_d \colon \\
		\limsup_{[(0,T)\times\mc O] \setminus \{(t,x)\} \ni (s,y) \to (t,x)}
		\left[
		\tfrac{u(s,y) - u(t,x) - b ( s - t ) 
			- \langle p, y - x \rangle
			- \frac12 \langle 
			A (y-x), y-x 
			\rangle
		}{|t-s| + \norm{ x - y }^2}  
		\right] 
		\leq 0
		\bigg\}
		\end{multline}
	(cf.~\cref{symmetric_matrices}). 
\end{definition}
	
\begin{definition}[Parabolic subjets] \label{def:parabolic_subjets}
	Let $ d \in \N $, 
		$ T \in (0,\infty) $, 
	let $ \mc O \subseteq \R^d $ be a non-empty open set, 
	let $ t \in ( 0, T ) $, 
		$ x \in \mc O $, 
	let $ \langle\cdot,\cdot\rangle \colon \R^d\times\R^d \to \R $ be the standard Euclidean scalar product on $ \R^d $, 
	let $ \norm{\cdot} \colon \R^d \to [0,\infty) $ be the standard Euclidean norm on $ \R^d $, 
	and let $ u \colon ( 0, T ) \times \mc O \to \R $ be a function. 
	Then we denote by $ ( \mc P^{-}_{d,T,\mc O}u)(t,x) $ (we denote by $ ( \mc P^{-} u )( t, x ) $) the set given by
		\begin{multline}
		( \mc P^{-}_{d,T,\mc O}u)(t,x)
		= 
		( \mc P^{-} u )( t, x ) 
		= 
		\bigg\{ 
		(b,p,A) \in \R \times \R^d \times \Sym_d \colon \\
		\liminf_{[(0,T)\times\mc O] \setminus \{(t,x)\} \ni (s,y) \to (t,x)}
		\left[
		\tfrac{u(s,y) - u(t,x) - b ( s - t ) 
			- \langle p, y - x \rangle
			- \frac12 \langle 
			A (y-x), y-x 
			\rangle
		}{|t-s| + \norm{ x - y }^2}  
		\right]  
		\geq 0
		\bigg\}
		\end{multline}
	(cf.~\cref{symmetric_matrices}). 
\end{definition} 

\begin{definition}[Generalized parabolic superjets] \label{def:relaxed_parabolic_superjets}
	Let $ d \in \N $, 
		$ T \in (0,\infty) $, 
	let $ \mc O \subseteq \R^d $ be a non-empty open set, 
	let $ t \in ( 0, T ) $, 
		$ x \in \mc O $, 
	let $ \langle\cdot,\cdot\rangle \colon \R^d\times\R^d \to \R $ be the standard Euclidean scalar product on $ \R^d $, 
	let $ \norm{\cdot} \colon \R^d \to [0,\infty) $ be the standard Euclidean norm on $ \R^d $, 
	and let $ u \colon ( 0, T ) \times \mc O \to \R $ be a function. 
	Then we denote by $ ( \JetClosure^{+}_{d,T,\mc O} u )( t, x ) $ (we denote by $ ( \JetClosure^{+} u )( t, x ) $) the set given by 
		\begin{multline}
		( \JetClosure^{+}_{d,T,\mc O} u )( t, x )
		= 
		( \JetClosure^{+} u )( t, x ) 
		= 
		\Bigg\{ 
		(b,p,A) \in \R \times \R^d \times \Sym_d \colon 
		\\
		\Bigg( 
		\begin{array}{c}\Exists (t_n,x_n,b_n,p_n,A_n)_{n\in\N} 
		\subseteq (0,T) \times \mc O \times \R \times  \R^d \times \Sym_d\colon  
		\\ (\Forall n\in\N\colon 
		(b_n,p_n,A_n) \in (\mc P^{+} u)(t_n,x_n)) 
		\quad \text{and}\quad 
		\\
		\lim_{n\to\infty} 
		(t_n,x_n,u(t_n,x_n),b_n,p_n,A_n) 
		= 
		(t,x,u(t,x),b,p,A)
		\end{array}
		\Bigg) 
		\Bigg\}
		\end{multline}
	(cf.~\cref{symmetric_matrices,def:parabolic_superjets}). 
\end{definition} 

\begin{definition}[Generalized parabolic subjets] \label{def:relaxed_parabolic_subjets}
	Let $ d \in \N $, 
		$ T \in (0,\infty) $, 
	let $ \mc O \subseteq \R^d $ be a non-empty open set, 
	let $ t \in ( 0, T ) $, 
		$ x \in \mc O $, 
	let $ \langle\cdot,\cdot\rangle \colon \R^d\times\R^d \to \R $ be the standard Euclidean scalar product on $ \R^d $, 
	let $ \norm{\cdot} \colon \R^d \to [0,\infty) $ be the standard Euclidean norm on $ \R^d $, 
	and let $ u \colon ( 0, T ) \times \mc O \to \R $ be a function. 
	Then we denote by $ (\JetClosure^{-}_{d,T,\mc O} u)( t, x ) $ (we denote by $ (\JetClosure^{-} u)( t, x ) $) the set given by
		\begin{multline}
		(\JetClosure^{-}_{d,T,\mc O} u)( t, x )
		= 
		(\JetClosure^{-} u)(t,x) 
		= 
		\Bigg\{ 
		(b,p,A) \in \R \times \R^d \times \Sym_d \colon 
		\\
		\Bigg( 
		\begin{array}{c}\Exists (t_n,x_n,b_n,p_n,A_n)_{n\in\N}
		\subseteq (0,T) \times \mc O \times \R \times \R^d \times \Sym_d\colon  
		\\
		(\Forall n\in\N\colon 
		(b_n,p_n,A_n) \in (\mc P^{-} u)(t_n,x_n) )
		\quad\text{and}\quad 
		\\
		\lim_{n\to\infty} 
		(t_n,x_n,u(t_n,x_n),b_n,p_n,A_n) 
		= 
		(t,x,u(t,x),b,p,A)
		\end{array}
		\Bigg) 
		\Bigg\}
		\end{multline}
	(cf.~\cref{symmetric_matrices,def:parabolic_subjets}). 
\end{definition}

\begin{lemma}\label{lem:realisation_of_parabolic_superjets_by_derivatives_of_smooth_comparison_functions}
	Let $ d \in \N $, 
		$ \varepsilon,T \in (0,\infty) $, 
	let $ \mc O \subseteq \R^d $ be a non-empty open set, 
	let $ u \colon (0,T)\times\mc O \to \R $ be upper semi-continuous, 
	and let 
		$ t \in (0,T) $, 
		$ x \in \mc O $,  
		$ (b,p,A) \in (\mc P^{+} u)(t,x) $. 
	Then there exists $ \phi \in C^{1,2}( (0,T)\times\mc O, \R ) $ such that 
		\begin{enumerate}[(i)] 
			\item \label{realisation_of_parabolic_superjets_by_derivatives_of_smooth_comparison_functions:item1} we have that $ (b,p,A+\varepsilon\operatorname{Id}_{\R^d}) = ( (\frac{\partial}{\partial t}\phi)(t,x), (\nabla_x \phi)(t,x), (\operatorname{Hess}_x \phi)(t,x) ) $ 
			and 
			\item \label{realisation_of_parabolic_superjets_by_derivatives_of_smooth_comparison_functions:item2} we have that $ u - \phi $ has a local maximum at $ ( t, x ) \in (0,T)\times\mc O $
		\end{enumerate}
	(cf.~\cref{def:parabolic_superjets}). 
\end{lemma} 

\begin{proof}[Proof of \cref{lem:realisation_of_parabolic_superjets_by_derivatives_of_smooth_comparison_functions}]
	Throughout this proof let 
		$ \Phi \colon (0,T)\times\mc O \to \R $
	satisfy for all 
		$ s \in (0,T) $, 
		$ y \in \mc O $ 
	that 
		\begin{equation} \label{realisation_of_parabolic_superjets_by_derivatives_of_smooth_comparison_functions:definition_of_Phi}
		\Phi(s,y) 
		= \begin{cases} 
		\max\!\left\{ \tfrac{ u(s,y) - u(t,x) - b(s-t) - \langle p,y-x \rangle - \frac12 \langle (A+\varepsilon\operatorname{Id}_{\R^d})(y-x), y-x \rangle }{ | s - t | } , 0 \right\} & \colon s \neq t \\
		0 & \colon s = t. 
		\end{cases} 
		\end{equation} 
	Observe that \eqref{parabolic_superjets} ensures that 
		\begin{equation} 
		\limsup_{ (0,T)\times\mc O\setminus\{(t,x)\} \ni (s,y) \to (t,x) } \left[ 
		\tfrac{ u(s,y) - u(t,x) - b(s-t) - \langle p,y-x \rangle - \frac12 \langle A(y-x), y-x \rangle }{ |s-t| + \norm{y-x}^2 } 
		\right] \leq 0.  
		\end{equation} 
	This and the assumption that $ \varepsilon \in ( 0, \infty ) $ imply that there exists $ \rho \in ( 0, \infty ) $ which satisfies that 
	\begin{enumerate}[(I)]
		\item we have that 
		$ [ t - \rho, t + \rho ] \times \{ y \in \R^d \colon \norm{ y - x } \leq \rho \} \subseteq ( 0, T ) \times \mc O $
		and 
		\item we have for all 
			$ s \in [ t - \rho, t + \rho ] $, 
			$ y \in \{ z \in \R^d \colon \norm{ z - x } \leq \rho \} $
		that 
			\begin{equation} \label{realisation_of_parabolic_superjets_by_derivatives_of_smooth_comparison_functions:epsilon_inequality}
			u(s,y)-u(t,x)-b(s-t)-\langle p,y-x \rangle-\tfrac12 \langle (A+\varepsilon\operatorname{Id}_{\R^d})(y-x), y-x \rangle \leq \varepsilon |s-t| .   
			\end{equation}
	\end{enumerate}
	Next let $ \eta \colon \R \to \R $ satisfy for all 
		$ r \in \R $ 
	that 
		\begin{equation} \label{realisation_of_parabolic_superjets_by_derivatives_of_smooth_comparison_functions:definition_of_eta}
		\eta( r ) = 
		\sup\!\left\{ \Phi(s,y) \colon (s,y) \in (0,T)\times\mc O, |s-t| \leq |r|, \norm{y-x}\leq \rho \right\}\!,
		\end{equation} 
	let $ \Psi \colon \R \to \R $ satisfy for all 
		$ r \in \R $ 
	that 
		\begin{equation} \label{realisation_of_parabolic_superjets_by_derivatives_of_smooth_comparison_functions:definition_of_Psi}
		\Psi( r ) 
		=
		\begin{cases}  
		\frac{2}{|r|} \int_0^{2r} \int_0^{s} \eta( \theta ) \,d\theta \,ds & \colon r \neq 0 \\
		0 & \colon r = 0, 
		\end{cases}
		\end{equation} 
	and let 
	$ \psi \colon (0,T)\times\mc O \to \R $ 	
	satisfy for all 
		$ s \in (0,T) $, 
		$ y \in \mc O $
	that 
		$ \psi( s, y ) = b(s-t) + \Psi( s - t ) + \langle p, y-x \rangle + \frac12 \langle (A+\varepsilon\operatorname{Id}_{\R^d})(y-x), y-x \rangle $. 
	Note that the assumption that $ u $ is upper semi-continuous, 
	\eqref{realisation_of_parabolic_superjets_by_derivatives_of_smooth_comparison_functions:definition_of_Phi}, \eqref{realisation_of_parabolic_superjets_by_derivatives_of_smooth_comparison_functions:epsilon_inequality},  \eqref{realisation_of_parabolic_superjets_by_derivatives_of_smooth_comparison_functions:definition_of_eta}, and \eqref{realisation_of_parabolic_superjets_by_derivatives_of_smooth_comparison_functions:definition_of_Psi} ensure that 
	\begin{enumerate}[(a)] 
		\item we have for all $ r \in [0,\infty)$, $ s \in [r,\infty) $ that $ \eta(r) \leq \eta(s) < \infty $, $ \eta(0) = 0 $, and $ \eta(-r) = \eta(r) $,
		\item we have that $ \Psi \in C^1(\R,\R) $, 
		\item we have that $ \Psi^{\prime}(0) = 0 $, 
		\item we have for all 
			$ s \in [t-\rho,t+\rho] $, 
			$ y \in \{ z \in \R^d \colon \Norm{z-x} \leq \rho \} $ 
		that 
			\begin{equation}
			\begin{split}  
			& 
			u(s,y) - u(t,x) - b(s-t) - \langle p, y-x \rangle - \tfrac12 \langle (A+\varepsilon\operatorname{Id}_{\R^d})(y-x), y-x \rangle 
			\\
			& \leq |s-t|\Phi(s,y) \leq |s-t|\eta(|s-t|), 
			\end{split}
			\end{equation}  
		and 
		\item we have for all 
			$ r \in \R $
		that 
			\begin{equation} 
			\Psi( r ) = \frac{ 2 }{ |r| } \int_0^{2|r|} \int_0^s \eta(\theta) \,d\theta\,ds 
			\geq \frac{2}{|r|} \int_{|r|}^{2|r|} \int_{|r|}^s \eta(r) \,d\theta\,ds 
			= |r| \eta(r) . 
			\end{equation} 
	\end{enumerate} 
	The fact that for all 
		$ s \in (0,T) $, 
		$ y \in \mc O $ 
	we have that 
		$ \psi(s,y) = b(s-t) + \Psi( s - t ) + \langle p, y-x \rangle + \frac12 \langle (A+\varepsilon\operatorname{Id}_{\R^d})(y-x), y-x \rangle $ 
	hence ensures that 
	\begin{enumerate}[(A)]
		\item we have that $ \psi \in C^{1,2}((0,T)\times\mc O,\R) $, 
		\item we have that $ (b,p,A+\varepsilon\operatorname{Id}_{\R^d}) = ((\frac{\partial}{\partial t}\psi)(t,x),(\nabla_x\psi)(t,x),(\operatorname{Hess}_x\psi)(t,x)) $,
		and 
		\item we have for all 
			$ s \in [t-\rho,t+\rho] $, 
			$ y \in \{z\in\R^d \colon \norm{z-x}\leq \rho \} $ 
		that 
			$ u(s,y)-\psi(s,y) \leq u(t,x) = u(t,x)-\psi(t,x) $. 
	\end{enumerate} 
	This establishes Items~\eqref{realisation_of_parabolic_superjets_by_derivatives_of_smooth_comparison_functions:item1} and \eqref{realisation_of_parabolic_superjets_by_derivatives_of_smooth_comparison_functions:item2}. 
	This completes the proof of \cref{lem:realisation_of_parabolic_superjets_by_derivatives_of_smooth_comparison_functions}. 
\end{proof}

\begin{lemma} \label{lem:equivalent_conditions_for_viscosity_solutions:left_to_right}
	Let $ d \in \N $, 
		$ T \in (0,\infty) $, 
	let $ \mc O \subseteq \R^d $ be a non-empty open set, 
	let $ G \colon (0,T) \times \mc O \times \R \times \R^d \times \Sym_d \to \R $ be degenerate elliptic and upper semi-continuous, 
	and let $ u \colon [0,T]\times \mc O \to \R $ be a viscosity solution of 
		\begin{equation} 
		(\tfrac{\partial}{\partial t} u)(t,x) 
		+ 
		G( t, x, u(t,x), (\nabla_x u)(t,x), (\operatorname{Hess}_x u)(t,x) ) 
		\geq 0
		\end{equation}
	(cf.~\cref{symmetric_matrices,degenerate_elliptic,def:viscosity_subsolution}). 
	Then we have for all 
		$ t \in (0,T) $, 
		$ x \in \mc O $, 
		$ (b,p,A) \in (\JetClosure^{+}u)(t,x) $ 
	that 
		\begin{equation}\label{equivalent_conditions_for_viscosity_solutions:left_to_right:claim}
			b + G(t,x,u(t,x),p,A) \geq 0 
		\end{equation}
	(cf.~\cref{def:relaxed_parabolic_superjets}).
\end{lemma} 

\begin{proof}[Proof of \cref{lem:equivalent_conditions_for_viscosity_solutions:left_to_right}]
	Throughout this proof let  
		$ (t_n,x_n,b_n,p_n,A_n) \in (0,T)\times\mc O\times\R\times\R^d\times\Sym_d $, $ n\in\N_0 $,
	satisfy for all 
		$ n \in \N $ 
	that 
		$ (b_n,p_n,A_n) \in (\mc P^{+} u)(t_n,x_n) $
	and 
		\begin{equation} 
		\lim_{ m \to \infty } (t_m,x_m,u(t_m,x_m),b_m,p_m,A_m) = (t_0,x_0,u(t_0,x_0),b_0,p_0,A_0) .  
		\end{equation}
	Observe that \cref{lem:realisation_of_parabolic_superjets_by_derivatives_of_smooth_comparison_functions} ensures that there exist $ \phi_{\varepsilon,n} \in C^{1,2}((0,T)\times\mc O,\R) $, $ \varepsilon \in (0,\infty) $, $ n \in \N $, which satisfy that 
	\begin{enumerate}[(i)]
	\item we have for all
		$ \varepsilon \in (0,\infty) $,
		$ n \in \N $ 
	that 
		\begin{equation} 
		( b_n, p_n, A_n+\varepsilon\operatorname{Id}_{\R^d} ) 
		= ((\tfrac{\partial}{\partial t}\phi_{\varepsilon,n})(t_n,x_n), (\nabla_x\phi_{\varepsilon,n})(t_n,x_n),
		(\operatorname{Hess}_x \phi_{\varepsilon,n})(t_n,x_n) ) 
		\end{equation} 
	and 
	\item we have  for all
		$ \varepsilon \in (0,\infty) $,
		$ n \in \N $ 
	that 
		$ u - \phi_{\varepsilon,n} $ 
	has a local maximum at $ (t_n,x_n) \in (0,T)\times\mc O $.  
	\end{enumerate}
	\cref{lem:local_maximum_enough} therefore demonstrates that for all 
		$ \varepsilon \in (0,\infty) $, 
		$ n \in \N $ 
	we have that 
		\begin{equation} 
		\begin{split} 
		& 
		b_n + G( t_n, x_n, u( t_n, x_n ), p_n, A_n+\varepsilon\operatorname{Id}_{\R^d} ) 
		\\
		& = (\tfrac{\partial}{\partial t}\phi_{\varepsilon,n})(t_n,x_n) 
		+ 
		G( t_n, x_n, u(t_n,x_n), (\nabla_x \phi_{\varepsilon,n})(t_n,x_n), (\operatorname{Hess}_x \phi_{\varepsilon,n})(t_n,x_n) ) \geq 0 . 
		\end{split} 
		\end{equation} 
	The assumption that $ G $ is upper semi-continuous therefore ensures for all 
		$ n \in \N $ 
	that 
		\begin{equation} 
		\begin{split} 
		& b_n + G( t_n, x_n, u( t_n, x_n ), p_n, A_n ) 
		\geq \limsup_{ ( 0, \infty ) \ni \varepsilon \to 0 } \left[ b_n + G( t_n, x_n, u_n( t_n, x_n ), p_n, A_n+\varepsilon\operatorname{Id}_{\R^d} ) \right]
		\geq 0. 
		\end{split} 
		\end{equation} 
	Combining this with the assumption that $ G $ is upper semi-continuous proves that 
		\begin{equation} 
		b + G( t_0, x_0, u( t_0, x_0 ), p_0, A_0 ) 
		\geq 
		\limsup_{ n \to \infty } \left[ 
		b_n + G( t_n, x_n, u( t_n, x_n) )
		\right] 
		\geq 0 . 
		\end{equation} 
	This establishes \eqref{equivalent_conditions_for_viscosity_solutions:left_to_right:claim}. 
	This completes the proof of \cref{lem:equivalent_conditions_for_viscosity_solutions:left_to_right}. 
\end{proof} 

\begin{lemma} \label{lem:equivalent_conditions_for_viscosity_solutions:right_to_left}
	Let $ d \in \N $, 
		$ T \in (0,\infty) $, 
	let $ \mc O \subseteq \R^d $ be a non-empty open set, 
	let $ G \colon (0,T) \times \mc O \times \R \times \R^d \times \Sym_d \to \R$ be degenerate elliptic, 
	let $u\colon [0,T]\times\R^d\to\R$ be upper semi-continuous, 
	and assume for all 
		$ t \in (0,T) $, 
		$ x \in \mc O $, 
		$ (b,p,A) \in (\mc P^{+}u)(t,x) $ 
	that 
		\begin{equation}
		b + G(t,x,u(t,x),p,A) \geq 0 
		\end{equation}
	(cf.~\cref{symmetric_matrices,degenerate_elliptic,def:parabolic_superjets}). 
	Then we have that $u $ is a viscosity solution of 
		\begin{equation}
		(\tfrac{\partial}{\partial t} u)(t,x) 
		+ 
		G( t, x, u(t,x), (\nabla_x u)(t,x), (\operatorname{Hess}_x u)(t,x) ) 
		\geq 0
		\end{equation}
	for $ (t,x) \in (0,T)\times\mc O $
	(cf.~\cref{def:viscosity_subsolution}). 
\end{lemma}

\begin{proof}[Proof of \cref{lem:equivalent_conditions_for_viscosity_solutions:right_to_left}]
	First, observe that for all 
		$ t \in (0,T) $, 
		$ x \in \mc O $, 
		$ \phi \in C^{1,2}((0,T)\times\mc O,\R) $ 
	with 
		$ \phi \geq u $ 
	and 
		$ \phi(t,x) = u(t,x) $
	we have that 
		$ ((\frac{\partial}{\partial t}\phi)(t,x),(\nabla_x \phi)(t,x),(\operatorname{Hess}_x \phi)(t,x)) \in (\mc P^{+} u)(t,x) $. 
	Hence, we obtain that for all 
		$ t \in (0,T) $, 
		$ x \in \mc O $, 
		$ \phi \in C^{1,2}((0,T)\times\mc O,\R) $
	with 
		$ \phi \geq u $ 
	and 
		$ \phi(t,x) = u(t,x) $ 
	we have that 
		$ (\tfrac{\partial}{\partial t}\phi)(t,x) + G(t,x,u(t,x),(\nabla_x\phi)(t,x),(\operatorname{Hess}_x \phi)(t,x)) \geq 0 $. 
	This establishes that $ u $ is a viscosity solution of 
		\begin{equation} 
		(\tfrac{\partial}{\partial t}u)(t,x) 
		+ G(t,x,u(t,x),(\nabla_x u)(t,x),(\operatorname{Hess}_x u)(t,x)) \geq 0 
		\end{equation} 
	for $ (t,x) \in (0,T)\times\mc O $. 
	This completes the proof of \cref{lem:equivalent_conditions_for_viscosity_solutions:right_to_left}. 
\end{proof}

\subsection{Approximation results for viscosity solutions of suitable PDEs}
\label{subsec:approximation_of_viscosity_solutions}

\begin{lemma} \label{lem:stability_of_subsolutions}
	Let $ d \in \N $, 
		$ T \in (0,\infty) $, 
	let $ \mc O \subseteq \R^d $ be a non-empty open set, 
	let $ u_n \colon (0,T)\times\mc O \to \R $, $ n \in \N_0 $, be functions, 
	let $ G_n \colon (0,T)\times\mc O \times \R \times \R^d \times \Sym_d \to \R $, $ n \in \N_0 $, be degenerate elliptic, 
	assume that 
		$ G_0 $ is upper semi-continuous, 
	assume for all non-empty compact $ \mc K \subseteq (0,T) \times \mc O \times \R \times \R^d \times \Sym_d $ that 
		\begin{equation} \label{stability_of_subsolutions:locally_uniform_convergence}
		\limsup_{n\to\infty} 
		\left[ 
		\sup_{(t,x,r,p,A)\in\mc K} 
		\Big(  
		|u_n(t,x) - u_0(t,x)| + |G_n(t,x,r,p,A)-G_0(t,x,r,p,A)|
		\Big)
		\right] = 0, 
		\end{equation} 
	and assume for all 
		$ n \in \N $ 
	that $ u_n $ is a viscosity solution of 
		\begin{equation} \label{stability_of_subsolutions:viscosity_subsolutions}
		(\tfrac{\partial}{\partial t}u_n)(t,x) 
		+ 
		G_n(t,x,u_n(t,x),(\nabla_x u_n)(t,x),(\operatorname{Hess}_x u_n)(t,x)) \geq 0 
		\end{equation}
	for $ (t,x) \in (0,T)\times\mc O $ (cf.~\cref{def:viscosity_subsolution,degenerate_elliptic,symmetric_matrices}). 
	Then we have that $ u_0 $ is a viscosity solution of 
		\begin{equation} \label{stability_of_subsolutions:claim}
		(\tfrac{\partial}{\partial t}u_0)(t,x) 
		+ 
		G_0(t,x,u_0(t,x),(\nabla_x u_0)(t,x),(\operatorname{Hess}_x u_0)(t,x)) \geq 0 
		\end{equation}
	for $ (t,x) \in (0,T)\times\mc O $ (cf.~\cref{def:viscosity_subsolution}). 
\end{lemma}

\begin{proof}[Proof of \cref{lem:stability_of_subsolutions}]
	Throughout this proof let 
		$\norm{\cdot}\colon\R^d\to [0,\infty)$ 
	be the standard Euclidean norm on $\R^d$, 
	let $ t_0 \in (0,T) $, 
		$ x_0 \in \mc O $, 
		$ (\phi_{\varepsilon})_{\varepsilon\in (0,\infty)} \subseteq C^{1,2}((0,T)\times\mc O,\R)$
	satisfy for all 
		$ \varepsilon \in (0,\infty) $, 
		$ s \in (0,T) $, 
		$ y \in \mc O $ 
	that 
		$ \phi_0(t_0,x_0) = u_0(t_0,x_0) $, 
		$ \phi_0(s,y) \geq u_0(s,y) $,  
	and 
		\begin{equation} \label{stability_of_subsolutions:definition_of_phi_epsilon}
		\phi_{\varepsilon}(s,y) 
		=
		\phi_0(s,y) 
		+ 
		\tfrac{\varepsilon}{2} ( |s-t_0|^2 + \norm{ y - x_0 }^2 ),     
		\end{equation}
	and let $ \eta \in (0,\infty) $ satisfy that 
		$ \{ (s,y) \in \R\times\R^d \colon \max\{|s-t_0|,\Norm{ y - x_0 }\} \leq \eta \} \subseteq (0,T)\times\mc O $. 
	Note that \eqref{stability_of_subsolutions:locally_uniform_convergence}  and the fact that for all 
		$ n \in \N $ 
	we have that $ u_n $ is upper semi-continuous ensure that $ u_0 $ is upper semi-continuous. 
	Moreover, note that \eqref{stability_of_subsolutions:locally_uniform_convergence} assures that there exists 
		$ \mf n = ( \mf n_{ \varepsilon } )_{ \varepsilon \in (0,\infty) } \colon (0,\infty) \to \N $ 
	which satisfies for all 
		$ \varepsilon \in (0,\infty) $, 
		$ n \in \N \cap [\mf n_{\varepsilon},\infty) $ 
	that 
		\begin{equation}
		\sup\!
		\big\{
		|u_n(s,y) - u_0(s,y)| 
		\colon 
		(s,y) \in (0,T) \times \mc O, 
		\max\{ |s-t_0|, \Norm{ y - x_0 } \} \leq \eta
		\big\}
		< \tfrac{\varepsilon\eta^2}{4}. 
		\end{equation}
	Combining this with \eqref{stability_of_subsolutions:definition_of_phi_epsilon} implies that for all 
		$ \varepsilon \in (0,\infty) $,
		$ n \in \N \cap [\mf n_{\varepsilon},\infty) $, 
		$ s \in (0,T) $, 
		$ y \in \mc O $ 
	with 
		$ |s-t_0| \leq \eta $, 
		$\norm{y-x_0}\leq\eta$, 
	and 
		$|s-t_0|^2+\norm{y-x_0}^2\geq\eta^2$
	we have that 
		\begin{equation}
		\begin{split} 
		u_n(t_0,x_0)
		-
		\phi_{\varepsilon}(t_0,x_0) 
		& =
		u_n(t_0,x_0)
		- 
		\phi_0(t_0,x_0)
		=
		u_n(t_0,x_0) 
		- 
		u_0(t_0,x_0)
		> 
		-\tfrac{\varepsilon\eta^2 }{4}
		\\
		& 
		> 
		u_n(s,y)
		-
		u_0(s,y)
		-
		\tfrac{\varepsilon}{2}(|s-t_0|^2+\norm{y-x_0}^2)
		\\
		& \geq 
		u_n(s,y)
		-
		\phi_0(s,y)
		-
		\tfrac{\varepsilon}{2}(|s-t_0|^2+\norm{y-x_0}^2)
		= 
		u_n(s,y)
		-
		\phi_{\varepsilon}(s,y). 
		\end{split}
		\end{equation}
	The fact that for every 
		$ \varepsilon \in (0,\infty) $,
		$ n \in \N $ 
	we have that $ u_n - \phi_{\varepsilon} $ is upper semi-continuous therefore guarantees that there exist 
		$ \mf t = (\mf t^{(\varepsilon)}_n)_{ (\varepsilon,n) \in \R\times\N } \colon \R\times\N \to (0,T) $
	and 
		$ \mf x = (\mf x^{(\varepsilon)}_{n})_{ (\varepsilon,n) \in \R\times\N } \colon \R\times\N \to \mc O $
	which satisfy for all 
		$ \varepsilon \in (0,\infty) $, 
		$ n \in \N \cap [\mf n_{\varepsilon}, \infty) $, 
		$ s \in [t_0-\eta,t_0+\eta] $, 
		$ y \in \{ z \in \mc O \colon \Norm{ z - x_0 } \leq \eta \} $ 
	that 
		$ \mf t^{(\varepsilon)}_n \in (t_0-\eta,t_0+\eta) $, 
		$ \Norm{ \mf x^{(\varepsilon)}_{n} - x_0 } < \eta $, 
	and 
		\begin{equation} \label{stability_of_subsolutions:inequality_at_local_max}
		u_n(\mf t^{(\varepsilon)}_{n},\mf x^{(\varepsilon)}_{n})
		-
		\phi_{\varepsilon}(\mf t^{(\varepsilon)}_{n},\mf x^{(\varepsilon)}_{n})
		\geq 
		u_n(s,y)
		-
		\phi_{\varepsilon}(s,y) 
		. 
		\end{equation}
	\cref{lem:local_maximum_enough} and \eqref{stability_of_subsolutions:viscosity_subsolutions} hence prove that for all 
		$ \varepsilon \in (0,\infty)$,
		$ n \in\N\cap [\mf n_{\varepsilon},\infty)$ 
	we have that 
		\begin{equation} \label{stability_of_subsolutions:phi_epsilon_inequality}
		(\tfrac{\partial}{\partial t}\phi_{\varepsilon})(\mf t^{(\varepsilon)}_{n},\mf x^{(\varepsilon)}_{n}) 
		+ 
		G_n(\mf t^{(\varepsilon)}_{n},\mf x^{(\varepsilon)}_{n},u_n(\mf t^{(\varepsilon)}_{n},\mf x^{(\varepsilon)}_{n}),
		(\nabla_x \phi_{\varepsilon} )(\mf t^{(\varepsilon)}_{n},\mf x^{(\varepsilon)}_{n}),
		(\operatorname{Hess}_x \phi_{\varepsilon})(\mf t^{(\varepsilon)}_{n},\mf x^{(\varepsilon)}_{n}))
		\geq 0.      
		\end{equation}
	Moreover, note that \eqref{stability_of_subsolutions:locally_uniform_convergence}, \eqref{stability_of_subsolutions:definition_of_phi_epsilon}, and \eqref{stability_of_subsolutions:inequality_at_local_max} imply that for all 
		$ \varepsilon \in (0,\infty) $
	we have that 
		\begin{equation} \label{stability_of_subsolutions:convergence_of_approximate_maximizers}
		\begin{split}
		0 
		&
		= 
		\limsup_{n\to\infty} \left[ 
		u_0(t_0,x_0) - u_n(t_0,x_0)
		\right]
		= 
		\limsup_{n\to\infty} \left[ 
		\phi_{\varepsilon}(t_0,x_0) - u_n(t_0,x_0) 
		\right]
		\\
		& \geq 
		\limsup_{n\to\infty}
		\left[ 
		\phi_{\varepsilon}(\mf t^{(\varepsilon)}_{n},\mf x^{(\varepsilon)}_{n})
		- 
		u_n(\mf t^{(\varepsilon)}_{n},\mf x^{(\varepsilon)}_{n})
		\right]
		\\
		& = 
		\limsup_{n\to\infty} 
		\left[ 
		\phi_0(\mf t^{(\varepsilon)}_{n},\mf x^{(\varepsilon)}_{n})
		+ 
		\tfrac{\varepsilon}{2}
		( |\mf t^{(\varepsilon)}_{n}-t_0|^2 + \Norm{\mf x^{(\varepsilon)}_{n}-x_0}^2 ) 
		- 
		u_n(\mf t^{(\varepsilon)}_{n},\mf x^{(\varepsilon)}_{n})
		\right]
		\\
		& \geq 
		\limsup_{n\to\infty}
		\left[ 
		u_0(\mf t^{(\varepsilon)}_{n},\mf x^{(\varepsilon)}_{n})
		- 
		u_n(\mf t^{(\varepsilon)}_{n},\mf x^{(\varepsilon)}_{n})
		+ 
		\tfrac{\varepsilon}{2}
		( |\mf t^{(\varepsilon)}_{n}-t_0|^2 + \Norm{\mf x^{(\varepsilon)}_{n}-x_0}^2 )
		\right]
		\\
		& = 
		\limsup_{n\to\infty} 
		\left[ 
		\tfrac{\varepsilon}{2}
		( |\mf t^{(\varepsilon)}_{n}-t_0|^2 + \Norm{\mf x^{(\varepsilon)}_{n}-x_0}^2 ) 
		\right]\!.  
		\end{split}
		\end{equation}
	The fact that $ u_0 $ is upper semi-continuous and \eqref{stability_of_subsolutions:locally_uniform_convergence} hence ensure that for all 
		$ \varepsilon \in (0,\infty) $ 
	we have that 
		\begin{equation} \label{stability_of_subsolutions:convergence_of_function_values}
		\begin{split} 
		& \limsup_{n\to\infty} \left[u_n(\mf t^{(\varepsilon)}_{n},\mf x^{(\varepsilon)}_n) - u_0(t_0,x_0)\right] 
		\\
		& = \limsup_{n\to\infty} \left[u_n(\mf t^{(\varepsilon)}_{n},\mf x^{(\varepsilon)}_{n}) - u_0(\mf t^{(\varepsilon)}_{n},\mf x^{(\varepsilon)}_{n}) + u_0(\mf t^{(\varepsilon)}_{n},\mf x^{(\varepsilon)}_{n}) - u_0(t_0,x_0)\right] 		
		\\
		& = 
		\limsup_{n\to\infty} \left[u_0(\mf t^{(\varepsilon)}_{n},\mf x^{(\varepsilon)}_{n}) - u_0(t_0,x_0)\right] 
		\leq 0 . 
		\end{split} 
		\end{equation} 
	Moreover, note that the fact that $ \phi_0 \in C^{1,2}((0,T)\times\mc O,\R) $, \eqref{stability_of_subsolutions:locally_uniform_convergence}, and \eqref{stability_of_subsolutions:inequality_at_local_max} 
	prove that for all 
		$ \varepsilon \in (0,\infty) $ 
	we have that
		\begin{equation} 
		\begin{split}
		& \liminf_{n\to\infty} \left[u_n(\mf t^{(\varepsilon)}_{n},\mf x^{(\varepsilon)}_{n}) - u_0(t_0,x_0)\right] 
		\\
		& \geq 
		\liminf_{n\to\infty} \left[u_n(t_0,x_0) - u_0(t_0,x_0) + \phi_{\varepsilon}(\mf t^{(\varepsilon)}_{n},\mf x^{(\varepsilon)}_{n}) - \phi_{\varepsilon}(t_0,x_0)\right] 
		= 0 . 
		\end{split} 
		\end{equation}
	This and \eqref{stability_of_subsolutions:convergence_of_function_values} show for all 
		$ \varepsilon \in (0,\infty) $ 
	that 
		\begin{equation}
		\limsup_{n\to\infty} |u_n(\mf t^{(\varepsilon)}_{n},\mf x^{(\varepsilon)}_{n})-u_0(t_0,x_0)| = 0 . 
		\end{equation} 
	The assumption that $ G_0 $ is upper semi-continuous, the fact that $\phi_0\in C^{1,2}((0,T)\times\mc O,\R)$,  
	\eqref{stability_of_subsolutions:locally_uniform_convergence}, 
	\eqref{stability_of_subsolutions:definition_of_phi_epsilon}, and \eqref{stability_of_subsolutions:convergence_of_approximate_maximizers} hence imply that for all 
		$ \varepsilon \in (0,\infty) $ 
	we have that 
		$ \limsup_{n\to\infty} | (\frac{\partial}{\partial t}\phi_{\varepsilon})(\mf t^{(\varepsilon)}_{n},\mf x^{(\varepsilon)}_{n}) - (\frac{\partial}{\partial t}\phi_0)(t_0,x_0)| = 0 $ 
	and 
		\begin{equation} 
		\begin{split} 
		& 	G_0(t_0,x_0,\phi_0(t_0,x_0),(\nabla_x\phi_0)(t_0,x_0),
		(\operatorname{Hess}_x\phi_0)(t_0,x_0)+\varepsilon I)
		\\
		& = 
		G_0(t_0,x_0,u_0(t_0,x_0),(\nabla_x\phi_{\varepsilon})(t_0,x_0),
		(\operatorname{Hess}_x\phi_{\varepsilon})(t_0,x_0))
		\\ 
		& \geq 
		\limsup_{n\to\infty} 
		\left[ 
		G_0(\mf t^{(\varepsilon)}_{n},\mf x^{(\varepsilon)}_{n},u_n(\mf t^{(\varepsilon)}_{n},\mf x^{(\varepsilon)}_{n}),
		(\nabla_x \phi_{\varepsilon} )(\mf t^{(\varepsilon)}_{n},\mf x^{(\varepsilon)}_{n}),
		(\operatorname{Hess}_x \phi_{\varepsilon})(\mf t^{(\varepsilon)}_{n},\mf x^{(\varepsilon)}_{n}))
		\right] 
		\\
		& = 
		\limsup_{n\to\infty} 
		\left[ 
		G_n(\mf t^{(\varepsilon)}_{n},\mf x^{(\varepsilon)}_{n},u_n(\mf t^{(\varepsilon)}_{n},\mf x^{(\varepsilon)}_{n}),
		(\nabla_x \phi_{\varepsilon} )(\mf t^{(\varepsilon)}_{n},\mf x^{(\varepsilon)}_{n}),
		(\operatorname{Hess}_x \phi_{\varepsilon})(\mf t^{(\varepsilon)}_{n},\mf x^{(\varepsilon)}_{n}))
		\right] \!.
		\end{split} 
		\end{equation} 
	Combining this with \eqref{stability_of_subsolutions:phi_epsilon_inequality} assures for all 
		$\varepsilon\in (0,\infty)$ 
	that 
	\begin{align}
	& 
	(\tfrac{\partial}{\partial t}\phi_0)(t_0,x_0)
	+ 
	G_0(t_0,x_0,\phi_0(t_0,x_0),(\nabla_x\phi_0)(t_0,x_0),
	(\operatorname{Hess}_x\phi_0)(t_0,x_0)+\varepsilon \operatorname{Id}_{\R^d})
	\geq 0.  
	\end{align}
	The assumption that $ G_0 $ is upper semi-continuous therefore demonstrates that 
		\begin{equation}
		(\tfrac{\partial}{\partial t}\phi_0)(t_0,x_0)
		+ 
		G_0(t_0,x_0,u_0(t_0,x_0),(\nabla_x\phi_0)(t_0,x_0),
		(\operatorname{Hess}_x\phi_0)(t_0,x_0))
		\geq 0. 
		\end{equation}
	This establishes \eqref{stability_of_subsolutions:claim}. 
	This completes the proof of  \cref{lem:stability_of_subsolutions}. 
\end{proof} 

\begin{cor} \label{cor:stability_of_supersolutions}
	Let $ d \in \N $, 
		$ T \in (0,\infty) $, 
	let $ \mc O \subseteq \R^d $ be a non-empty open set, 
	let $ u_n \colon (0,T)\times\mc O \to \R $, $ n \in \N_0 $, be functions, 
	let $ G_n \colon (0,T)\times\mc O \times \R \times \R^d \times \Sym_d \to \R $, $ n \in \N_0 $, be degenerate elliptic, 
	assume that $ G_0 $ is lower semi-continuous, 
	assume for all non-empty compact $ \mc K \subseteq (0,T) \times \mc O \times \R \times \R^d \times \Sym_d $ that 
		\begin{equation} \label{stability_of_supersolutions:convergence_assumption}
		\limsup_{n\to\infty} 
		\left[ 
		\sup_{(t,x,r,p,A)\in\mc K} 
		\Big(  
		|u_n(t,x) - u_0(t,x)| + |G_n(t,x,r,p,A)-G_0(t,x,r,p,A)|
		\Big)
		\right] = 0, 
		\end{equation} 
	and assume for all 
		$ n \in \N $ 
	that $ u_n $ is a viscosity solution of 
		\begin{equation} \label{stability_of_supersolutions:supersolution_assumption}
		(\tfrac{\partial}{\partial t}u_n)(t,x) 
		+ 
		G_n(t,x,u_n(t,x),(\nabla_x u_n)(t,x),(\operatorname{Hess}_x u_n)(t,x)) \leq 0 
		\end{equation}
	for $ (t,x) \in (0,T)\times\mc O $ (cf.~\cref{symmetric_matrices,degenerate_elliptic,def:viscosity_supersolution}). 
	Then we have that $ u_0 $ is a viscosity solution of 
		\begin{equation} 
		(\tfrac{\partial}{\partial t}u_0)(t,x) 
		+ 
		G_0(t,x,u_0(t,x),(\nabla_x u_0)(t,x),(\operatorname{Hess}_x u_0)(t,x)) \leq 0 
		\end{equation}
	for $ (t,x) \in (0,T)\times\mc O $ (cf.~\cref{def:viscosity_supersolution}). 			
\end{cor} 

\begin{proof}[Proof of \cref{cor:stability_of_supersolutions}]
	Throughout this proof let 
		$ v_n \colon (0,T)\times\mc O \to \R $, $ n \in \N_0 $, 
	and 
		$ H_n \colon (0,T)\times\mc O \times \R \times \R^d \times \Sym_d \to \R $, $ n \in \N_0 $, 
	satisfy for all 
		$ n \in \N_0 $, 
		$ t \in (0,T) $, 
		$ x \in \mc O $, 
		$ r \in \R $, 
		$ p \in \R^d $, 
		$ A \in \Sym_d $ 
	that 
		$ v_n(t,x) = -u_n(t,x) $ 
	and 
		$ H_n(t,x) = -G_n(t,x,-r,-p,-A) $. 
	Observe that the assumption that $ G_0 $ is lower semi-continuous ensures that $ H_0 $ is upper semi-continuous. 
	Moreover, note that the assumption that for all 
		$ n \in \N_0 $ 
	we have that $ G_n $ is degenerate elliptic shows that we have for all 
		$ n \in \N_0 $ 
	that $ H_n $ is degenerate elliptic.  
	Furthermore, observe that \eqref{stability_of_supersolutions:supersolution_assumption} assures that for all
		$ n \in \N $ 
	we have that $ v_n $ is a viscosity solution of 
		\begin{equation} \label{stability_of_supersolutions:v_subsolutions}
		(\tfrac{\partial}{\partial t}v_n)(t,x) + H_n(t,x,v_n(t,x),(\nabla_x v_n)(t,x),(\operatorname{Hess}_x v_n)(t,x)) \geq 0 
		\end{equation} 
	for $ (t,x) \in (0,T)\times\mc O $. 
	In addition, note that \eqref{stability_of_supersolutions:convergence_assumption} proves that for all non-empty compact $ \mc K \subseteq (0,T)\times\mc O\times\R\times\R^d\times\Sym_d $ we have that 
		\begin{equation} 
		\limsup_{n\to\infty} 
		\left[ 
		\sup_{(t,x,r,p,A)\in\mc K} 
		\Big(  
		|v_n(t,x) - v_0(t,x)| + |H_n(t,x,r,p,A)-H_0(t,x,r,p,A)|
		\Big)
		\right] = 0.  
		\end{equation} 
	Combining this, \eqref{stability_of_supersolutions:v_subsolutions}, 
	the fact that $ H_0 $ is upper semi-continuous, and
	the fact that for all $ n \in \N_0 $ we have that $ H_n $ is degenerate elliptic with \cref{lem:stability_of_subsolutions} demonstrates that $ v_0 $ is a viscosity solution of 
		\begin{equation} 
		(\tfrac{\partial}{\partial t}v_0)(t,x) + H_0(t,x,v_0(t,x),(\nabla_x v_0)(t,x),(\operatorname{Hess}_x v_0)(t,x)) \geq 0 
		\end{equation} 
	for $ (t,x) \in (0,T)\times\mc O $. 
	Hence, we obtain that $ u_0 $ is a viscosity solution of 
		\begin{equation} 
		(\tfrac{\partial}{\partial t}u_0)(t,x) + G_0(t,x,u_0(t,x),(\nabla_x u_0)(t,x),(\operatorname{Hess}_x u_0)(t,x)) \leq 0 
		\end{equation} 
	for $ (t,x) \in (0,T)\times\mc O $. 
	This completes the proof of  \cref{cor:stability_of_supersolutions}. 	
\end{proof}

\begin{cor}	\label{cor:stability_result_for_viscosity_solutions}
	Let $ d \in \N $, 
		$ T \in (0,\infty) $,
	let $ \mc O \subseteq  \R^d $ be a non-empty open set, 
	let $ u_n \colon (0,T) \times \mc O \to \R $, $ n \in \N_0 $, be functions, 
	let $ G_n \colon (0,T) \times \mc O \times \R \times \R^d \times \Sym_d \to \R $, $ n \in \N_0 $, be degenerate elliptic, 
	assume that $ G_0 \colon (0,T) \times \mc O \times \R \times \R^d \times \Sym_d \to \R $ is continuous, 
	assume for all non-empty compact $\cK \subseteq (0,T) \times \mc O \times \R \times \R^d \times \Sym_d $ that 
		\begin{equation} 
		\limsup_{n\to\infty} 
		\left[ 
		\sup_{(t,x,r,p,A) \in \cK} 
		\Big( 
		|G_n(t,x,r,p,A) - G_0(t,x,r,p,A)|
		+ 
		|u_n(t,x) - u_0(t,x)|
		\Big) \right] = 0, 
		\end{equation}
	and assume for all 
		$n\in\N$ 
	that $u_n$ is a viscosity solution of 
		\begin{equation} 
		(\tfrac{\partial}{\partial t}u_n)(t,x) 
		+ 
		G_n(t,x,u_n(t,x),
		(\nabla_x u_n)(t,x),
		(\operatorname{Hess}_x u_n)(t,x))
		= 0
		\end{equation}
	for $ (t,x) \in (0,T)\times\mc O $ (cf.~\cref{symmetric_matrices,degenerate_elliptic,def:viscosity_solution}). 
	Then we have that $ u_0 $ is a viscosity solution of 
		\begin{equation} \label{stability_result_for_viscosity_solutions:claim}
		(\tfrac{\partial}{\partial t}u_0)(t,x) 
		+ 
		G_0(t,x,u_0(t,x),
		(\nabla_x u_0)(t,x),
		(\operatorname{Hess}_x u_0)(t,x))
		= 0
		\end{equation}
	for $ (t,x) \in (0,T)\times\mc O $ (cf.~\cref{def:viscosity_solution}). 
\end{cor}

\begin{proof}[Proof of \cref{cor:stability_result_for_viscosity_solutions}]
	First, observe that \cref{lem:stability_of_subsolutions} ensures that $ u_0 $ is a viscosity solution of 
		\begin{equation} \label{stability_results_for_viscosity_solutions:viscosity_subsolutions}
			(\tfrac{\partial}{\partial t}u_0)(t,x) + G_0(t,x,u_0(t,x),(\nabla_x u_0)(t,x),(\operatorname{Hess}_x u_0)(t,x)) \geq 0 
		\end{equation} 
	for $ (t,x) \in (0,T)\times\mc O $. 
	Next note that \cref{cor:stability_of_supersolutions} proves that $ u_0 $ is a viscosity solution of 
		\begin{equation} 
		(\tfrac{\partial}{\partial t}u_0)(t,x) + G_0(t,x,u_0(t,x),(\nabla_x u_0)(t,x),(\operatorname{Hess}_x u_0)(t,x)) \leq 0 
		\end{equation} 
	for $ (t,x) \in (0,T)\times\mc O $. 
	Combining this with \eqref{stability_results_for_viscosity_solutions:viscosity_subsolutions} establishes \eqref{stability_result_for_viscosity_solutions:claim}.  
	This completes the proof of  \cref{cor:stability_result_for_viscosity_solutions}.
\end{proof}

\subsection{Approximation results for solutions of stochastic differential equations (SDEs)} 
\label{subsec:approximation_sdes}

\begin{lemma} \label{lem:stability_for_sdes}
	Let $d,m\in\N$, 
    	$T\in (0,\infty)$, 
 	let $\norm{\cdot}\colon\R^d\to [0,\infty)$ be the standard Euclidean norm on $\R^d$, 
 	let $\HSnorm{\cdot}\colon\R^{d\times m}\to [0,\infty)$ be the Frobenius norm on $\R^{d\times m}$,    
 	let $\mc O \subseteq  \R^d$ be an non-empty open set, 
 	let 
 	   $\mu_{n}\in C([0,T]\times\mc O,\R^d)$, $n\in\N_0$, 
 	and 
 	   $\sigma_{n}\in C([0,T]\times\mc O,\R^{d\times m})$, $n\in\N_0$,   
 	satisfy for all 
 		$ n \in \N_0 $ 
 	that 
 		\begin{equation}
  		\sup_{t\in [0,T]} 
  		\sup_{x\in\mc O} \sup_{y\in \mc O\setminus\{x\}} 
  		\left( 
 		\frac{\norm{
    	\mu_{n}(t,x) 
    	- 
    	\mu_{n}(t,y) } 
  		+ 
  		\HSnorm{ 
    	\sigma_{n}(t,x) 
    	- 
    	\sigma_{n}(t,y)} }{\norm{ x - y }} 
  		\right) 
  		< \infty, 
 		\end{equation} 
	assume that 
 		\begin{equation}
 		\label{stability_for_sdes:convergence_ass}
  		\limsup_{n\to\infty} 
  		\left[
  		\sup_{t\in [0,T]}
  		\sup_{x\in \mc O}
  		\Big( 
  		\norm{ \mu_{n}(t,x) - \mu_{0}(t,x) } 
  		+ 
  		\HSnorm{ \sigma_{n}(t,x) -  \sigma_{0}(t,x) } 
  		\Big) 
  		\right] 
  		= 
  		0,  
 		\end{equation}
	let $(\Omega,\mathcal{F},\P,(\mathbb{F}_t)_{t\in [0,T]})$ be a stochastic basis, 
 	let $W\colon [0,T]\times\Omega\to\R^m$ be a standard $(\mathbb{F}_t)_{t\in [0,T]}$-Brownian motion, 
	and for every 
		$n\in\N_0$, 
    	$t\in [0,T]$, 
    	$x\in \mc O$ 
 	let $X^{n,t,x} = (X^{n,t,x}_{s})_{s\in [t,T]}\colon [t,T]\times\Omega\to\mc O$ 
    be an $(\mathbb{F}_s)_{s\in [t,T]}$-adapted stochastic process with continuous sample paths satisfying that for all 
    	$s\in [t,T]$ 
    we have $\P$-a.s.~that 
		\begin{equation}\label{stability_for_sdes:sde}
  		X^{n,t,x}_{s} 
  		= 
  		x 
  		+ 
  		\int_t^s 
  		\mu_{n}( r, X^{n,t,x}_{r} ) \,dr
  		+ 
  		\int_t^s \sigma_{n}( r, X^{n,t,x}_{r} ) \,dW_r . 
 		\end{equation}
	Then 
		\begin{equation}
		 \label{stability_for_sdes:claim}
		\limsup_{n\to\infty} 
		\left[ 
		 \sup_{t\in [0,T]}
		 \sup_{s\in [t,T]}
		 \sup_{x\in \mc O}
		 \left( 
		 \Exp{ \norm{ X^{n,t,x}_{s} - X^{0,t,x}_{s} }^2 }
		 \right) 
		\right] 
		= 
		0. 
		\end{equation} 
\end{lemma}

\begin{proof}[Proof of \cref{lem:stability_for_sdes}]
	Throughout this proof let 
		$ L \in \R $
	satisfy for all 
		$t\in [0,T]$, 
		$x,y\in\mc O$ 
	that 
		\begin{equation}
		\label{stability_for_sdes:lipschitz_constant}
		\Norm{ 
			\mu_{0}(t,x) - \mu_{0}(t,y)
		} 
		+ 
		\HSnorm{
			\sigma_{0}(t,x) - \sigma_{0}(t,y)  
		} 
		\leq 
		L \Norm{ x - y } . 
		\end{equation}
	Note that, e.g., Karatzas \& Shreve~\cite[Theorem 5.2.9]{KaSh1991_BrownianMotionAndStochasticCalculus} 
	ensures that for all
		$ n \in \N_0 $, 
		$ t \in [0,T] $, 
		$ x \in \mc O $
	we have that 
		\begin{equation}
		\label{stability_for_sdes:l2_integrability}
		\sup_{s\in [t,T]}
		\Exp{\norm{X^{n,t,x}_{s}}^2} 
		< \infty . 
		\end{equation}
    Next observe that \eqref{stability_for_sdes:sde} proves that for all 
		$ n \in \N $, 
		$ t \in [0,T] $, 
		$ s \in [t,T] $, 
		$ x \in \mc O $ 
	we have $\P$-a.s.~that
		\begin{equation}
		\label{stability_for_sdes:difference_equation}
		X^{n,t,x}_{s} 
		-
		X^{0,t,x}_{s} 
		= 
		\int_t^s ( \mu_n( r, X^{n,t,x}_{r} ) - \mu_{0}(r, X^{0,t,x}_{r} ) ) \,dr
		+ 
		\int_t^s ( \sigma_n(r, X^{n,t,x}_{r} ) - \sigma_0(r, X^{0,t,x}_{r} ) ) \,dW_r . 
		\end{equation}
	Minkowski's inequality and It\^o's isometry hence ensure that for all 
		$ n \in \N $, 
		$ t \in [0,T] $, 
		$ s \in [t,T] $, 
		$ x \in \mc O $
	we have that 
		\begin{equation}
		\begin{split}
		\left( 
		\Exp{ 
		\norm{ X^{n,t,x}_{s} - X^{0,t,x}_{s} }^2 
		}
		\right)^{\!\nicefrac12} 
		& \leq 
		\int_t^s \left( 
			\Exp{ \norm{ \mu_n( r, X^{n,t,x}_{r} ) - \mu_{0}( r, X^{0,t,x}_{r} ) }^2} \right)^{\!\nicefrac{1}{2}} \,dr  
		\\
		& \quad 
		+ 
		\left(
		\Exp{\norm{ \int_t^s ( \sigma_n( r, X^{n,t,x}_{r} ) - \sigma_0( r, X^{0,t,x}_{r} ) ) \,dW_r }^2
		} 
		\right)^{\!\!\nicefrac{1}{2}}
		\\
		&
		\leq 
		\int_t^s \left(
		\Exp{ \norm{ \mu_n( r, X^{n,t,x}_{r} ) - \mu_{0}( r, X^{0,t,x}_{r} ) 
		}^2}
		\right)^{\!\nicefrac{1}{2}} \,dr  
		\\
		& \quad 
		+ 
		\left(
		\int_t^s 
		\Exp{\HSnorm{ \sigma_n( r, X^{n,t,x}_{r} ) - \sigma_{0}( r, X^{0,t,x}_{r} )
			}^2}
		\!\,dr
		\right)^{\!\!\nicefrac{1}{2}} . 
		\end{split}
		\end{equation}
	The fact that for all 
		$ a,b \in \R $
	we have that 
		$ (a+b)^2 \leq 2 a^2+2 b^2 $ 
	and the Cauchy-Schwarz inequality therefore show that for all 
		$ n \in \N $, 
		$ t \in [0,T] $, 
		$ s \in [t,T] $,
		$ x \in \mc O $ 
	we have that 
		\begin{equation} 
		\begin{split} 
		\Exp{\norm{ X^{n,t,x}_{s} - X^{0,t,x}_{s}}^2}
		& \leq 
		2
		\left[ 
		\int_t^s 
		\left(
		\Exp{
			\norm{ 
				\mu_n( r, X^{n,t,x}_{r} ) 
				- 
				\mu_0( r, X^{0,t,x}_{r} ) 
			}^2}\right)^{\!\nicefrac12}\,dr 
		\right]^2
		\\
		& \quad 
		+
		2 
		\int_t^s 
		\Exp{\HSnorm{ 
				\sigma_n( r, X^{n,t,x}_{r} ) - 
				\sigma_0( r, X^{0,t,x}_{r} )
			}^2}
		\!\,dr 
		\\
		&  
		\leq 
		2T \int_t^s 
		\Exp{
			\norm{ 
				\mu_{n}( r, X^{n,t,x}_{r} ) 
				- 
				\mu_{0}( r, X^{0,t,x}_{r} ) 
			}^2}\!\,dr 
		\\
		& \quad 
		+
		2 
		\int_t^s 
		\Exp{\HSnorm{ 
				\sigma_{n}( r, X^{n,t,x}_{r} ) 
				- 
				\sigma_{0}( r, X^{0,t,x}_{r} )
			}^2}
		\!\,dr 
		\end{split} 
		\end{equation} 
	This and the fact that for all 
		$ a,b \in \R $ 
	we have that 
		$ (a+b)^2 \leq 2 a^2 + 2 b^2 $ 
	prove that for all 
		$n\in\N$, 
		$t\in [0,T]$, 
		$s\in [t,T]$, 
		$x\in \mc O$ 
	we have that   		
		\begin{equation}
		\begin{split} 		
		&\Exp{\norm{ X^{n,t,x}_{s} - X^{0,t,x}_{s}}^2}	
		\\
		& 
		\leq 
		2T \int_t^s \bigg(
		2\,\Exp{\norm{ \mu_{n}( r, X^{n,t,x}_{r} ) - \mu_0( r, X^{n,t,x}_{r} ) }^2}
		+ 
		2\,\Exp{\norm{ \mu_{0}( r, X^{n,t,x}_{r} ) - \mu_0( r, X^{0,t,x}_{r} ) }^2} \bigg) \,dr 
		\\
		& 
		+ 
		2 \int_t^s \bigg( 2 \, \Exp{\HSnorm{			\sigma_n( r, X^{n,t,x}_{r} ) - 
				\sigma_{0}( r, X^{n,t,x}_{r} )
			}^2}
		+ 
		2 \, \Exp{
			\HSnorm{
				\sigma_0( r, X^{n,t,x}_{r} ) - 
				\sigma_0( r, X^{0,t,x}_{r} )
			}^2}
		\bigg) \,dr. 
		\end{split} 
		\end{equation}
	Combining this with \eqref{stability_for_sdes:lipschitz_constant} demonstrates that for all 
		$ n \in \N $, 
		$ t \in [0,T] $, 
		$ s \in [t,T] $, 
		$ x \in \mc O $ 
	we have that   
		\begin{equation}
		\begin{split}
		& \Exp{\norm{ X^{n,t,x}_{s} - X^{0,t,x}_{s}}^2}
		\leq 
		4L^2 (T+1) 
		\int_t^s 
		\Exp{\norm{ X^{n,t,x}_{r} - X^{0,t,x}_{r} }^2}
		\!\,dr
		\\
		&  
		+
		4 T(T+1) \left[
		\sup_{r\in [0,T]}
		\sup_{y\in \R^d} 
		\left(
		\norm{ 
			\mu_n(r,y) 
			- 
			\mu_0(r,y)
		}^2
		+
		\HSnorm{ 
			\sigma_n(r,y) - \sigma_0(r,y) 
		}^2
		\right)
		\right]\!. 
		\end{split}
		\end{equation}
	Gronwall's inequality and  \eqref{stability_for_sdes:l2_integrability} hence imply that for all 
		$ n \in \N $, 
		$ t \in [0,T] $, 
		$ s \in [t,T] $, 
		$ x\in \mc O $ 
	we have that 
		\begin{equation}
		\begin{split} 
		& 
		\Exp{
			\norm{ 
				X^{n,t,x}_{s}
				- 
				X^{0,t,x}_{s}
			}^2
		}
		\\
		& 
		\leq 4T(T+1) 
		\left[
		\sup_{r\in [0,T]}
		\sup_{y\in \R^d} 
		\left(
		\norm{ 
			\mu_n(r,y) - \mu_0(r,y)
		}^2 
		+
		\HSnorm{ 
			\sigma_n(r,y) - \sigma_0(r,y) 
		}^2
		\right)
		\right] 
		e^{4L^2 T(T+1)}. 
		\end{split}
		\end{equation}   
	This and \eqref{stability_for_sdes:convergence_ass} establish \eqref{stability_for_sdes:claim}. 
	This completes the proof of  \cref{lem:stability_for_sdes}. 
\end{proof}

\subsection{Existence results for viscosity solutions of linear inhomogeneous Kolmogorov PDEs}
\label{subsec:viscosity_inhomogeneous}

\begin{lemma} \label{lem:viscosity_solution_compact_lipschitz}
	Let $ d,m \in \N $, 
		$ T \in (0,\infty) $,
	let $ \langle\cdot,\cdot\rangle \colon \R^d\times\R^d \to \R $ be the standard Euclidean scalar product on $ \R^d $,
	let $ \norm{\cdot} \colon \R^d\to [0,\infty) $ be the standard Euclidean norm on $ \R^{d} $, 
	let $ \HSnorm{\cdot} \colon \R^{d\times m}\to [0,\infty) $ be the Frobenius norm on $ \R^{d\times m} $,    
	let $ \mu \in C( [0,T] \times \R^d, \R^d ) $,  
    	$ \sigma \in C( [0,T] \times \R^d, \R^{d\times m} ) $,  
    	$ g \in C(\R^d,\R) $, 
    	$ h \in C([0,T] \times \R^d,\R) $, 
    assume that $ \mu $ and $ \sigma $ have compact supports, 
    assume that 
	    \begin{equation}   
	     \sup_{t\in [0,T]}
	     \sup_{x\in \R^d}
	     \sup_{y\in \R^d\setminus\{x\}}
	     \left[ 
	      \frac{
	      \norm{ \mu(t,x) - \mu(t,y) } 
	      + 
	      \HSnorm{ \sigma(t,x) 
	      - 
	      \sigma(t,y) }}{\norm{ x - y }}
	     \right] 
	     < \infty,
	    \end{equation}
	let $ (\Omega,\mathcal{F},\P, (\mathbb{F}_t)_{t\in [0,T]}) $ 
    be a stochastic basis, 
	let $ W \colon [0,T]\times\Omega\to\R^m $ be a standard $ (\mathbb{F}_t)_{t\in [0,T]} $-Brownian motion, 
	for every 
 		$ t \in [0,T] $, 
 		$ x \in \R^d $
 	let $ X^{t,x} = (X^{t,x}_{s})_{s\in [t,T]}\colon [t,T]\times\Omega\to\R^d $ be an $(\mathbb{F}_{s})_{s\in [t,T]}$-adapted stochastic process with continuous sample paths satisfying that for all 
	 	$ s \in [t,T] $ 
	we have $\P$-a.s.~that
	 	\begin{equation}
	 	\label{viscosity_solution_compact_lipschitz:ass1}
	  	X^{t,x}_{s} 
	  	= 
	  	x 
	  	+ 
	  	\int_t^s \mu( r,X^{t,x}_{r} )\,dr 
	  	+ 
	 	\int_t^s \sigma( r,X^{t,x}_{r} )\,dW_r, 
		\end{equation}
	and let $ u \colon [0,T]\times\R^d \to \R $ satisfy for all
		$ t \in [0,T] $, 
		$ x \in \R^d $ 
	that 
		\begin{equation} \label{viscosity_solution_compact_lipschitz:definition_of_u}
		u(t,x) = \Exp{ g(X^{t,x}_{T}) + \int_t^T h( s,X^{t,x}_{s} )\,ds }
		\end{equation} 
	(cf.~\cref{lem:integrability}). 
	Then we have that $u$ is a viscosity solution of 
		\begin{equation} \label{viscosity_solution_compact_lipschitz:claim}
		(\tfrac{\partial}{\partial t} u)(t,x) 
		+
		\tfrac12 \operatorname{Trace}\!\left( 
		\sigma(t,x)[\sigma(t,x)]^{*}(\operatorname{Hess}_x u)(t,x)
		\right)
		+
		\langle \mu(t,x),(\nabla_x u)(t,x)\rangle
		+ 
		h(t,x) = 0  
	\end{equation}
	with $ u(T,x) = g(x) $ for $(t,x) \in (0,T)\times\R^d$ (cf.~\cref{def:viscosity_solution}).
\end{lemma}

\begin{proof}[Proof of \cref{lem:viscosity_solution_compact_lipschitz}] 
	Throughout this proof let $\mc K = (\operatorname{supp}(\mu)\cup\operatorname{supp}(\sigma)) \subseteq [0,T]\times\R^d$, 
	let $ \rho \in (0,\infty) $ satisfy $ \mc K \subseteq [0,T]\times (-\rho,\rho)^d $, 
 	let $ \mathfrak{m}_n \in C^{\infty}([0,T]\times\R^d,\R^d) $, $ n\in\N $, 
 	and $ \mathfrak{s}_n \in C^{\infty}([0,T]\times\R^d,\R^{d\times m}) $, $ n\in\N $, 
	satisfy  
	    $ \bigcup_{n\in\N} [\operatorname{supp}(\mathfrak{m}_{n}) \cup \operatorname{supp}(\mathfrak{s}_{n})] 
	    \subseteq [0,T]\times (-\rho,\rho)^d $
 	and 
	 	\begin{equation}
	 	\label{viscosity_solution_compact_lipschitz:mu_sigma_approximations}
	 	\limsup_{n\to\infty} 
	 	\left[
	 	\sup_{t\in [0,T]}\sup_{x\in\R^d} 
	 	\Big( 
	 	\norm{ \mathfrak{m}_{n}(t,x) - \mu(t,x) } 
	 	+ 
	 	\HSnorm{ \mathfrak{s}_{n}(t,x) - \sigma(t,x) }
	 	\Big) 
	 	\right] 
	 	= 
	 	0, 
	 	\end{equation}
	let	$ \mathfrak{g}_{n} \in C^{\infty}(\R^d,\R)$, $ n \in \N $, and $ \mathfrak{h}_{n} \in C^{\infty}([0,T]\times\R^d,\R)$, $ n \in \N $, satisfy for all 		
	 	$ n \in \N $ 
	that 
		\begin{equation}
	 	\label{viscosity_solution_compact_lipschitz:g_and_h_approximation}
	 	\begin{split}
	 	\sup\nolimits_{ t \in [0,T] }
	 	\sup\nolimits_{ x \in \R^d, \norm{x} \leq n }
	  	\big( | \mathfrak{g}_{n}(x) - g(x) | + | \mathfrak{h}_{n}(t,x) - h(t,x) | \big)
 	 	\leq \tfrac{1}{n},              
 		\end{split}
 		\end{equation}
	let $ G^{n,k}\colon (0,T)\times\R^d\times\R\times\R^{d}\times\Sym_{d}\to\R $, $ n,k \in \N_0 $, 
	satisfy for all 
		$ n,k\in\N $,
		$ t \in (0,T) $, 
		$ x \in \R^d $, 
		$ r \in \R $, 
		$ p \in \R^{d} $, 
		$ A \in \Sym_{d} $ 
	that 
		\begin{equation}
		G^{0,0}( t, x , r , p , A ) 
		= 
		\tfrac12
		\operatorname{Trace}\!\left( 
		\sigma(t,x)\sigma(t,x)^{*}A\right) 
		+ 
		\langle \mu(t,x), p\rangle
		+ h(t,x), 
		\end{equation}
		\begin{equation}
		G^{0,k}( t, x, r, p, A ) 
		= 
		\tfrac12 \operatorname{Trace}\!\left( 
		\sigma(t,x)\sigma(t,x)^{*}A\right) 
		+ 
		\langle \mu(t,x), p \rangle
		+ \mathfrak{h}_{k}(t,x), 
		\end{equation}
	and 
		\begin{equation}
		G^{n,k}( t, x, r, p, A ) 
		= 
		\tfrac12 \operatorname{Trace}\!\left( 
		\mathfrak{s}_n(t,x)\mathfrak{s}_n(t,x)^{*}A \right) 
		+
		\langle \mathfrak{m}_n(t,x), p\rangle
		+ 
		\mathfrak{h}_k(t,x), 
		\end{equation}
 	for every 
 		$ t \in [0,T] $, 
 		$ x \in \R^d $, 
 		$ n \in \N $ 
 	let $ \mathfrak{X}^{n,t,x} = (\mathfrak{X}^{n,t,x}_{s})_{s\in [t,T]}\colon [t,T]\times\Omega\to\R^d $ be an $(\mathbb{F}_s)_{s\in [t,T]}$-adapted stochastic process with continuous sample paths satisfying that for all 
 		$ s \in [t,T] $ 
 	we have $\P$-a.s.~that 
 		\begin{equation}\label{viscosity_solution_compact_lipschitz:approximation_SDEs}
		\mathfrak{X}^{n,t,x}_{s} 
  		= 
  		x + \int_t^s \mathfrak{m}_{n}(r,\mathfrak{X}^{n,t,x}_{r})\,dr + \int_t^s \mathfrak{s}_{n}(r,\mathfrak{X}^{n,t,x}_{r})\,dW_r  
		\end{equation}
	(cf., for example, Karatzas \& Shreve~\cite[Theorem 5.2.9]{KaSh1991_BrownianMotionAndStochasticCalculus}), 
	and let $\mathfrak{u}^{n,k}\colon [0,T]\times\R^d\to\R$, $n\in\N_0$, $k\in\N$, satisfy for all 
		$ n,k \in \N $, 
		$ t \in [0,T] $, 
	 	$ x \in \R^d $
	that
		\begin{equation}
		\label{viscosity_solution_compact_lipschitz:definition_u_nk}
		 \mathfrak{u}^{n,k}(t,x) 
		 = 
		 \Exp{ 
		 \mathfrak{g}_{k}(\mathfrak{X}^{n,t,x}_{T}) 
		 + 
		 \int_t^T 
		 \mathfrak{h}_{k}(s,\mathfrak{X}^{n,t,x}_{s})\,ds
		 } 
		\end{equation}
	and 
	\begin{equation}
	\label{viscosity_solution_compact_lipschitz:definition_u_0k}
	 \mathfrak{u}^{0,k}(t,x) 
	 = 
	 \Exp{ 
	 \mathfrak{g}_{k}( X^{t,x}_{T} ) 
	 + 
		 \int_t^T 
	 \mathfrak{h}_{k}(s, X^{t,x}_{s} ) \,ds
	 } 
	\end{equation}
	(cf.~\cref{lem:integrability}). 
	Note that \cref{lem:smooth_solutions} (applied with
		$ g \is \mathfrak{g}_{k} $, 
		$ h \is \mathfrak{h}_{k} $, 
		$ \mu \is \mathfrak{m}_{n} $, 
		$ \sigma \is \mathfrak{s}_{n} $, 
		$ X^{t,x} \is \mf X^{n,t,x} $ 
	for 
		$ n \in \N $, 
		$ t \in [0,T] $, 
		$ x \in \mc O $
	in the notation of \cref{lem:smooth_solutions})
	establishes that for all 
		$ n,k \in \N $, 
		$ t \in [0,T] $, 
		$ x \in \R^d $ 
	we have that 
		$\mathfrak{u}^{n,k}\in C^{1,2}([0,T]\times\R^d,\R)$, 
		$\mathfrak{u}^{n,k}(T,x) = \mathfrak{g}_{k}(x)$, 
	and 
		\begin{multline}
		 (\tfrac{\partial }{\partial t}\mathfrak{u}^{n,k})(t,x) 
		 + 
		 \tfrac12 
		 \operatorname{Trace}\!\left(\mathfrak{s}_{n}(t,x)[\mathfrak{s}_{n}(t,x)]^{*}(\operatorname{Hess}_x \mathfrak{u}^{n,k})(t,x)\right)  
		 + 
		 \langle \mathfrak{m}_{n}(t,x),(\nabla_x \mathfrak{u}^{n,k})(t,x)\rangle
		 \\
		 + 
		 \mathfrak{h}_{k}(t,x) = 0 . 
		 \end{multline}
	\cref{lem:classical_solutions_are_viscosity_solutions} hence implies that for all 
		$ n,k \in \N $ 
	we have that 
		$ \mathfrak{u}^{n,k} $ 
	is a viscosity solution of 
		\begin{multline}
		\label{viscosity_solution_compact_lipschitz:u_nk_viscosity_solution}
		(\tfrac{\partial}{\partial t} \mathfrak{u}^{n,k})(t,x) 
		+
		\tfrac12 
		\operatorname{Trace}\!\left(\mathfrak{s}_{n}(t,x)[\mathfrak{s}_{n}(t,x)]^{*}(\operatorname{Hess}_x \mathfrak{u}^{n,k})(t,x)\right)   
		+ 
		\langle \mathfrak{m}_{n}(t,x),(\nabla_x \mathfrak{u}^{n,k})(t,x)\rangle
		\\
		+ 
		\mathfrak{h}_{k}(t,x) = 0 
		\end{multline}
	for $(t,x) \in (0,T) \times \R^d$.
	Next note that \eqref{viscosity_solution_compact_lipschitz:ass1}, 
	\eqref{viscosity_solution_compact_lipschitz:approximation_SDEs}, and
 	the fact that for all 
 		$ n \in \N $ 
    we have that 
		$ (\operatorname{supp}(\mathfrak{m}_n)
		\cup \operatorname{supp}(\mathfrak{s}_n)
		\cup \operatorname{supp}(\mu)
		\cup \operatorname{supp}(\sigma)) \subseteq [0,T]\times( -\rho, \rho )^d $
	demonstrate that for all 
	    $ n \in \N $, 
	    $ t \in [0,T] $, 
	    $ x \in \R^d \setminus ( -\rho, \rho )^d $
    we have that  
	    $ \P(\Forall s\in [t,T]\colon \mathfrak{X}^{n,t,x}_{s} = x = X^{t,x}_{s})=1 $
	(cf., e.g., \cite[Item~(i) in Lemma 3.4]{StochasticFixedPointEquations}).
 	Hence, we obtain for all 
 	    $ n,k \in \N $,
    	$ t \in [0,T] $, 
    	$ x \in \R^d \setminus ( -\rho, \rho )^d $
    that 
    	$ \mathfrak{u}^{n,k}(t,x) = \mathfrak{u}^{0,k}(t,x) $. 
    Combining this with \eqref{viscosity_solution_compact_lipschitz:definition_u_nk} and \eqref{viscosity_solution_compact_lipschitz:definition_u_0k}     
	assures that for all 
	    $ n,k \in \N $ 
	we have that 
		\begin{equation}
		\label{viscosity_solution_compact_lipschitz:u_estimate}
		\begin{split}
		&
		\sup_{t\in [0,T]} \sup_{x \in \R^d} \Big[ | \mathfrak{u}^{n,k}(t,x) - \mathfrak{u}^{0,k}(t,x) | \Big]
		= 
		\sup_{t\in [0,T]} \sup_{x \in ( -\rho, \rho )^d} \Big[ | \mathfrak{u}^{n,k}(t,x) - \mathfrak{u}^{0,k}(t,x) | \Big]
		\\
		& 
		\leq 
		\sup_{t\in [0,T]} \sup_{x \in ( -\rho, \rho )^d} \left( 
		\Exp{ | \mathfrak{g}_{k}(  \mathfrak{X}^{n,t,x}_{T} ) - 
		 \mathfrak{g}_{k}( X^{t,x}_{T} ) | }
		  + 
		  \Exp{
		  \int_t^T 
		   | \mathfrak{h}_{k}( s, \mathfrak{X}^{n,t,x}_{s} ) 
		   - \mathfrak{h}_{k}( s, X^{t,x}_{s} ) | \,ds }
		  \right) 
		  \!. 
		 \end{split} 
		\end{equation}
	Moreover, observe that the fact that for all 
		$ k \in \N $
	we have that 
		$ \mathfrak{g}_{k} \in C^{\infty}( \R^d, \R ) $, 
	the fact that for all 
		$ k \in \N $ 
	we have that 
    	$ \mathfrak{h}_{k} \in C^{\infty}( [0,T] \times \R^d, \R ) $, 
	the fact that $ ( -\rho, \rho )^d $ is convex, 
	the fact that $ [ -\rho, \rho ]^d $ is compact, 
	and the fact that for all 
	    $ n \in \N $,
    	$ t \in [0,T] $, 
    	$ x \in ( -\rho, \rho )^d $ 
    we have that 
		$ \P(\Forall s\in [t,T]\colon X^{t,x}_{s}\in [ -\rho, \rho ]^d)
		=
		\P(\Forall s\in [t,T]\colon \mathfrak{X}^{n,t,x}_{s}\in [-\rho,\rho]^d) 
		= 
		1 $ 
	(cf., e.g., \cite[Item~(ii) in Lemma 3.4]{StochasticFixedPointEquations}) yields that for all 
		$ n,k \in \N $ 
	we have that 
		\begin{equation}
		\label{viscosity_solution_compact_lipschitz:g_estimate}
		 \begin{split}
		 & 
		 \sup_{t\in [0,T]} 
		 \sup_{x\in ( -\rho, \rho )^d} 
		 \left( 
		 \Exp{
		 | \mathfrak{g}_{k}( \mathfrak{X}^{n,t,x}_{T} ) - 
		 \mathfrak{g}_{k}( X^{t,x}_{T} ) | }
		  \right) 
		 \\
		 & 
		 \leq 
		 \sup_{t\in [0,T]} 
		 \sup_{x\in ( -\rho, \rho )^d} 
		 \left( 
		 \Exp{ \left( \sup_{ y \in ( -\rho, \rho )^d } \norm{ (\nabla \mathfrak{g}_{k})(y) } \right) \! 
		 \Norm{ \mathfrak{X}^{n,t,x}_{T} - X^{t,x}_{T} } } 
		  \right) 
		  \\
		  & 
		  \leq 
		 \underset{< \infty}{
		 \underbrace{
		 \left( 
		 \sup_{ y \in ( -\rho, \rho )^d } 
		 \norm{ (\nabla \mathfrak{g}_{k})(y) }
		 \right) 
		 }}
		 \left[ 
		 \sup_{t\in [0,T]} 
		 \sup_{ x \in ( -\rho, \rho )^d } 
		 \left( 
		 \Exp{
		  \Norm{ \mathfrak{X}^{n,t,x}_{T} - X^{t,x}_{T} }
		  } 
		  \right)
		  \right] 
		 \end{split}
		\end{equation}	
	and 
		\begin{equation}
		\label{viscosity_solution_compact_lipschitz:h_estimate}
		\begin{split}
		 & \sup_{t\in [0,T]} 
		 \sup_{ x \in ( -\rho, \rho )^d } 
		 \left( 
		  \Exp{\int_t^T 
		  | \mathfrak{h}_{k}( s, \mathfrak{X}^{n,t,x}_{s} ) 
		  - \mathfrak{h}_{k}( s, X^{t,x}_{s} ) | \,ds }
		 \right) 
		 \\
		 & 
		 \leq 
		 \sup_{ t \in [0,T] } 
		 \sup_{ x \in ( -\rho, \rho )^d } 
		 \left( 
		 \Exp{ \int_t^T 
		  	\left( \sup_{y\in ( -\rho, \rho )^d }
		    \norm{ (\nabla_x \mathfrak{h}_k)(s,y) } 
		 \right) \!
		 \Norm{ \mathfrak{X}^{n,t,x}_{s} - X^{t,x}_{s} } \,ds }
		 \right) 
		 \\
		 & 
		 \leq
		 \left( 
		 \sup_{ t \in [0,T] }
		 \sup_{ x \in ( -\rho, \rho )^d }
		 \norm{ (\nabla_x \mathfrak{h}_{k})(t,x) } \right) 
		 \left[
		 \sup_{ t \in [0,T] } 
		 \sup_{ x \in ( -\rho, \rho )^d } 
		 \left( 
		  \Exp{\int_t^T 
		  \Norm{ \mathfrak{X}^{n,t,x}_{s} - X^{t,x}_{s} } \,ds }
		 \right)
		 \right]\!. 
		\end{split}
		\end{equation}
	Furthermore, note that \cref{lem:stability_for_sdes}  ensures that 
		\begin{equation}
		 \limsup_{n\to\infty} 
		 \left[
		 \sup_{ t \in [0,T] }
		 \sup_{ x \in ( -\rho, \rho )^d } 
		 \left( 
		 \Exp{
		  \Norm{ \mathfrak{X}^{n,t,x}_{T} - X^{t,x}_{T} }
		  }
		  + 
		  \Exp{\int_t^T 
		   \Norm{ \mathfrak{X}^{n,t,x}_{s} - X^{t,x}_{s} } \,ds }
		 \right)
		 \right]
		  = 
		  0. 
		\end{equation}
	Combining this with \eqref{viscosity_solution_compact_lipschitz:u_estimate}--\eqref{viscosity_solution_compact_lipschitz:h_estimate} guarantees that for all 
		$ k \in \N $ 
    we have that 
		\begin{equation}
		\label{viscosity_solution_compact_lipschitz:u_nk_convergence}
		 \limsup_{n\to\infty} 
		 \left[ 
		 \sup_{t\in [0,T]}
		 \sup_{x\in \R^d}
		 \Big(
		 | \mathfrak{u}^{n,k}(t,x) - \mathfrak{u}^{0,k}(t,x) |
		 \Big)
		 \right]
		 = 0. 
		\end{equation}
	Moreover, observe that \eqref{viscosity_solution_compact_lipschitz:u_nk_viscosity_solution} proves that for all 
		$ n,k \in \N $ 
	we have that $\mathfrak{u}^{n,k}$ is a viscosity solution of 
		\begin{equation}
		\label{viscosity_solution_compact_lipschitz:G_nk}
		 (\tfrac{\partial }{\partial t}\mathfrak{u}^{n,k})(t,x) 
		 + 
		 G^{n,k}(t,x,\mathfrak{u}^{n,k}(t,x),(\nabla_x \mathfrak{u}^{n,k})(t,x),(\operatorname{Hess}_x\mathfrak{u}^{n,k})(t,x))
		 = 
		 0
		\end{equation}
	for $(t,x) \in (0,T) \times \R^d$. 
	Furthermore, observe that \eqref{viscosity_solution_compact_lipschitz:mu_sigma_approximations} yields that for all non-empty compact $\cC\subseteq (0,T) \times \cO \times \R \times \R^{d} \times \Sym_{d}$ we have that 
		\begin{equation}
		 \begin{split} 
		 & 
		 \limsup_{n\to\infty} 
		 \left[\sup_{(t,x,r,p,A) \in \cC} 
		 \left|
		 G^{n,k}(t,x,r,p,A) - G^{0,k}(t,x,r,p,A)
		 \right|
		 \right]
		 \\
		 & 
		 \leq 
		 \limsup_{n\to\infty}
		 \left[\sup_{(t,x,r,p,A) \in \cC}
		 \Big( \norm{ \mu(t,x) - \mathfrak{m}_n(t,x) } \norm{ p } \Big)   
		 \right]
		 \\
		 & 
		 +
		 \limsup_{n\to\infty}
		 \left[\sup_{(t,x,r,p,A) \in \cC}
		 \Big( \HSnorm{ \mathfrak{s}_n(t,x)[\mathfrak{s}_n(t,x)]^{*}-\sigma(t,x)[\sigma(t,x)]^{*}} \HSnorm{ A }	\Big) 
		 \right]
		= 0 .  
		 \end{split}
		 \end{equation}
	This, \eqref{viscosity_solution_compact_lipschitz:u_nk_convergence}, \eqref{viscosity_solution_compact_lipschitz:G_nk}, the fact that $ G^{0,0} $ is continuous, and  \cref{cor:stability_result_for_viscosity_solutions} demonstrate that for all
		$ k \in \N $ 
	we have that $ \mathfrak{u}^{0,k} $ is a viscosity solution of 
		\begin{equation} \label{viscosity_solution_compact_lipschitz:uk_viscosity_solution}
		(\tfrac{\partial }{\partial t}\mathfrak{u}^{0,k})(t,x)
	 	+ 
	 	G^{0,k}(t,x,\mathfrak{u}^{0,k}(t,x),(\nabla_x \mathfrak{u}^{0,k})(t,x),(\operatorname{Hess}_x\mathfrak{u}^{0,k})(t,x))
		= 0
		\end{equation}
	for $(t,x)\in (0,T)\times\R^d$. 
	Moreover, observe that \eqref{viscosity_solution_compact_lipschitz:g_and_h_approximation} ensures that for all compact $ \mc C \subseteq [0,T]\times\R^d $ 
	we have that 
		\begin{equation}
		\label{viscosity_solution_compact_lipschitz:gk_convergence_to_g}
		 \begin{split}
		 & 
		 \limsup_{k\to\infty} 
		 \left[\sup_{(t,x)\in \mc C} 
		 \Exp{ | \mathfrak{g}_{k}( X^{t,x}_{T} ) - g( X^{t,x}_{T} ) | } 
		 \right]
		 \leq 
		 \limsup_{k\to\infty} 
		 \left[
		 \sup_{\substack{(t,x)\in \mc C, \\ \norm{x}\leq k}} 
		 \Exp{ | \mathfrak{g}_{k} ( X^{t,x}_{T} ) - g( X^{t,x}_{T} ) | } 
		 \right]
		 \\
		 & \leq 
		 \limsup_{k\to\infty} 
		 \left[ 
		 \sup_{t\in [0,T]}
		 \sup_{\substack{x\in \mc C \cup ( -\rho, \rho )^d, \\ \norm{x}\leq k}}
		 \Exp{ | \mathfrak{g}_{k} ( X^{t,x}_{T} ) - g( X^{t,x}_{T} ) | } 
		 \right] 
		 \\
		 & \leq 
		 \limsup_{k\to\infty} 
		 \left[ \sup_{t\in [0,T]}
		 \sup_{ \substack{x\in\R^d, \\ \norm{x}\leq k}} 
		 | \mathfrak{g}_{k} ( x ) - g ( x ) | 
		 \right] 
		 \leq 
		 \limsup_{k\to\infty} 
		 \left(\frac{1}{k}\right) = 0
		 \end{split}
		\end{equation}
	and 
		\begin{equation}
		\label{viscosity_solution_compact_lipschitz:hk_convergence_to_h}
		\begin{split}
		 & \limsup_{k\to\infty} 
		 \left[\sup_{(t,x)\in \cC}
		 \Exp{\int_t^T 
		 \left| 
		 \mathfrak{h}_{k} ( s, X^{t,x}_s ) - 
		 h( s, X^{t,x}_{s} ) 
		 \right| 
		 \,ds
		 } 
		 \right] 
		 \\
		 &  
		 = 
		 \limsup_{k\to\infty} 
		 \left[ \sup_{ \substack{ (t,x)\in \cC, \\ \norm{ x }\leq k} } 
		 \Exp{\int_t^T \left|
		 	\mathfrak{h}_{k}( s, X^{t,x}_{s} ) 
		 	- 
		 h(s, X^{t,x}_{s} ) 
		 	\right| 
		 \,ds }
		 \right] 
		 \\
		 & 
		 \leq 
		 \limsup_{k\to\infty} 
		 \left[ 
		  \sup_{t\in [0,T]}
		  \sup_{\substack{x\in\R^d, \\ \norm{x}\leq k}}
		  \Exp{ 
		  	\int_t^T \left|
		  	\mathfrak{h}_{k}( s, X^{t,x}_{s} ) 
		  	- 
		  	h(s, X^{t,x}_{s} ) 
		  	\right| 
		  	\,ds 
		  }
		 \right] 
		 \\
		 & \leq 
		 \limsup_{k\to\infty} 
		 \left[ 
		 T \sup_{t\in [0,T]} 
		 \sup_{ \substack{ x \in \R^d, \\ \norm{ x }\leq k} } 
		 \left| \mathfrak{h}_k(t,x) - h(t,x) \right|
		 \right] 
		 \leq 
		 \limsup_{k\to\infty} \left( \frac{T}{k} \right) = 0. 
		 \end{split}
		\end{equation}
	Combining this with \eqref{viscosity_solution_compact_lipschitz:definition_of_u}, \eqref{viscosity_solution_compact_lipschitz:definition_u_0k}, and \eqref{viscosity_solution_compact_lipschitz:gk_convergence_to_g} proves that for all compact $\cC\subseteq (0,T)\times\R^d$ we have that 
		\begin{equation}
		\label{viscosity_solution_compact_lipschitz:uk_convergence}
		 \limsup_{k\to\infty} 
		 \left[
		 \sup_{(t,x)\in \cC} 
		 \left| \mathfrak{u}^{0,k}(t,x) - u(t,x) \right| 
		 \right]
		 = 0 . 
		\end{equation}
	\cref{cor:stability_result_for_viscosity_solutions}, the fact that for all non-empty compact $ \mc C \subseteq (0,T)\times\R^d\times\R\times\R^d\times\Sym_d $ we have that 
		$ \limsup_{ k \to \infty } [\sup_{ (t,x,r,p,A) \in \mc C } | G^{0,k}(t,x,r,p,A) - G^{0,0}(t,x,r,p,A) | ] = 0 $,  \eqref{viscosity_solution_compact_lipschitz:uk_viscosity_solution}, and \eqref{viscosity_solution_compact_lipschitz:uk_convergence}
	show that $u$ is a viscosity solution of 
		\begin{equation}
		 (\tfrac{\partial}{\partial t}u)(t,x) 
		 + 
		 G^{0,0}(t,x,u(t,x),(\nabla_x u)(t,x),(\operatorname{Hess}_x u)(t,x))
		 = 0
		\end{equation}
	for $(t,x) \in (0,T)\times\R^d$. 
	This assures that $u$ is a viscosity of 
		\begin{equation} \label{viscosity_solution_compact_lipschitz:finally_there} 
		 (\tfrac{\partial}{\partial t}u)(t,x) 
		  +
		  \tfrac12 \operatorname{Trace}\!\left(
		   \sigma(t,x)[\sigma(t,x)]^{*}(\operatorname{Hess}_x u)(t,x)
		  \right)
		  +
		  \langle \mu(t,x),(\nabla_x u)(t,x)\rangle
		  + 
		  h(t,x) = 0  
		\end{equation}
	for $ (t,x) \in (0,T) \times \R^d $. 
	Next note that \eqref{viscosity_solution_compact_lipschitz:ass1} and \eqref{viscosity_solution_compact_lipschitz:definition_of_u} ensure that for all 
		$ x \in \R^d $ 
	we have that 
		$ u(T,x) = g(x) $.  
	Combining this with \eqref{viscosity_solution_compact_lipschitz:finally_there} establishes \eqref{viscosity_solution_compact_lipschitz:claim}. 
	This completes the proof of  \cref{lem:viscosity_solution_compact_lipschitz}.
\end{proof}

\begin{prop}
\label{prop:viscosity_solution_lyapunov}
	Let $ d,m \in \N $, 
	    $ T \in (0,\infty) $, 
  	let $ \mc O \subseteq \R^d $ be a non-empty open set, 
	let $ \langle\cdot,\cdot\rangle \colon \R^d\times\R^d \to \R $ be the standard Euclidean scalar product on $ \R^d $, 
	let $ \norm{\cdot} \colon \R^d \to [0,\infty) $ be the standard Euclidean norm on $ \R^d $, 
	let $ \HSnorm{\cdot} \colon \R^{d\times m} \to [0,\infty) $ be the Frobenius norm on $ \R^{d\times m} $,
	for every 
    	$r\in (0,\infty)$ 
    let 
    	$ O_r \subseteq \cO$ 
    satisfy 
    	$ O_r = \{ x\in\cO\colon (\norm{x}\leq r~\text{and}~\{y\in\R^d\colon \norm{y-x} < \nicefrac{1}{r}\} \subseteq \cO) \} $,
	let $ g \in C( \cO, \R ) $, 
		$ h \in C( [0,T] \times \cO, \R ) $, 
		$ \mu \in C( [0,T] \times \cO, \R^d ) $,  
	    $ \sigma \in C( [0,T] \times \cO, \R^{d\times m} ) $, 
	    $ V \in C^{1,2}([0,T]\times\cO,(0,\infty)) $, 
	assume for all 
    	$r\in (0,\infty)$ 
    that
	    \begin{equation}
	     \sup\!
	     \left(
	     \left\{
	      \frac{
	      \norm{ 
	        \mu(t,x) - \mu(t,y) } 
	      + 
	      \HSnorm{
	        \sigma(t,x) - \sigma(t,y)}
	      }
	      {
	      \norm{ x - y }
	      }
		  \colon 
		  t\in [0,T], 
		  x,y\in O_r, 
		  x \neq y
		  \right\}
		  \cup \{0\}
	     \right) 
	     < 
	     \infty, 
	    \end{equation}
	assume for all 
	    $ t \in [0,T] $, 
    	$ x \in \mc O $ 
    that 
		\begin{equation}
		(\tfrac{\partial }{\partial t} V)(t,x) 
		+ 
		\tfrac12 
		\operatorname{Trace}\!\left(
		\sigma(t,x)[\sigma(t,x)]^{*} 
		(\operatorname{Hess}_x V)(t,x) 
		\right) 
		+ 
		\langle \mu(t,x),(\nabla_x V)(t, x) \rangle
		\leq 0,   
		\end{equation}
    assume that
		$ \sup_{ r \in (0,\infty) } [ \inf_{ t \in [0,T] } \inf_{ x \in \cO\setminus O_r } V(t,x) ] = \infty $ 
	and 
		$ \inf_{ r \in (0,\infty) } [ \sup_{ t \in [0,T] } \sup_{ x \in \cO\setminus O_r } ( \frac{|g(x)|}{V(T,x)} + 		\frac{|h(t,x)|}{V(t, x)} ) ] = 0 $, 
	let $ ( \Omega, \mathcal{F}, \P, (\mathbb{F}_t )_{ t \in [0,T] } ) $ be a stochastic basis, 
	let $ W \colon [0,T]\times\Omega\to\R^m $ be a standard $(\mathbb{F}_t)_{t\in [0,T]}$-Brownian motion, 
	for every 
		$ t \in [0,T] $, 
		$ x \in \mc O $ 
	let 
    	$ X^{t,x} = (X^{t,x}_s)_{s\in [t,T]}\colon \allowbreak [t,T]\times\Omega \to \mc O $ 
    be an $ (\mathbb{F}_s)_{s\in [t,T]} $-adapted stochastic process with continuous sample paths satisfying that for all 
    	$ s \in [t,T] $ 
    we have $\P$-a.s.~that 
		\begin{equation}
		\label{viscosity_solution_lyapunov:X_SDE}
		    X^{t,x}_s 
		    = 
		    x 
		    + 
		    \int_t^s 
		    \mu( r, X^{t,x}_{r} )\,dr
		    + 
		    \int_t^s
		    \sigma( r, X^{t,x}_{r} )\,dW_r,  
		\end{equation}
	and let $ u \colon [0,T]\times\R^d\to\R $ satisfy for all 
		$ t \in [0,T] $, 
		$ x \in \R^d $
	that 
		\begin{equation} \label{viscosity_solution_lyapunov:definition_of_u}
		u(t,x) = \Exp{g( X^{t,x}_{T} )
			+ 
			\int_t^T h( s, X^{t,x}_s )\,ds}\!.
		\end{equation}  
	Then we have that $u$ is a viscosity solution of 
		\begin{equation}\label{viscosity_solution_lyapunov:claim}
		(\tfrac{\partial}{\partial t} u)(t,x) 
		+ 
		\tfrac12 \operatorname{Trace}\!\left( 
		\sigma(t,x)[\sigma(t,x)]^{*}(\operatorname{Hess}_x u)(t,x)
		\right) 
		+ 
		\langle 
		\mu(t,x),(\nabla_x u)(t,x) 
		\rangle
		+ 
		h(t,x) 
		= 
		0 
		\end{equation}
	with $ u(T,x) = g(x) $ for $(t,x)\in (0,T)\times\cO$ (cf.~\cref{def:viscosity_solution}). 
\end{prop}

\begin{proof}[Proof of \cref{prop:viscosity_solution_lyapunov}]
	Throughout this proof let 
	    $ \mathfrak{g}_n \in C( \R^d, \R ) $, $ n \in \N $, 
	and $ \mathfrak{h}_n \in C( [0,T]\times\R^d, \R ) $, $ n \in \N $, 
	be compactly supported functions which satisfy  
   	    $ \big[\bigcup_{n\in\N}\operatorname{supp}(\mathfrak{h}_n)\big] \subseteq [0,T]\times\mc O $,
	    $ \big[\bigcup_{n\in\N}\operatorname{supp}(\mathfrak{g}_{n})\big] \subseteq \mc O $, 
    and 
	    \begin{equation} 
	    \label{viscosity_solution_lyapunov:gh_approximation}
	     \limsup_{n\to\infty}
	     \left[ 
	      \sup_{t\in [0,T]}
	      \sup_{x\in \cO} 
	      \left( 
	       \frac{| \mathfrak{g}_{n}(x) - g(x) |}
	       { V(T,x) }
	       +
	       \frac{| \mathfrak{h}_{n}(t,x) - h(t,x) |
	       }
	       { V(t,x) }
	      \right) 
	     \right] 
	     = 
	     0, 
	    \end{equation}
	let $ \mathfrak{m}_{n}\in C( [0,T]\times\R^d, \R^d ) $, $ n \in \N $, 
	and $ \mathfrak{s}_{n}\in C( [0,T]\times\R^d, \R^{d\times m} ) $, $n\in\N$, 
	satisfy that 
		\begin{enumerate}[(I)]
 		\item \label{viscosity_solution_lyapunov:proof_item1} 
 		we have for all 
 			$ n \in \N $ 
 		that 
 			\begin{equation}
  			\sup_{t\in [0,T]}
  			\sup_{x,y\in \R^d, x\neq y}
  			\left[ 
   			\frac{ \norm{ \mathfrak{m}_n(t,x) - \mathfrak{m}_n(t,y) }
    		+ \norm{ \mathfrak{s}_n(t,x) - \mathfrak{s}_n(t,y) } }{ \norm{ x-y } }
  			\right] < \infty, 
 			\end{equation}
  		\item \label{viscosity_solution_lyapunov:proof_item2} 
  		we have for all 
	  		$ n \in \N $, 
	  		$ t \in [0,T] $, 
	  		$ x \in \mc O $ 
  		that 
			\begin{equation}
		 	\label{viscosity_solution_lyapunov:mus_coinciding_on_V_leq_n}
		 	\mathbbm{1}_{\{V\leq n\}}(t,x)
		  	\left[
		   	\norm{ 
		    \mathfrak{m}_n(t,x) - \mu(t,x)} 
		   	+ 
		   	\HSnorm{
		    \mathfrak{s}_n(t,x) - \sigma(t,x) }
		 	\right] 
		  	= 0, 
		 	\end{equation}
 		and
 		\item \label{viscosity_solution_lyapunov:proof_item3} 
 		we have for all 
    		$ n \in \N $, 
    		$ t \in [0,T] $, 
    		$ x \in \R^d\setminus\{V\leq n+1\} $
 		that 
			$ \norm{ \mathfrak{m}_n(t,x) } + \HSnorm{ \mathfrak{s}_n(t,x) } = 0 $, 
\end{enumerate}
	for every 
	 	$ n \in \N $, 
	 	$ t \in [0,T] $, 
	 	$ x \in \R^d $ 
	let 
	 	$ \mathfrak{X}^{n,t,x}=(\mathfrak{X}^{n,t,x}_{s})_{s\in [t,T]}\colon [t,T]\times\Omega \to \R^d $ 
	be an $(\F_s)_{s\in [t,T]}$-adapted stochastic process with continuous sample paths satisfying that for all 
	 	$ s \in [t,T] $ 
	we have $\P$-a.s.~that
		\begin{equation}
		\label{viscosity_solution_lyapunov:X_n_equation}
	 	\mathfrak{X}^{n,t,x}_{s} 
	 	= 
	 	x + \int_t^s \mathfrak{m}_n( r, \mathfrak{X}^{n,t,x}_{r} ) \,dr + \int_t^s \mathfrak{s}_n( r, \mathfrak{X}^{n,t,x}_{r} ) \,dW_r 
	 	\end{equation}
	(cf., for instance, Karatzas \& Shreve~\cite[Theorem 5.2.9]{KaSh1991_BrownianMotionAndStochasticCalculus}), 
 	let 
 		$ \mathfrak{u}^{n,k}\colon [0,T]\times\R^d \to \R $, $n \in \N_0 $, $ k \in \N $, 
 	satisfy for all 
 		$ n,k \in \N $, 
 		$ t \in [0,T] $, 
 		$ x \in \R^d $
 	that 
	 	\begin{equation}
	   	\mathfrak{u}^{n,k}(t,x) = \Exp{\mathfrak{g}_k(\mathfrak{X}^{n,t,x}_{T})+ 
	   	\int_t^T \mathfrak{h}_k( s, \mathfrak{X}^{n,t,x}_{s} )\,ds } 
	 	\end{equation}
 	and 
 		\begin{equation} \label{viscosity_solution_lyapunov:definition_u_k}
  		\mathfrak{u}^{0,k}(t,x) = 
  		\Exp{\mathfrak{g}_k(X^{t,x}_T) 
  		+
  		\int_t^T \mathfrak{h}_k( s, X^{t,x}_{s})\,ds} 
 		\end{equation}
 	(cf., e.g., \cite[Lemma 2.1]{StochasticFixedPointEquations}), 
 	and for every 
 		$ n \in \N $, 
 		$ t \in [0,T] $, 
 		$ x \in \mc O $
 	let  
		$ \tau^{t,x}_n \colon \Omega \to [t,T] $ 
	satisfy 
		$ \tau^{t,x}_n 
		= 
		\inf(\{s\in [t,T]\colon \allowbreak \max\{V(s,\mf X^{n,t,x}_{s}),V(s,X^{t,x}_{s})\}\geq n\}\cup \{T\})$. 
	Next observe that \cref{lem:viscosity_solution_compact_lipschitz} (applied with 
		$ \mu \is \mathfrak{m}_n $, 
		$ \sigma \is \mathfrak{s}_{n} $, 
		$ g \is \mathfrak{g}_k $, 
		$ h \is \mathfrak{h}_k $ 
	for $ k,n \in \N $
	in the notation of \cref{lem:viscosity_solution_compact_lipschitz}) ensures that for all
    	$ n,k \in \N $ 
 	we have that $ \mathfrak{u}^{n,k} $ is a viscosity solution of 
 		\begin{multline}\label{viscosity_solution_lyapunov:unk_viscosity_solution}
		(\tfrac{\partial}{\partial t} \mathfrak{u}^{n,k})(t,x) 
 		+ 
 		\tfrac12 \operatorname{Trace}\!\left( 
  		\mathfrak{s}_{n}(t,x)[\mathfrak{s}_{n}(t,x)]^{*}(\operatorname{Hess}_x \mathfrak{u}^{n,k})(t,x)
 		\right) 
 		+ 
 		\langle \mathfrak{m}_{n}(t,x),(\nabla_x \mathfrak{u}^{n,k})(t,x) \rangle
 		\\ 
 		+ 
 		\mathfrak{h}_{k}(t,x) 
 		= 
 		0	  
 		\end{multline}
 	for $(t,x) \in (0,T)\times\R^d$. 
 	Moreover, observe that  Items~\eqref{viscosity_solution_lyapunov:proof_item1}--\eqref{viscosity_solution_lyapunov:proof_item3} and \eqref{viscosity_solution_lyapunov:X_n_equation} assure that for all 
 		$ n \in \N $, 
 		$ t \in [0,T] $, 
 		$ x \in \mc O $ 
 	we have that 
 		\begin{equation}
  		\P\!\left(\Forall s\in [t,T]\colon \mathbbm{1}_{\{s\leq \tau^{t,x}_n\}} \mathfrak{X}^{n,t,x}_{s} = \mathbbm{1}_{\{s\leq \tau^{t,x}_n\}} X^{t,x}_s 
		\right)
  		= 1
 		\end{equation}
 	(cf., e.g., \cite[Lemma 3.5]{StochasticFixedPointEquations}). 
 	This implies that for all 
 		$ n,k \in \N $, 
 		$ t \in [0,T] $, 
 		$ x \in \mc O $ 
 	we have that 
		\begin{equation}
		\begin{split}
		& \Exp{ | \mathfrak{g}_k( \mathfrak{X}^{n,t,x}_{T} ) - \mathfrak{g}_{k}( X^{t,x}_{T} ) | }
		\\
		& = 
		\Exp{ \mathbbm{1}_{\{\tau^{t,x}_n < T\}} | \mathfrak{g}_{k}( \mathfrak{X}^{n,t,x}_{T} ) - \mathfrak{g}_{k}( X^{t,x}_{T} ) | }
		\leq 
		2
		\left[ \sup_{y\in\cO} |\mathfrak{g}_k(y)| \right] \P(\tau^{t,x}_n < T)
		\end{split}
		\end{equation}
	and
		\begin{equation}
		\begin{split}
		&
		\int_t^T  
		\Exp{ | \mathfrak{h}_{k}( s, \mathfrak{X}^{n,t,x}_{s} ) - \mathfrak{h}_{k}( s, X^{t,x}_{s} ) | }\!\,ds
		\\
		& 
		= 
		\int_t^T 
		\Exp{
			\mathbbm{1}_{\{ \tau^{t,x}_n < T \}} | \mathfrak{h}_{k}( s, \mathfrak{X}^{n,t,x}_{s} ) 
			- \mathfrak{h}_{k}( s, X^{t,x}_{s} ) | }\!\,ds
		\leq 
		2 T \left[ 
		\sup_{s\in [0,T]}\sup_{y\in\cO} |\mathfrak{h}_{k}(s,y)|\right]
		\P(\tau^{t,x}_n < T). 
		\end{split}
		\end{equation} 
	Combining this with the fact that for all 
		$ t \in [0,T] $, 
		$ x \in \mc O $, 
		$ n \in \N $ 
	we have that $ \EXP{ V( \tau^{t,x}_n, X^{t,x}_{ \tau^{t,x}_n } ) } \leq V( t, x ) $ (cf., e.g., \cite[Lemma 3.1]{StochasticFixedPointEquations}) proves that for all 
		$ n,k \in \N $
	we have that 
		\begin{equation}
		\begin{split}
		|\mathfrak{u}^{n,k}(t,x)-\mathfrak{u}^{0,k}(t,x)| 
		&
		\leq 
		2
		\left[ 
		\sup_{y\in \cO}
		|\mathfrak{g}_{k}(y)| 
		+ 
		T 
		\sup_{s\in [0,T]}
		\sup_{y\in \cO} 
		|\mathfrak{h}_{k}(s,y)|
		\right]
		\P(\tau^{t,x}_n < T)
		\\
		& 
		\leq  
		2
		\left[ 
		\sup_{y\in \cO}
		|\mathfrak{g}_{k}(y)| 
		+ 
		T 
		\sup_{s\in [0,T]}
		\sup_{y\in \cO} 
		|\mathfrak{h}_{k}(s,y)|
		\right]
		\P\!\left( V(\tau^{t,x}_n, X^{t,x}_{\tau^{t,x}_n}) \geq n \right)
		\\
		& 
		\leq 
		\frac{2}{n}
		\left[ 
		\sup_{y\in \cO}
		|\mathfrak{g}_{k}(y)| 
		+ 
		T 
		\sup_{s\in [0,T]}
		\sup_{y\in \cO} 
		|\mathfrak{h}_{k}(s,y)|
		\right] 
		\Exp{ V(\tau^{t,x}_n, X^{t,x}_{\tau^{t,x}_n}) }
		\\
		& 
		\leq 
		\frac{2}{n}
		\left[ 
		\sup_{y\in \cO}
		|\mathfrak{g}_{k}(y)| 
		+ 
		T 
		\sup_{s\in [0,T]}
		\sup_{y\in \cO} 
		|\mathfrak{h}_{k}(s,y)|
		\right] 
		V(t,x). 
		\end{split}
		\end{equation}
	This demonstrates that for all 
 		$ k \in \N $ 
 	and all compact $ \cK \subseteq [0,T]\times \mc O $ we have that 
		\begin{equation} \label{viscosity_solution_lyapunov:locally_uniform_convergence}
 		\limsup_{k\to\infty} 
 		\left[ 
    		\sup_{(t,x)\in\cK} 
        	\Big( |\mathfrak{u}^{n,k}(t,x) - \mathfrak{u}^{0,k}(t,x)| \Big) 
 			\right] 
 		= 0. 
 		\end{equation} 
 	In addition, note that the assumption that $ \sup_{ r \in (0,\infty) } [ \inf_{ t\in [0,T], x \in \R^d \setminus O_r } V(t,x) ] = \infty $ and \eqref{viscosity_solution_lyapunov:mus_coinciding_on_V_leq_n} ensure that for all compact $ \cK \subseteq [0,T]\times\mc O $ we have that 
 		\begin{equation} 
 		\limsup_{ n \to \infty } 
 		\left[ 
 		\sup_{ (t,x) \in \mc K } \Big( \norm{ \mf m_n(t,x) - \mu(t,x) } + \HSnorm{ \mf s_n(t,x) - \sigma(t,x) } \Big) 
 		\right] = 0 . 
 		\end{equation} 
 	\cref{cor:stability_result_for_viscosity_solutions}, 
 	\eqref{viscosity_solution_lyapunov:unk_viscosity_solution}, 
 	and 
	\eqref{viscosity_solution_lyapunov:locally_uniform_convergence}
 	yield that for all
 		$ k \in \N $ 
 	we have that  
		$ \mathfrak{u}^{0,k} $ 
 	is a viscosity solution of 
		\begin{multline}
		\label{viscosity_solution_lyapunov:un_viscosity_solution}
  		(\tfrac{\partial}{\partial t} \mathfrak{u}^{0,k})(t,x) 
  		+ 
  		\tfrac12 
  		\operatorname{Trace}\!\left(\sigma(t,x)[\sigma(t,x)]^{*}(\operatorname{Hess}_x \mathfrak{u}^{0,k})(t,x)\right)
  		+ 
  		\langle \mu(t,x), (\nabla_x \mathfrak{u}^{0,k})(t,x) \rangle 
  		\\
  		+ 
  		h^{(k)}(t,x) 
  		= 0 
 		\end{multline}
 	for $(t,x) \in (0,T) \times \cO$. 
 	Moreover, note that \eqref{viscosity_solution_lyapunov:definition_of_u}, \eqref{viscosity_solution_lyapunov:gh_approximation}, and  \eqref{viscosity_solution_lyapunov:definition_u_k} prove that for all compact   
 		$ \mc K \subseteq (0,T)\times\mc O $ 
 	we have that 
 		\begin{equation} 
 		\limsup_{ k \to \infty } \left[ \sup_{ (t,x) \in \mc K } | \mathfrak{u}^{0,k}(t,x) - u(t,x) | \right]  = 0
 		\end{equation} 
 	(cf., e.g., \cite[Item~(iv) in Lemma 2.2]{StochasticFixedPointEquations}). 
 	This, \eqref{viscosity_solution_lyapunov:gh_approximation}, \eqref{viscosity_solution_lyapunov:un_viscosity_solution}, and \cref{cor:stability_result_for_viscosity_solutions} demonstrate that $u$ is a viscosity solution of 
		\begin{equation} \label{viscosity_solution_lyapunov:u_viscosity_solution}
		(\tfrac{\partial}{\partial t} u)(t,x) 
  		+ 
  		\tfrac12 
  		\operatorname{Trace}\!\left(
  		\sigma(t,x)[\sigma(t,x)]^{*}  (\operatorname{Hess}_x u)(t,x)\right)
  		+ 
  		\langle 
  		\mu(t,x),(\nabla_x u)(t,x)\rangle 
  		+ 
  		h(t,x) 
  		= 0
 		\end{equation} 
 	for $ (t,x) \in (0,T)\times\mc O $.  
 	Next note that \eqref{viscosity_solution_lyapunov:definition_of_u} ensures that for all 
 		$ x \in \R^d $ 
 	we have that 
 		$ u(T,x) = g(x) $. 
 	This and \eqref{viscosity_solution_lyapunov:u_viscosity_solution} establish \eqref{viscosity_solution_lyapunov:claim}. 
 	This completes the proof of  \cref{prop:viscosity_solution_lyapunov}.
\end{proof} 

\section{Semilinear Kolmogorov PDEs}
\label{sec:semilinear_kolmogorov_pdes}

In this section we establish in \cref{thm:existence_of_fixpoint} in \cref{subsec:viscosity_semilinear} below, the main result of this article, a one-to-one correspondence between suitable solutions of certain SFPEs and suitable viscosity solutions of certain semilinear Kolmogorov PDEs and we thereby obtain an existence, uniqueness, and Feynman--Kac type representation result for viscosity solutions of semilinear Kolmogorov PDEs. 
Our proof of \cref{thm:existence_of_fixpoint} employs the following four constituents: (i) the existence and uniqueness result for solutions of SFPEs in \cite[Theorem 3.8]{StochasticFixedPointEquations}, (ii) the Feynman--Kac type representation result for viscosity solutions of linear inhomogeneous Kolmogorov PDEs in \cref{prop:viscosity_solution_lyapunov} in \cref{subsec:viscosity_inhomogeneous} above, (iii) the uniqueness result for viscosity solutions of suitable degenerate parabolic PDEs in \cref{prop:uniqueness_viscosity_semilinear} in \cref{subsec:uniqueness_viscosity_semilinear} below, and (iv) the existence and uniqueness result for solutions of SDEs in \cref{lem:krylov_existence} in \cref{subsec:sde_existence} below. 

In \cref{subsec:uniqueness_viscosity_semilinear} we establish in \cref{prop:uniqueness_viscosity_semilinear} under suitable assumptions that a semilinear Kolmogorov PDE with Lipschitz continuous nonlinearity possesses at most one viscosity solution which satisfies a certain growth condition. 
\cref{prop:uniqueness_viscosity_semilinear} generalizes Hairer et al.~\cite[Corollary 4.14]{HaHuJe2017_LossOfRegularityKolmogorov} with respect to the possible time dependence of the drift and diffusion coefficient functions of the PDE as well as with respect to the possible appearance of a one-sided Lipschitz continuous nonlinearity in the PDE. 
Our proof of \cref{prop:uniqueness_viscosity_semilinear} is strongly inspired by Hairer et al.~\cite[Section 4.3]{HaHuJe2017_LossOfRegularityKolmogorov}. 
Our proof of \cref{prop:uniqueness_viscosity_semilinear} employs the comparison result for viscosity sub- and supersolutions of suitable degenerate parabolic PDEs in \cref{cor:comparison_viscosity}. \cref{cor:comparison_viscosity}, in turn, is a rather direct consequence of \cref{lem:viscosity_domination}. 
Our proof of \cref{lem:viscosity_domination} is strongly inspired by, e.g., Crandall et al.~\cite[Section 8]{UsersGuide}, Hairer et al.~\cite[Corollary 4.11]{HaHuJe2017_LossOfRegularityKolmogorov}, and Imbert \& Silvestre~\cite[Section 2.3]{ImbertSilvestre_FullyNonlinearParabolic}. 
For completeness we provide in \cref{subsec:uniqueness_viscosity_semilinear} a detailed proof for \cref{lem:viscosity_domination}. 
Our proof of \cref{lem:viscosity_domination} is based on the well-known result in Crandall et al.~\cite[Proposition 3.7]{UsersGuide} (cf.\,also Hairer et al.~\cite[Lemma 4.9]{HaHuJe2017_LossOfRegularityKolmogorov}), which we recall in \cref{lem:lemma_cil_called_elementary} below, and on a special case of the result in Crandall et al.~\cite[Theorem 8.3]{UsersGuide} (cf.\ also Peng~\cite[Theorem 2.1 in Appendix C]{peng2010nonlinear}), which we recall in \cref{lem:jensen_ishii_lemma} below. 
In \cref{subsec:sde_existence} we establish in \cref{lem:krylov_existence} an existence result for solutions of SDEs with drift and diffusion coefficient functions which satisfy certain Lipschitz and coercivity type conditions. 
\cref{lem:krylov_existence} is essentially well-known in the scientific literature (see, e.g., Gy\"ongy \& Krylov~\cite[Corollary 2.6]{GyoengyKrylov1995_existenceStrong}). 
For completeness we provide in \cref{subsec:sde_existence} a detailed proof for \cref{lem:krylov_existence}. 
In \cref{subsec:viscosity_semilinear} we establish in \cref{thm:existence_of_fixpoint} an existence, uniqueness, and Feynman--Kac type representation result for viscosity solutions of semilinear Kolmogorov PDEs. 
Our proof of \cref{thm:existence_of_fixpoint} is based on the existence and uniqueness result for solutions of SFPEs in \cite[Theorem 3.8]{StochasticFixedPointEquations}, the Feynman--Kac type representation result for viscosity solutions of linear inhomogeneous Kolmogorov PDEs in \cref{prop:viscosity_solution_lyapunov} in \cref{subsec:viscosity_inhomogeneous}, the uniqueness result for viscosity solutions in \cref{prop:uniqueness_viscosity_semilinear} in \cref{subsec:uniqueness_viscosity_semilinear}, and the existence and uniqueness result for solutions of SDEs in \cref{lem:krylov_existence} in \cref{subsec:sde_existence}. 
We conclude this article by providing in \cref{cor:existence_of_fixpoint_product_lyapunov}, \cref{existence_of_fixpoint_polynomial_growth}, and \cref{cor:existence_bounded_heat_equation_type} below several sample applications of \cref{thm:existence_of_fixpoint}. 

\subsection{Uniqueness results for viscosity solutions of semilinear Kolmogorov PDEs}
\label{subsec:uniqueness_viscosity_semilinear}

\begin{lemma}
\label{lem:lemma_cil_called_elementary}
	Let $ d \in \N $, 
	let $ \norm{\cdot} \colon \R^d \to [0,\infty) $ be a norm on $ \R^d $, 
 	let $ \mc O \subseteq \R^d $ be a non-empty set, 
 	let $ \eta \colon \mc O \to \R $ be upper semi-continuous,
	let $ \phi \colon \mc O \to [0,\infty) $ be lower semi-continuous,  
 	assume that 
 		$ \inf_{\alpha\in (0,\infty)} [ \sup_{ y \in \cO } ( \eta( y ) - \alpha \phi( y ) ) ] \in \R $, 
	and let 
		$ x = (x_{\alpha})_{\alpha\in (0,\infty)}\colon (0,\infty) \to \mc O $
    satisfy that 
	    \begin{equation}
	    \label{lemma_cil_called_elementary:approximative_maxima}
	     \limsup_{\alpha\to\infty} 
	     \left[ 
	     \sup_{y\in\cO}
	     \big(\eta(y) - \alpha\phi(y) \big)
	     - 
	     \big(\eta(x_{\alpha}) 
	     - \alpha \phi(x_{\alpha}) 
	     \big) 
	     \right]
	     = 
	     0.  
	    \end{equation}
 	Then 
 		\begin{enumerate}[(i)]
 		\item \label{lemma_cil_called_elementary:item1}
  		we have that 
  			$ \limsup_{\alpha\to\infty} [ \alpha \phi(x_{\alpha}) ] = 0 $ 
  		and 
  		\item \label{lemma_cil_called_elementary:item2} 
  		we have for all 
  			$ \mf x \in \mc O $ 
  		and all 
  			$ \alpha_n \in (0,\infty) $, $ n \in \N $, 
		with 
  			$ \limsup_{n\to\infty} \norm{ x_{\alpha_n} - \mf x  } = 0 < \infty = \liminf_{n \to \infty} \alpha_n $ 
		that 
  			$ \phi(\mf x ) = 0 $ 
  		and 
  			$ \eta(\mf x ) = \lim_{\alpha\to\infty} 
  			[ \sup_{y\in \cO} (\eta(y)-\alpha\phi(y)) ]
  			= \sup_{y\in \phi^{-1}(0)} \eta(y) $. 
 \end{enumerate} 
\end{lemma}

\begin{proof}[Proof of \cref{lem:lemma_cil_called_elementary}] 
 	Throughout this proof let 
		$ S_{\alpha} \in (-\infty,\infty] $, $ \alpha \in (0,\infty) $, 
 	and  
		$ \varepsilon_{\alpha}\in [0,\infty] $, $ \alpha \in (0,\infty) $, 
	satisfy for all
		$ \alpha \in (0,\infty) $ 
	that 
		\begin{equation} \label{lemma_cil_called_elementary:S_alpha_epsilon_alpha}
	  	S_{\alpha} 
	  	= 
	  	\sup_{y\in\cO} 
	  	\big( 
	  	  \eta(y) - \alpha\phi(y) 
	  	\big) 
	  	\qandq
	  	\varepsilon_{\alpha} 
	  	= 
	  	\sup_{y\in\cO} 
	  	\big( 
	  	  \eta(y) - \alpha\phi(y) 
	  	\big)
	  	- 
	  	\big( 
	  	  \eta(x_{\alpha}) - \alpha\phi(x_{\alpha}) 
	  	\big).  
	 	\end{equation}
	Observe that \eqref{lemma_cil_called_elementary:S_alpha_epsilon_alpha} assures that for all 
 		$ \alpha \in (0,\infty) $ 
 	we have that 
	 	$ S_{\alpha} = \eta(x_{\alpha}) - \alpha \phi(x_{\alpha}) + \varepsilon_{\alpha} $. 
	Moreover, note that \eqref{lemma_cil_called_elementary:approximative_maxima} ensures that 
 		$ \lim_{\alpha\to\infty}  S_{\alpha} \in \R $ 
 	and 
 		$ \liminf_{\alpha\to\infty} \varepsilon_{\alpha} = \limsup_{\alpha\to\infty} \varepsilon_{\alpha} = 0 $. 
	Hence, we obtain that there exists 
		$ \mf a \in (0,\infty) $ 
	such that 
		$ \big( \bigcup_{ \alpha \in [ \mf a, \infty ) } \{ S_{\alpha}, \varepsilon_{\alpha} \} \big) \subseteq \R $.
	This yields that 
 		\begin{equation} \label{lemma_cil_called_elementary:proving_item1}
 		\begin{split}
 		0 & \leq \limsup_{\alpha\to\infty} \left(\tfrac{\alpha}{2} \phi(x_{\alpha}) \right) 
 		= \limsup_{\alpha\to\infty} \left[ \big( \eta(x_{\alpha}) - \tfrac{\alpha}{2} \phi(x_{\alpha}) \big) 
 		- \big( \eta(x_{\alpha}) - \alpha \phi(x_{\alpha}) \big) \right] 
 		\\ & \leq \limsup_{\alpha\to\infty} \left[ \sup_{y\in\cO} \big( \eta(y) - \tfrac{\alpha}{2} \phi(y) \big) - \big(  \eta(x_{\alpha}) - \alpha \phi(x_{\alpha}) \big) \right] 
 		= \limsup_{\alpha\to\infty} \left( S_{\nicefrac{\alpha}{2}} - S_{\alpha} + \varepsilon_{\alpha} \right) 
 		= 0. 
 		\end{split}
 		\end{equation}
	This establishes Item~\eqref{lemma_cil_called_elementary:item1}.  
 	It remains to establish Item\ \eqref{lemma_cil_called_elementary:item2}. 
 	For this let $ \mf x  \in \mc O $ 
 	and let $ \alpha_n \in (0,\infty) $, $ n \in \N $, 
 	satisfy $ \liminf_{n\to\infty} \alpha_n = \infty $ and $ \limsup_{n\to\infty} \norm{ x_{\alpha_n} - \mf x } = 0 $. 
 	Note that \eqref{lemma_cil_called_elementary:proving_item1} ensures that 
 		$ \liminf_{\alpha\to\infty} \phi(x_{\alpha}) = \limsup_{\alpha\to\infty} \phi(x_{\alpha}) = 0 $.
 	Combining this with the assumption that $ \phi $ is lower semi-continuous demonstrates that 
		\begin{equation} \label{lemma_cil_called_elementary:phi_of_mf_x_is_zero}
  		0 \leq \phi(\mf x ) \leq \liminf_{n\to\infty} \phi(x_{\alpha_n}) = 0.   
  		\end{equation}
 	The assumption that $ \eta $ is upper semi-continuous and the fact that for all 
 		$ y \in \mc O $ 
 	we have that 
 		$ \phi( y ) \geq 0 $ 
 	hence imply that 
 		\begin{equation}
  		\begin{split}
  		\eta(\mf x ) 
  		& 
  		\geq 
  		\limsup_{n\to\infty} 
  		\eta(x_{\alpha_n}) 
  		\geq 
  		\limsup_{n\to\infty} 
  		\big( 
  		\eta(x_{\alpha_n}) 
  		- 
  		\alpha_n \phi(x_{\alpha_n})
  		\big)
  		= 
  		\limsup_{n\to\infty} 
  		\left( 
  		S_{\alpha_n} 
  		- 
  		\varepsilon_{\alpha_n}
  		\right) 
  		\\
  		& 
  		= 
  		\lim_{n\to\infty} 
  		S_{\alpha_n}
  		\geq 
  		\sup_{y\in\phi^{-1}(0)}
  		\eta(y)
  		\geq 
  		\eta(\mf x ). 
  		\end{split}
 		\end{equation}
	This and \eqref{lemma_cil_called_elementary:phi_of_mf_x_is_zero} establish Item \eqref{lemma_cil_called_elementary:item2}. 
 	This completes the proof of  \cref{lem:lemma_cil_called_elementary}. 
\end{proof}

\begin{lemma} \label{lem:jensen_ishii_lemma}
	Let $ d,k \in \N $, 
    	$ \varepsilon,T \in (0,\infty) $, 
 	let $ \norm{\cdot} \colon (\cup_{m\in\N} \R^m) \to [0,\infty) $ satisfy for all 
 		$ m \in \N $, 
 		$ x=(x_1,x_2,\ldots,x_m) \in \R^m $ 
 	  that 
 		$ \norm{x} = (\sum_{i=1}^m |x_i|^2)^{\nicefrac12} $, 
 	let $ \HSnorm{\cdot} \colon (\cup_{m\in\N} \R^{m \times m} ) \to [0,\infty) $ satisfy for all 
		$ m \in \N $, 
 		$ A \in \R^{m\times m} $ 
 	that 
 		$ \HSnorm{A} = \sup_{x\in \R^m\setminus\{0\}} (\norm{Ax}\norm{x}^{-1}) $,  	
	let $ \mc O \subseteq \R^d $ be a non-empty open set, 
	let $ \Phi = ( \Phi(t,x) )_{ (t,x) \in (0,T)\times\mc O^k } \in C^{1,2}((0,T)\times\cO^k, \R) $, 
 	let $ G_i \colon (0,T) \times \mc O \times \R \times \R^d \times \Sym_d \to \R $, $ i \in \{1,2,\ldots,k\} $, satisfy for all $ i \in \{1,2,\ldots,k\} $ that $ G_i $ is degenerate elliptic and upper semi-continuous,
 	let $ u_i \colon (0,T) \times \mc O \to \R $, $ i \in \{1,2,\ldots, k\} $, satisfy for all 
    	$ i \in \{1,2,\ldots,k\} $ 
    that $ u_i $ is a viscosity solution of 
	    \begin{equation}   
	    \label{jensen_ishii:viscosity_solution} 
	     (\tfrac{\partial}{\partial t} u_i)(t,x) 
	     + G_i(t,x,u_i(t,x),(\nabla_x u_i)(t,x),(\operatorname{Hess}_x u_i)(t,x)) 
	     \geq 0 
	    \end{equation}
	for $ (t,x) \in (0,T) \times \mc O $, 
	and let $ (\mf t, \mf x) = (\mf t, \mf x_1, \mf x_2,\ldots, \mf x_k )\in (0,T) \times \mc O^k $ be a global maximum 
    point of  
    	$ (0,T) \times \cO^k \ni (t,x) = (t,x_1,x_2,\ldots,x_k) \mapsto [\sum_{i=1}^k u_i(t,x_i)] - \Phi(t,x_1,x_2,\ldots,x_k) \in \R $
    (cf.~\cref{symmetric_matrices,degenerate_elliptic,def:viscosity_subsolution}).    	
	Then there exist 
 		$ b_1, b_2, \ldots, b_k \in \R $, 
		$ A_1, A_2, \ldots, A_k \in \Sym_{d} $ 
	such that for all 
		$ i \in \{1,2,\ldots,k\} $ 
	we have that 
		$ ( b_i, (\nabla_{x_i}\Phi)( \mf t, \mf x), A_i ) \in (\JetClosure^{+} u_i)( \mf t, \mf x_i ) $, 
		$ \sum_{i=1}^k b_i = (\frac{\partial}{\partial t}\Phi)( \mf t, \mf x ) $, 
	and 
		\begin{equation}
		\label{jensen_ishii_lemma:matrix_inequality}
		-\left[\frac{1}{\varepsilon} + \HSnorm{ (\operatorname{Hess}_{x} \Phi)( \mf t, \mf x ) } \right] \operatorname{Id}_{\R^{kd}}
  		\leq 
  		\begin{pmatrix}
   		A_1 & \ldots & 0 \\
   		\vdots & \ddots & \vdots \\
   		0 & \ldots & A_k
  		\end{pmatrix}
  		\leq 
  		(\operatorname{Hess}_x \Phi)( \mf t, \mf x ) 
  		+ 
  		\varepsilon 
  		[ (\operatorname{Hess}_x \Phi)( \mf t, \mf x ) ]^2  
 		\end{equation}
 	(cf.~\cref{def:relaxed_parabolic_superjets}). 
\end{lemma}

\begin{proof}[Proof of \cref{lem:jensen_ishii_lemma}] 
	Throughout this proof let 
		$ v_i \colon (0,T) \times \mc O \to \R $, $ i \in \{1,2,\ldots,k\} $, 
	satisfy for all 
		$ i \in \{1,2,\ldots,k\} $, 
		$ t \in (0,T) $, 
		$ x \in \mc O $ 
	that 
		$ v_i(t,x) = u_i(T-t,x) $ 
	and let $ \Psi \colon (0,T)\times\mc O^k \to \R $ satisfy for all 
		$ t \in (0,T) $, 
		$ x = (x_1,x_2,\ldots,x_k) \in \mc O^k $ 
	that 
		\begin{equation}
		\Psi(t,x_1,x_2,\ldots,x_k) 
		= 
		\Phi(T-t, x_1,x_2,\ldots,x_k). 
		\end{equation} 
	Observe that \eqref{jensen_ishii:viscosity_solution} guarantees that
	\begin{enumerate}[(i)] 
		\item we have for all 
			$ i \in \{1,2,\ldots,k\} $ 
		that $ v_i $ is upper semi-continuous, 
		\item we have that $ \Psi\in C^{1,2}((0,T)\times\cO^k,\R) $, and 
		\item we have that $ ( T - \mf t, \mf x_1, \mf x_2, \ldots, \mf x_k ) $ is a global maximum point of $ (0,T)\times\cO^k \ni (t,x_1,x_2,\ldots,x_k) \mapsto (\sum_{i=1}^k v_i(t,x) ) - \Psi(t,x_1,x_2,\ldots,x_k) \in \R $.
	\end{enumerate}
	In addition, note that \eqref{parabolic_superjets} ensures that for all 
		$ i \in \{1,2,\ldots,k\} $, 
		$ t \in (0,T) $, 
		$ x \in \mc O $
	we have that 
		$ (\mc P^{+} v_i)(t,x) = \{ (b,p,A) \in \R\times\R^d\times\Sym_d \colon (-b,p,A) \in (\mc P^{+} u_i)(T-t,x) \} $
	and 
		\begin{equation} \label{jensen_ishii_lemma:transformation_of_superjets} 
		(\JetClosure^{1,2}_{+} v_i)(t,x) = \{ (b,p,A) \in \R\times\R^d\times\Sym_d \colon (-b,p,A) \in (\JetClosure^{+}u_i)(T-t,x) \}.
		\end{equation}   
	The fact that for all $ i \in \{1,2,\ldots,k\} $ we have that $u_i$ is a viscosity solution of 
		\begin{equation}
		(\tfrac{\partial}{\partial t} u_i)(t,x) + G_i(t,x,u_i(t,x),(\nabla_x u_i)(t,x),(\operatorname{Hess}_x u_i)(t,x) ) \geq 0 
		\end{equation} 
	for $ (t,x) \in (0,T)\times\mc O $ and \cref{lem:equivalent_conditions_for_viscosity_solutions:left_to_right} hence imply that for all 
		$ i \in \{1,2,\ldots,k\} $, 
		$ t \in (0,T) $, 
		$ x \in \mc O $, 
		$ (b,p,A) \in (\mc P^{+} v_i)(t,x) $ 
	we have that 
		\begin{equation} 
		b - G_i(T-t,x,v_i(t,x),p,A) = -[-b+G_i(T-t,x,u_i(T-t,x),p,A)] \leq 0 . 
		\end{equation}  
	This and the assumption that for all $ i \in \{1,2,\ldots,k\} $ we have that $ G_i $ is upper semi-continuous ensure that for all 
		$ i \in \{1,2,\ldots,k\} $, 
		$ M \in (0,\infty) $ 
	and all compact $ \mc K \subseteq (0,T)\times\mc O $ we have that 
		\begin{equation} 
		\begin{split}
		& 
		\sup\!\left\{ b \in \R \colon (b,p,A) \in (\mc P^{+} v_i)(t,x), (t,x) \in \mc K, |v_i(t,x)| + \norm{p} + \HSnorm{A} \leq M \right\} 
		\\
		& \leq \sup\!\left\{ G_i(T-t,x,v_i(t,x),p,A) \colon \left( \begin{array}{c} (b,p,A) \in (\mc P^{+} v_i)(t,x), (t,x) \in \mc K, \\
		|v_i(t,x)| + \norm{p} + \HSnorm{A} \leq M \end{array} \right) \right\} 
		\\
		& \leq \sup\!\big\{ G_i(s,y,r,p,A) \colon (T-s,y) \in \mc K, |r| + \norm{p} + \HSnorm{A} \leq M \big\} < \infty. 
		\end{split} 
		\end{equation} 
	Crandall et al.~\cite[Theorem 8.3]{UsersGuide} (applied with 
		$ k \is k $, 
		$ u_i \is v_i $, 
		$ \mc O_i \is \mc O $, 
		$ \varphi \is \Psi $, 
		$ \hat{t} \is T-\mf t $, 
		$ \hat{x} \is \mf x $ 
	for $ i \in \{ 1, 2, \ldots, k \} $ in the notation of Crandall et al.~\cite[Theorem 8.3]{UsersGuide}) hence guarantees that there exist 
		$ a_1, a_2, \ldots, a_k \in \R $, 
		$ A_1, A_2, \ldots, A_k \in \Sym_d $
	which satisfy that 
		\begin{enumerate}[(I)]
			\item 
			we have for all 
				$ i \in \{1,2,\ldots,k\} $  
			that 
				$ ( a_i, (\nabla_{x_i}\Psi)( T-\mf t, \mf x_1,\ldots, \mf x_k), A_i ) 
				\in (\JetClosure^{+} v_i)( T-\mf t, \mf x_i ) $,  
	 		\item we have that 
				\begin{equation}
				\begin{split}
				& -\left[ \frac{1}{\varepsilon} + \HSnorm{ (\operatorname{Hess}_x \Psi)( T - \mf t, \mf x ) }\right] \operatorname{Id}_{\R^{kd}}
				\\
				& \leq 
			  	\begin{pmatrix}
			  	 A_1 & \ldots & 0 \\
			  	\vdots & \ddots & \vdots \\
			  	 0 & \ldots & A_k 
			  	\end{pmatrix}
			  	\leq 
			  	(\operatorname{Hess}_x \Psi)( T - \mf t, \mf x ) + 
			  	\varepsilon [ ( \operatorname{Hess}_x \Psi )( T - \mf t, \mf x ) ]^2,
			  	\end{split}
				\end{equation}
	  		and  
			\item we have that 
				$ \sum_{i=1}^k a_i = (\tfrac{\partial }{\partial t}\Psi)( T - \mf t, \mf x_1, \mf x_2, \ldots, \mf x_k ) $. 
	\end{enumerate}
	This and \eqref{jensen_ishii_lemma:transformation_of_superjets} prove that  
		\begin{enumerate}[(A)]
			\item we have for all
				$ i \in \{ 1, 2, \ldots, k \} $ 
			that $ ( -a_i, (\nabla_{x_i}\Phi)(\mf t, \mf x_1, \ldots, \mf x_k ), A_i ) \in (\JetClosure^{+} u_i)( \mf t, \mf x_i ) $, 
		 	\item we have that 
		 		\begin{equation}
		 		\begin{split} 
			  	& -\left[ \frac{1}{\varepsilon} + \HSnorm{(\operatorname{Hess}_x \Phi)( \mf t, \mf x ) }\right] \operatorname{Id}_{\R^{kd}}  
		  		\\
		  		& \leq 
		 		\begin{pmatrix}
		 		A_1 & \ldots & 0 \\
				\vdots & \ddots & \vdots \\
		 		0  & \ldots & A_k 
		 		\end{pmatrix}
		 		\leq 
		 		(\operatorname{Hess}_x \Phi)( \mf t, \mf x ) + 
		 		\varepsilon [(\operatorname{Hess}_x \Phi)(\mf t, \mf x ) ]^2,
		 		\end{split}
		 		\end{equation}
			and 
			\item we have that 
				$ \sum_{i=1}^k (-a_i) = (\frac{\partial }{\partial t}\Phi)( \mf t, \mf x_1, \mf x_2, \ldots, \mf x_k ) $. 
		\end{enumerate}
This establishes \eqref{jensen_ishii_lemma:matrix_inequality}. This completes the proof of  \cref{lem:jensen_ishii_lemma}. 
\end{proof}

\begin{lemma} \label{lem:viscosity_domination}
 	Let $ d,k \in \N $, 
	    $ T \in (0,\infty) $, 
	let $ \langle\cdot,\cdot\rangle \colon \R^d\times\R^d \to\R $ be the standard Euclidean scalar product on $ \R^d $, 
 	let $ \norm{\cdot}\colon\R^d\to [0,\infty)$ be the  
    standard Euclidean norm on $\R^d$,
	let $ \mc O \subseteq  \R^d $ be a non-empty open set, 
	for every 
		$ r \in (0,\infty) $ 
    let $ O_r \subseteq \mc O $ satisfy $ O_r = \{ x \in \mc O \colon ( \norm{x}\leq r~\text{and}~\{y\in\R^d  \colon \norm{y-x}<\nicefrac{1}{r} \} \subseteq \mc O ) \} $,
 	let $ G_i \colon (0,T)\times\mc O \times\R\times\R^d\times\Sym_d\to\R $, $ i \in \{1,2,\ldots,k\} $, satisfy for all 
 		$ i \in \{1,2,\ldots,k\} $ 
 	that $ G_i $ is degenerate elliptic and upper semi-continuous, 
	let $ u_i \colon [0,T]\times \mc O \to \R $, $ i \in \{1,2,\ldots,k\} $, satisfy for all 
		$ i \in \{1,2,\ldots,k\} $ 
	that $ u_i $ is a viscosity solution of 
    	\begin{equation}
    	(\tfrac{\partial}{\partial t} u_i)(t,x) 
     	+ 
     	G_i(t,x,u_i(t,x),
     	(\nabla_x u_i)(t,x),
		(\operatorname{Hess}_{x} u_i)(t,x)) 
		\geq 0 
		\end{equation} 
	for $ (t,x) \in (0,T)\times\mc O $, 
	assume that  
		\begin{equation} \label{comparison_viscosity:boundary_condition}
		\sup_{ x \in \mc O} 
		\left[\sum_{i=1}^k u_i(T,x)\right] \leq 0     
		\qquad\text{and}\qquad
		\lim_{n\to\infty} 
		\left[
		\sup_{t\in (0,T)}
		\sup_{x\in \cO\setminus O_n} 
		\left(\sum_{i=1}^k 
		u_i(t,x) 
		\right)
		\right]
		\leq 0,   
		\end{equation}
	and assume for all 
		$ t^{(n)} \in (0,T) $, $ n \in \N_0 $, 
	and all
    	$ ( x_i^{(n)}, r_i^{(n)}, A_i^{(n)} ) \in \mc O \times \R \times \Sym_d $, $ n \in \N_0 $, $ i \in \{1,2,\ldots,k\} $, 
    with 
    	$ \limsup_{ n \to \infty } [ | t^{(n)} - t^{(0)} | + \Norm{ x^{(n)}_1 - x^{(0)}_1 } + \sqrt{n} \sum_{i=2}^k \Norm{ x_i^{(n)}-x_{i-1}^{(n)} } ] = 0 < \liminf_{n\to\infty} [ \sum_{i=1}^k r_i^{(n)} ] = \limsup_{n\to\infty} [ \sum_{i=1}^k r_i^{(n)} ] \leq  \sup_{n\in\N} [ \sum_{i=1}^k |r_i^{(n)}| ] < \infty $
    and 
        $ \Forall ( n \in \N, z_1, \ldots, z_k \in \R^d ) \colon 
        -5 \sum_{i=1}^k \norm{ z_i }^2 \leq \sum_{i=1}^k \langle z_i, A_i^{(n)} z_i \rangle \leq 5 \sum_{i=2}^k \norm{ z_i - z_{i-1} }^2 $
    that  
	    \begin{equation}
	    \label{viscosity_domination:limsup_ass}
	     \limsup_{n\to\infty} 
	     \left[ 
	      \sum_{i=1}^k 
	      G_i( 
	       t^{(n)}, 
	       x_i^{(n)}, 
	       r_i^{(n)}, 
	       n ( \mathbbm{1}_{[2,k]}(i) 
	       [x_i^{(n)}-x_{i-1}^{(n)}] 
	       + 
	       \mathbbm{1}_{[1,k-1]}(i) 
	       [x_i^{(n)} - x_{i+1}^{(n)}] ), 
	       nA^{(n)}_i
	      )      
	     \right] 
	     \leq 0 
	    \end{equation}
	(cf.~\cref{symmetric_matrices,degenerate_elliptic,def:viscosity_subsolution}).
Then we have for all 
	$t\in (0,T]$, 
	$x\in\cO$ 
that 
	$\sum_{i=1}^k u_i(t,x) \leq 0$. 
\end{lemma}

\begin{proof}[Proof of \cref{lem:viscosity_domination}]
	We intend to prove that for all 
		$ t \in (0,T] $, 
		$ x \in \mc O $
	we have that 
		$ \sum_{i=1}^k u_i(t,x) \leq 0 $ 
	by showing that for all 
 	 	$ \delta \in (0,\infty) $,  
 		$ t \in (0,T] $, 
 		$ x \in \mc O $
 	we have that 
 		$ \sum_{i=1}^k u_i(t,x) \leq \frac{k\delta}{t} $. 
	Throughout this proof let  
		$ \delta \in (0,\infty) $, 
	let $ v_i \colon [0,T]\times \mc O \to [-\infty,\infty) $, $ i \in \{1,2,\ldots,k\} $, satisfy for all 
		$ i \in \{1,2,\ldots,k\} $, 
	 	$ t \in [0,T] $, 
	 	$ x \in \mc O $ 
 	that 
 		\begin{equation} 
 		v_i(t,x) = 
 		\begin{cases} 
 		u_i(t,x) - \frac{\delta}{t} & \colon t > 0 \\
		-\infty & \colon t = 0, 
		\end{cases}
 		\end{equation}  
	let $ H_i \colon (0,T) \times \mc O \times \R \times \R^d \times \Sym_d \to \R $, $ i \in \{1,2,\ldots,k\} $, satisfy for all 
 		$ i \in \{1, 2, \ldots, k \} $, 
 		$ t \in (0,T) $, 
 		$ x \in \mc O $, 
 		$ r \in \R $,
 		$ p \in \R^d $, 
 		$ A \in \Sym_d $
 	that 
		\begin{equation} \label{viscosity_domination:H_definition}
		 H_i(t,x,r,p,A)  
		 = 
		 G_i\!\left(t,x,r+\tfrac{\delta}{t},p,A\right)
		 -\tfrac{\delta}{t^2}, 
		\end{equation}
	let 
		$ \Phi \colon [0,T]\times \mc O^k \to [0,\infty) $ 
	and  
		$ \eta \colon [0,T]\times\mc O^k \to [-\infty,\infty) $ 
	satisfy for all 
		$ t \in [0,T] $, 
		$ x = (x_1,x_2,\ldots,x_k) \in \mc O^k$ 
	that 
		$ \eta(t,x) = \sum_{i=1}^k v_i(t,x_i)$ 
	and 
		\begin{equation} \label{viscosity_domination:Phi_definition}
		\Phi(t,x) = \frac12 \left[ \sum_{i=2}^k \norm{ x_i - x_{i-1} }^2 \right]\!,  	
		\end{equation}  
	let $ \mc S \in (-\infty,\infty] $ satisfy $ \mc S = \sup_{ t \in [0,T] } \sup_{ x \in \mc O } [ \sum_{i=1}^k v_i(t,x) ] $, 
	let  
		$ S_{\alpha,r} \in (-\infty,\infty] $, $\alpha,r\in [0,\infty) $, 
	satisfy for all 
		$ \alpha,r \in [ 0, \infty ) $
	that  
		\begin{equation} 
		S_{\alpha,r} 
		= \sup_{t \in [0,T]} \sup_{x\in (O_r)^k} \left[\eta(t,x)-\alpha\Phi(t,x)\right]\!, 
		\end{equation} 
	and let 
		$ \HSnorm{\cdot} \colon \R^{(kd)\times(kd)} \to [0,\infty) $ 
	satisfy for all 
		$ A \in \R^{(kd)\times(kd)} $
	that 
		\begin{equation} 
		\HSnorm{A} = \sup \left\{ \left[\sum_{i=1}^{kd} |y_{i}|^2\right]^{\nicefrac12}\left[\sum_{i=1}^{kd} |x_i|^2\right]^{-\nicefrac12} \colon \left(  \begin{array}{c} x=(x_1,x_2,\ldots,x_{kd}) \in \R^{kd}\setminus\{0\}, \\
		y=(y_1,y_2,\ldots,y_{kd}) \in \R^{kd},\\
		 y=Ax 
		 \end{array} \right)\right\} .
		\end{equation} 
	Observe that \eqref{comparison_viscosity:boundary_condition}, the fact that $\sup_{x\in\cO}\big[\sum_{i=1}^k v_i(0,x)\big] = -\infty$, and the fact that for all 			
		$ i \in \{1,2,\ldots,k\} $ 
	we have that 
		$ v_i \leq u_i $ 
	yield that 
		\begin{equation}\label{comparison_viscosity:boundary_condition_v}
		 \sup_{ x \in \mc O } \left[ \sum_{i=1}^k v_i(T,x) \right] \leq 0 
		 \qandq
		 \limsup_{n\to\infty} 
		 \left[ 
		 \sup_{t\in [0,T]} 
		 \sup_{x\in \mc O\setminus\mc O_n}
		 \sum_{i=1}^k 
		 v_i(t,x) 
		 \right] 
		 \leq 0. 
		 \end{equation} 
	Moreover, observe that the assumption that for all 
		$ i \in \{1,2,\ldots,k\} $ 
	we have that 
		$ u_i $ 
	is upper semi-continuous implies that for all 
		$ i \in \{1,2\ldots,k\} $ 
	we have that 
		$ v_i $ 
	is upper semi-continuous. 
	Furthermore, note that  \eqref{viscosity_domination:H_definition} shows that for all 
		$ i \in \{1,2,\ldots,k\} $ 
	we have that $ v_i $ is a viscosity solution of 
		\begin{equation}\label{viscosity_domination:transformed_equation}
  		(\tfrac{\partial}{\partial t} v_i)(t,x) 
  		+ 
  		H_i(t,x,v_i(t,x),(\nabla_x v_i)(t,x),(\operatorname{Hess}_x v_i)(t,x)) \geq 0 
		\end{equation}
	for $ (t,x) \in (0,T)\times\mc O $. 
	Next we claim that for all 
		$ t \in [0,T] $, 
 		$ x \in \mc O $ 
	we have that 
		\begin{equation} \label{viscosity_domination:auxiliary_claim}
		\mc S
		= \sup_{t\in [0,T]}\sup_{x\in\mc O} \left[\sum_{i=1}^k v_i(t,x)\right] \leq 0. 
		\end{equation} 
	We intend to prove \eqref{viscosity_domination:auxiliary_claim} by contradiction. For this assume that $ \mc S \in (0,\infty] $. 
 	Observe that the hypothesis that $ \mc S \in (0,\infty] $ and \eqref{comparison_viscosity:boundary_condition_v} ensure that there exists $ N \in \N $ which satisfies that 
 		\begin{enumerate}[(i)]
 		\item we have that $ O_N \neq \emptyset $, 
 		\item we have that $ O_N $ is compact, and  
 		\item we have that 
		 	$ \sup_{t\in [0,T]} \sup_{x\in O_N} [ \sum_{i=1}^k v_i(t,x) ] = \mc S $. 
	 	\end{enumerate}
 	The fact that for all 
 		$ i \in \{1,2,\ldots,k\} $ 
 	we have that $ v_i $ is upper semi-continuous hence proves that $ \mc S \in (0,\infty) $. 
 	Moreover, note that the fact that for all 
 		$ i \in \{1,2,\ldots,k\} $ 
 	we have that 
 		$ \sup_{x\in\mc O} v_i(0,x) = -\infty $ 
 	yields that  
	 	$ \mc S = \sup_{ t \in (0,T] } \sup_{ x \in O_N } [\sum_{i=1}^k v_i(t,x)] $. 
	Note that the fact that $ \Phi \in C( [0,T]\times\mc O^k, \R ) $ and the fact that for all 
		$ i \in \{1,2,\ldots,k\} $ 
	we have that 
		$ v_i $ is upper semi-continuous  
	ensure that for all 
		$ \alpha \in (0,\infty) $ 
	we have that 
		$ [0,T]\times \mc (O_N)^k \ni (t,x) \mapsto \eta(t,x) - \alpha\Phi(t,x) \in [-\infty,\infty) $  
	is upper semi-continuous. 
	This and the fact that $ O_N $ is compact prove that there exist 
		$ t^{(\alpha)} \in [0,T] $, $ \alpha \in (0,\infty) $, 
	and 
		$ x^{(\alpha)} =	 (x_1^{(\alpha)},x_2^{(\alpha)},\ldots,x_k^{(\alpha)}) \in (O_N)^k $, $ \alpha \in (0,\infty) $, 
	which satisfy for all 
		$ \alpha \in (0,\infty) $ 
	that 
		\begin{equation}
		\begin{split}
		& 
		\eta( t^{ (\alpha) }, x^{ (\alpha) } ) - \alpha \Phi( t^{ (\alpha) }, x^{ (\alpha) } )
		= 
		\sup_{t\in [0,T]}
		\sup_{x\in (\cO_N)^k}
		\left[ 
		\eta(t,x) - \alpha \Phi(t,x)
		\right] 
		= 
		S_{\alpha,N}. 
		\end{split}
		\end{equation}
	Moreover, note that the fact that for all 
		$ t \in [0,T] $, 
		$ y \in \mc O $ 
	we have that 
		$ \eta(t,y,y,\ldots,y) = \sum_{i=1}^k v_i(t,y) $
	and the fact that for all 
		$ t \in [0,T] $, 
		$ y \in \mc O $ 
	we have that 
		$ \Phi(t,y,y,\ldots,y) = 0 $
	imply that for all 
		$ \alpha \in (0,\infty) $ 
	we have that 
		\begin{equation}\label{viscosity_domination:lower_bound_on_S_alpha}
		S_{\alpha,N} 
		\geq 
		\sup_{t\in [0,T]}
		\sup_{y\in \cO} 
		\Big[ 
		 \eta(t,y,y,\ldots,y) 
		 - 
		 \alpha \Phi(t,y,y,\ldots,y) 
		\Big]
		= 
		\sup_{t\in [0,T]}
		\sup_{y\in \cO} 
		\left[ 
		\sum_{i=1}^k v_i(t,y) 
		\right] 
		= \mc S > 0. 
 		\end{equation}
	Combining this with the fact that for all 
		$ \alpha, \beta \in (0,\infty) $ 
	with 
		$ \alpha \geq \beta $ 
	we have that 
		$ S_{\alpha,N} \leq S_{\beta,N} $ 
	demonstrates that 
		$ \liminf_{\alpha \to \infty} S_{\alpha,N} 
		= \limsup_{\alpha \to \infty} S_{\alpha,N}
		\in [ \mc S, \infty ) \subseteq \R $. 
	Moreover, note that \eqref{viscosity_domination:lower_bound_on_S_alpha} and the fact that for all 
		$ \alpha \in (0,\infty) $ 
	we have that 
		$ \sup_{ x \in \mc O^k } [\eta(0,x)-\alpha\Phi(0,x)] = -\infty $ 
	yield that for all 
		$ \alpha \in (0,\infty) $
	we have that 
		\begin{equation} 
		S_{ \alpha, N } 
		= \sup_{ t \in (0,T] } \sup_{ x \in (O_N)^k } [\eta(t,x)-\alpha\Phi(t,x)]. 
		\end{equation}
	Item~\eqref{lemma_cil_called_elementary:item1} in \cref{lem:lemma_cil_called_elementary} (applied with 
		$ \mc O \is (0,T]\times (O_N)^k $, 
		$ \eta \is \eta|_{ (0,T]\times (O_N)^k }$, 
		$ \phi \is \Phi|_{ (0,T]\times (O_N)^k }$, 
		$ x \is ( (0,\infty) \ni \alpha \mapsto (t^{(\alpha)},x^{(\alpha)}) \in (0,T] \times (O_N)^k )$ 
	in the notation of \cref{lem:lemma_cil_called_elementary}) hence ensures that 
		\begin{equation}\label{viscosity_domination:alpha_phi_goes_to_zero}
  		0 
  		= 
  		\limsup_{\alpha\to\infty} 
  		\left[ 
  		\alpha \Phi( t^{(\alpha)}, 
 		 x^{(\alpha)})
  		\right]
  		= 
  		\limsup_{\alpha\to\infty} 
  		\left[ 
  		\tfrac{\alpha}{2} \sum_{i=2}^k 
  		\Norm{ 
    		x_i^{(\alpha)} 
    		-
    		x_{i-1}^{(\alpha)}
    	}^2 \right]\!. 
 		\end{equation}
	In addition, note that the fact that $ O_N $ is compact guarantees that there exist
		$ \mf t \in [0,T] $,
		$ \mf x = ( \mf x_1, \mf x_2, \ldots, \mf x_k ) \in (O_N)^k $,  
		$ (\alpha_n)_{ n \in \N } \subseteq \N $
	which satisfy 
 		$ \liminf_{n\to\infty} \alpha_n = \infty $ 
 	and 
 		$ \limsup_{n\to \infty} [ | t^{(\alpha_n)} - \mf t | + \Norm{ x^{(\alpha_n)} - \mf x } ] = 0 $. 
	Furthermore, observe that the fact that $ \eta $ is upper semi-continuous and the fact that $ \Phi $ is continuous imply that 
		\begin{equation}
  		\eta( \mf t, \mf x ) 
  		\geq 
  		\limsup_{n\to\infty} 
  		\left[ 
  		\eta(t^{(\alpha_n)},x^{(\alpha_n)})
  			- \alpha_{n} \Phi(t^{(\alpha_n)},x^{(\alpha_n)})
  		\right] 
  		\geq \mc S > 0. 
 		\end{equation}
	Combining this with the fact that for all 
		$ x \in \mc O^k $ 
	we have that 
		$ \eta(0,x) = -\infty $ 
	shows that $ \mf t \in (0,T] $. 
	Item\ \eqref{lemma_cil_called_elementary:item2} of \cref{lem:lemma_cil_called_elementary} hence ensures that 
 		$ 0 
 		= \Phi( \mf t, \mf x )
  		= \tfrac12 \sum_{i=2}^k \Norm{ \mf x_i - \mf x_{i-1} }^2 $
	and 
 		$ \eta( \mf t, \mf x ) 
 		= \sup_{ (t,x) \in [\Phi^{-1}(0)]\cap[(0,T]\times (O_N)^k ] } \eta(t,x) $. 
	Therefore, we obtain for all 
		$ i \in \{1, 2, \ldots, k \} $ 
	that 
		$ \mf x_i = \mf x_1 $ 
	and 
 		\begin{equation}
  		\mc S 
  		\leq \lim_{\alpha\to\infty} S_{\alpha,N} 
  		= \sum_{i=1}^k v_i( \mf t, \mf x_i ) 
  		= \eta( \mf t, \mf x ) 
  		= \sup_{ t \in [0,T] } \sup_{ y \in O_N } \left[ \sum_{i=1}^k v_i(t,y) \right] \leq \mc S. 
 		\end{equation}
 	Combining this with \eqref{comparison_viscosity:boundary_condition_v} ensures that $ \mf t \in (0,T) $. 
 	This implies that there exists 
 		$ \mf j \in \N $ 
 	such that for all 
 		$ n \in \N \cap [ \mf j, \infty ) $ 
 	we have that
 		$ t^{(\alpha_n)} \in (0,T) $.
	Hence, we obtain that there exists a non-empty set $ \mc N \subseteq \N $ which satisfies that 
 		$ \mc N = \{ \alpha_n \colon n\in\N, t^{(\alpha_n)} \in (0,T) \} $. 
	Note that \cref{lem:jensen_ishii_lemma} (applied with 	
		$ \varepsilon \is \frac1{n} $,
		$ \mc O \is (0,T)\times\mc O^k $, 
		$ \Phi \is n\Phi|_{(0,T)\times\mc O^k}$, 
		$ (u_i)_{i\in\{1,2,\ldots,k\}} \is (v_i|_{(0,T)\times\mc O})_{i\in\{1,2,\ldots,k\}} $,
		$ \mf t \is t^{(n)} $, 
 		$ \mf x \is x^{(n)} $
 	for $ n \in \mc N $
	in the notation of \cref{lem:jensen_ishii_lemma}) guarantees that there exist 
		$ b^{(n)}_1, b^{(n)}_2, \ldots, b^{(n)}_k \in \R $, 
		$ A^{(n)}_1, A^{(n)}_2, \ldots, A^{(n)}_k \in \Sym_d $
 	which satisfy that 
 		\begin{enumerate}[(I)]
 		\item \label{viscosity_domination:suitable_semijets} 
 		we have for all
	 		$ n \in \mc N $, 
 			$ i \in \{1,2,\ldots,k\} $ 
 		that 
 			\begin{equation}
			(b^{(n)}_i, n (\nabla_{x_i}\Phi)(t^{(n)},x^{(n)}),n A^{(n)}_i)\in(\JetClosure^{+} v_i) (t^{(n)},x_i^{(n)}),
			\end{equation}
		\item \label{viscosity_domination:sum_b} 
		we have for all
			$ n \in \mc N $ 
		that 
			$ 
			\sum_{i=1}^k b^{(n)}_i 
  			=
  			(\frac{\partial}{\partial t} \Phi)(t^{(n)},x^{(n)})
			=  0
			$, 
		and 
 		\item \label{viscosity_domination:matrix_inequalities} 
 		we have for all
 			$ n \in \mc N $ 
 		that 
			\begin{equation}
		 	\begin{split}
		  	& 
		  	-\left( n + n \HSnorm{ (\operatorname{Hess}_{x}\Phi)(t^{(n)},x^{(n)}) }\right) \operatorname{Id}_{\R^{kd}} 
		  	\\
		  	& \leq 
		  	n
		  	\begin{pmatrix}
		  	 A^{(n)}_1 & \ldots & 0 \\
		  	 \vdots & \ddots & \vdots \\
		  	 0 & \ldots & A^{(n)}_k
		  	\end{pmatrix}
		  	\leq 
		  	n (\operatorname{Hess}_x \Phi)(t^{(n)},x^{(n)}) 
		  	+ 
		  	n 
		  	[(\operatorname{Hess}_x \Phi)(t^{(n)},x^{(n)})]^2. 
		  	\end{split}
		 	\end{equation}
		\end{enumerate}
	Observe that the fact that for all 
		$ t \in (0,T) $, 
		$ x \in \mc O^k $ 
	we have that 
		$ (\operatorname{Hess}_x \Phi)(t,x) = (\operatorname{Hess}_x \Phi)(0,0) $ 
 	and Item~\eqref{viscosity_domination:matrix_inequalities} show that for all 
 		$ n \in \mc N $ 
 	we have that 
 		\begin{equation}\label{viscosity_domination:matrix_inequality_with_zero}
 		\begin{split}
  		&
  		-\left( 
  		1
  		+ 
  		\HSnorm{ (\operatorname{Hess}_{x}\Phi)(0,0) } \right) \operatorname{Id}_{\R^{kd}}
  		\\
  		& \leq 
  		\begin{pmatrix}
  		 A^{(n)}_1 & \ldots & 0 \\
   		\vdots & \ddots & \vdots \\
  		 0 & \ldots & A^{(n)}_k
  		\end{pmatrix}
 		\leq 
  		(\operatorname{Hess}_x \Phi)(0,0) 
  		+ 
  		[(\operatorname{Hess}_x \Phi)(0,0)]^2.
  		\end{split}
  \end{equation}
	Moreover, note that \cref{lem:equivalent_conditions_for_viscosity_solutions:left_to_right},
	Item~\eqref{viscosity_domination:suitable_semijets}, and 
	\eqref{viscosity_domination:transformed_equation} ensure that for all 
		$ n \in \mathcal{N} $,
		$ i \in \{1,2,\ldots,k\} $ 
 	we have that 
 		\begin{equation}
		b_i^{(n)} 
		+ 
		H_i(
		t^{(n)},
		x_i^{(n)},
		v_i(t^{(n)},x_i^{(n)}),
		n (\nabla_{x_i}\Phi)(t^{(n)},x^{(n)})
		,
		nA^{(n)}_i)
		\geq 0. 
		\end{equation}
	Combining this and Item~\eqref{viscosity_domination:sum_b} with \eqref{viscosity_domination:H_definition} proves that for all 
 		$ n \in \mc N $ 
	we have that 
		\begin{equation}\label{viscosity_domination:first_sum_inequality}
		\sum_{i=1}^k 
  		G_i( t^{(n)}, x_i^{(n)}, v_i(t^{(n)},x_i^{(n)}) + \frac{\delta}{t^{(n)}}, n (\nabla_{x_i} \Phi)(t^{(n)},x^{(n)}), n A^{(n)}_i) 
  		\geq 
  		\frac{k\delta}{[t^{(n)}]^2}. 
 		\end{equation}
	Next let 
		$ ( \mathbf{t}^{(n)},
		\mathbf{x}_i^{(n)},
		\mathbf{r}_i^{(n)},
		\mathbf{A}_{i}^{(n)} ) 
		\in 
		(0,T) \times \cO \times \R \times \Sym_{d} $, $ n \in \N $, $ i \in \{1,2,\ldots,k\} $, 
	satisfy for all 
	 	$ i \in \{1,2,\ldots,k\} $, 
		$ n \in \N $ 
	that 
		\begin{equation} \label{viscosity_domination:definition_sequence}
  		\begin{split}
 		& ( \mathbf{t}^{(n)},
 		\mathbf{x}_i^{(n)},
 		\mathbf{r}_i^{(n)},
 		\mathbf{A}_{i}^{(n)} ) = 
		\begin{cases}
   		\big( 
   		t^{(n)}, 
   		x^{(n)}_i, 
   		v_i(t^{(n)},
   		x^{(n)}_i)
   		+ 
   		\frac{\delta}{t^{(n)}},
   		A^{(n)}_i
   		\big) & 
   		\colon n \in \mc N 
   		\\
   		\big( \mf t, \mf x_i,  
   		\frac{ \mc S }{ k } + \frac{\delta}{ \mf t }, 0 \big) 
   		& \colon \text{else}. 
  		\end{cases}
  		\end{split}
 		\end{equation}
 	Observe that the fact that 
 		$ \limsup_{ \mc N \ni n \to \infty} [ t^{(n)} - \mf t | + \Norm{ x^{(n)} - \mf x } ] 
 		= \limsup_{ n \to \infty} [ | t^{(\alpha_n)} - \mf t | + \Norm{ x^{(\alpha_n)} - \mf x } ] = 0 $
 	implies that 
		\begin{equation}\label{viscosity_domination:t_n_x_n_1_convergence}
		\limsup_{n\to\infty} \left[ | \mathbf{t}^{(n)} - \mf t | + \Norm{ \mathbf{x}_1^{(n)} - \mf x_1 } \right] = 0.  
		\end{equation}
 	Moreover, note that the fact that for all 
 		$ i \in \{1,2,\ldots,k\} $ 
 	we have that 
 		$ \mf x_i = \mf x_1 $,
 	\eqref{viscosity_domination:alpha_phi_goes_to_zero}, and \eqref{viscosity_domination:definition_sequence} ensure that 
 		\begin{equation}\label{viscosity_domination:x_n_convergence}
		\begin{split}
  		0 
  		\leq 
  		\limsup_{n\to\infty} 
  		\left[
  		\sqrt{n} 
  		\sum_{i=2}^k 
   	 	\Norm{ 
    		\mathbf{x}^{(n)}_i
    		- 
    		\mathbf{x}^{(n)}_{i-1} 
    	} \right]
  		& =
  \limsup_{n\to\infty} 
  \left[n\left( \sum_{i=2}^k 
    \Norm{ \mathbf{x}^{(n)}_i - 
  \mathbf{x}^{(n)}_{i-1} }
  \right)^{\!\!2}
  \right]^{\nicefrac12}
  \\
  & \leq 
  \sqrt{k}
  \limsup_{n\to\infty} 
  \left[ 
   n \sum_{i=2}^k \Norm{ 
   \mathbf{x}^{(n)}_i - 
   \mathbf{x}^{(n)}_{i-1} }^2
  \right]^{\nicefrac12}
  = 0.
  \end{split}
 \end{equation}
	In addition, observe that \eqref{viscosity_domination:definition_sequence} and the fact that $ \limsup_{\mc N \ni n\to\infty} | \eta(t^{(n)},x^{(n)}) - \mc S | = 0 $ prove that 
		\begin{equation}
		\label{viscosity_domination:sum_r_i_n_convergent}
		 \liminf_{n\to\infty} 
		 \left[
		 \sum_{i=1}^k \mathbf{r}^{(n)}_i
		 \right]
		 = 
		 \limsup_{n\to\infty} 
		 \left[
		 \sum_{i=1}^k \mathbf{r}^{(n)}_i
		 \right]
		 =
		 \mc S + \frac{ k \delta }{ \mf t }
		 > 0. 
		\end{equation}
 	Furthermore, note that the fact that  $\{(\mathbf{t}^{(n)},\mathbf{x}^{(n)})\in (0,T)\times\cO^k\colon n\in\N\} \cup \{ ( \mf t, \mf x ) \} $ is compact and the fact that for all 
		$ i \in \{1,2,\ldots,k\} $ 
	we have that $ v_i $ is upper semi-continuous guarantee that  
		\begin{equation} \label{viscosity_domination:r_i_n_bounded_above}
		\sup\!\big\{ \mathbf{r}^{(n)}_i \colon n\in\N,  i\in\{1,2,\ldots,k\} \big\} < \infty . 
		\end{equation} 
	Moreover, observe that \eqref{viscosity_domination:sum_r_i_n_convergent} ensures that 
		$ \sup\!\big\{ | \sum_{i=1}^k \mathbf{r}^{(n)}_i | \colon  n \in \N  \big\} < \infty $.
	Combining this with \eqref{viscosity_domination:r_i_n_bounded_above} implies that 
	 	\begin{equation}
		\label{viscosity_domination:sum_of_abs_r_i_n_bounded}
  		\sup_{n\in\N} 
  		\left[ 
  		\sum_{i=1}^k |\mathbf{r}^{(n)}_i|
  		\right] < \infty. 
 		\end{equation}
 	Next note that \eqref{viscosity_domination:Phi_definition} ensures that $ \HSnorm{(\operatorname{Hess}_x \Phi)(0,0)} \leq 4 $ (cf., for instance, (4.41) in Hairer et al.~\cite{HaHuJe2017_LossOfRegularityKolmogorov}). 
 	This, \eqref{viscosity_domination:definition_sequence}, and \eqref{viscosity_domination:matrix_inequality_with_zero} imply that for all 
 		$ n \in \N $, 
 		$ z_1, z_2, \ldots, z_k \in \R^d $ 
	we have that 
 		\begin{equation}\label{viscosity_domination:bounds_for_A_i_n}
		 - 5 \left[ \sum_{i=1}^k 
    	\Norm{z_i}^2 \right]
 		\leq 
		\sum_{i=1}^k 
		\langle z_i, \mathbf{A}^{(n)}_i z_i \rangle
		\leq 
		5 \left[ \sum_{i=2}^k \Norm{z_i - z_{i-1}}^2 \right]\!. 
		\end{equation}
	In addition, observe that \eqref{viscosity_domination:Phi_definition} guarantees that for all 
		$ t \in (0,T) $, 
		$ x=(x_1,x_2,\ldots,x_k) \in \mc O^k $ 
	we have that 
		\begin{equation}
	 	\label{viscosity_domination:Phi_Gradient}
	 	\begin{split}
	  	(\nabla_{x_i}\Phi)(t,x) 
	  	& = 
	  	\begin{cases} 
	  	 x_1 - x_2 & \colon 1 = i < k \\
	  	 2x_i - x_{i-1} - x_{i+1} & \colon 1 < i < k \\
	  	 x_k - x_{k-1} & \colon 1 < i = k \\
	  	 0 & \colon 1 = i = k
	  	\end{cases}
	  	\\
	  	& =
	  	\mathbbm{1}_{[2,k]}(i)
	    [
	    x_i-x_{i-1}
	    ]
	    + 
	    \mathbbm{1}_{[1,k-1]}(i)
	    [
	    x_i-x_{i+1}
	    ]. 
	  	\end{split}
		\end{equation}
	Combining \eqref{viscosity_domination:limsup_ass} and 
 	\eqref{viscosity_domination:first_sum_inequality} with  \eqref{viscosity_domination:t_n_x_n_1_convergence}--\eqref{viscosity_domination:bounds_for_A_i_n} 
 	hence ensures that 
		\begin{multline}
  		0 
  		< \frac{ k \delta }{ \mf t^2 }
  		= \limsup_{n\to\infty} \frac{ k \delta }{ [ t^{(n)} ]^2 }
  		\leq 
  		\limsup_{n\to\infty} 
  		\Bigg[ 
  		 \sum_{i=1}^k 
  		 G_i\bigg( 
    		\mathbf{t}^{(n)},
    		\mathbf{x}_i^{(n)},
    		\mathbf{r}_i^{(n)},
    		n
    		\big( 
    		\mathbbm{1}_{[2,k]}(i)
    		[
    		\mathbf{x}_i^{(n)}-
    		\mathbf{x}_{i-1}^{(n)}
    		]
    		\\ 
    		+ 
    		\mathbbm{1}_{[1,k-1]}(i)
    		[
    		\mathbf{x}_i^{(n)}-
    		\mathbf{x}_{i+1}^{(n)}
   			 ]
   		 \big), 
   	 	n \mathbf{A}_{i}^{(n)} 
   		\bigg)
   		\Bigg] 
   		\leq 0. 
  		\end{multline}	
 	This contradiction implies that $ \mc S \leq 0 $. 
 	Therefore, we obtain that for all 
 		$ t \in (0,T] $, 
 		$ y \in \mc O $ 
 	we have that 
 		$ \sum_{i=1}^k u_i( t, y ) \leq \frac{ k \delta }{ t } $. 
	This completes the proof of \cref{lem:viscosity_domination}. 
\end{proof}

\begin{cor} \label{cor:comparison_viscosity}
	Let $ d \in \N $, 
		$ T \in (0,\infty) $, 
	let $ \langle\cdot,\cdot\rangle \colon \R^d\times\R^d \to \R $ be the standard Euclidean scalar product on $\R^d$, 
	let $ \norm{\cdot} \colon \R^d \to [0,\infty) $ be the standard Euclidean norm on $\R^d$, 
	let $ \mc O \subseteq \R^d $ be a non-empty open set, 
	for every 
		$ r \in (0,\infty) $ 
	let $ O_r \subseteq \mc O $ satisfy $ O_r = \{x\in\mc O \colon (\norm{x}\leq r~\text{and}~\{y\in\R^d\colon\norm{y-x}<\nicefrac{1}{r}\}\subseteq\mc O)\} $,
	let $ G \in C( (0,T) \times \mc O \times \R \times \R^d \times \Sym_d, \R ) $, 
		$ u,v \in C([0,T]\times\cO,\R) $, 
	assume that 
	    \begin{equation} 
     	\sup_{ x \in \mc O } 
     	\left( u(T,x) - v(T,x) \right) 
     	\leq 0 
	 	\qquad\text{and}\qquad 
	 	\inf_{ r \in ( 0, \infty ) } 
	 	\left[ 
	 	\sup_{ t \in (0,T)}
	 	\sup_{ x \in \mc O \setminus O_r } ( u(t,x) - v(t,x) )
	 	\right]
		\leq 0,
		\end{equation}
 	assume that $ G $ is degenerate elliptic, 
	assume that $ u $ is a viscosity solution of 
		\begin{equation}
	    \label{uniqueness_viscosity}
	     (\tfrac{\partial }{\partial t}u)(t,x) 
	     + 
	     G(t,x,u(t,x),(\nabla_x u)(t,x),(\operatorname{Hess}_x u)(t,x)) \geq 0
	    \end{equation}
    for $ (t,x)\in (0,T)\times\mc O $,  
	assume that $ v $ is a viscosity solution of 
    	\begin{equation}
    	(\tfrac{\partial }{\partial t}v)(t,x) 
     	+ 
     	G(t,x,v(t,x),(\nabla_x v)(t,x),(\operatorname{Hess}_x v)(t,x)) \leq 0
    \end{equation}
	for $(t,x)\in (0,T)\times\cO$, 
	and assume for all 
    	$ t_n \in (0,T) $, $ n \in \N_0 $, 
    all
    	$ ( x_n, r_n, A_n ) \in \mc O \times \R \times \Sym_d $, $ n \in \N_0 $,  
    and all
    	$ (\mf x_n, \mf r_n, \mf A_n) \in \mc O \times \R \times \Sym_d $, $ n \in \N_0 $, 
    with   
    	$ \limsup_{n\to\infty} [ | t_n - t_0 | + \norm{ x_n - x_0 } + \sqrt{n}  \norm{ x_n - \mf x_n } ] = 0 < \liminf_{ n \to \infty } ( r_n - \mf r_n ) = \limsup_{ n \to \infty } ( r_n - \mf r_n ) \leq \sup_{n\in\N} (|r_n| + |\mf r_n|) < \infty $ 
    and 
    	$ \Forall ( n \in \N, z, \mf z \in \R^d ) \colon \langle z,A_n z \rangle - \langle \mf z, \mf A_n \mf z \rangle \leq 5 \norm{ z - \mf z }^2 $
    that 
	    \begin{equation} \label{comparison_viscosity:ass}
	     \limsup_{n\to\infty} 
	     \left[ G( t_n, x_n, r_n, n( x_n - \mf x_n ), n A_n ) - G( t_n, \mf x_n, \mf r_n, n( x_n - \mf x_n ), n \mf A_n ) \right]  
	     \leq 0  
	    \end{equation}
	(cf.~\cref{symmetric_matrices,degenerate_elliptic,def:viscosity_subsolution,def:viscosity_supersolution}). 
	Then we have for all 
 		$ t \in [0,T] $, 
 		$ x \in \mc O $ 
 	that $ u(t,x) \leq v(t,x) $.     
\end{cor}

\begin{proof}[Proof of \cref{cor:comparison_viscosity}]
	Throughout this proof let 
 		$ H\colon (0,T)\times \mc O\times\R\times\R^d\times\Sym_{d} \to \R $ 
	satisfy for all 
 		$ t \in (0,T) $, 
 		$ x \in \mc O $, 
 		$ r \in \R $, 
 		$ p \in \R^d $, 
 		$ A \in \Sym_d $ 
 	that 
		\begin{equation} \label{comparison_viscosity:definition_of_H}
	  	H(t,x,r,p,A) = -G(t,x,-r,-p,-A).
	 	\end{equation}
	Note that \eqref{comparison_viscosity:definition_of_H} ensures that $H$ is degenerate elliptic. 
	Moreover, observe that \eqref{comparison_viscosity:definition_of_H} and the assumption that $ G \in C( (0,T)\times\mc O\times\R\times\R^d\times\Sym_d, \R ) $ assure that $ H \in C( (0,T)\times\mc O\times\R\times\R^d\times\Sym_d, \R ) $.
	In addition, note that \eqref{comparison_viscosity:definition_of_H} implies that $ -v $ is a viscosity solution of 
		\begin{equation}
		(\tfrac{\partial }{\partial t}(-v))(t,x) 
		+
		H\big(t,x,(-v)(t,x),(\nabla_x (-v))(t,x),(\operatorname{Hess}_x (-v))(t,x)\big)
		\geq 0	
		\end{equation}
	for $(t,x) \in (0,T) \times \cO$. 
 	Moreover, observe that \eqref{comparison_viscosity:ass} guarantees that for all 
 		$ t_n \in (0,T) $, $ n \in \N_0 $, 
 	all 
 		$ ( x_n, r_n, A_n ) \in \mc O \times \R \times \Sym_d $, $ n \in \N_0 $, 
 	and all 
 		$ ( \mf x_n, \mf r_n, \mf A_n ) \in \mc O \times \R \times \Sym_d $, $ n \in \N_0 $, 
 	with 
 		$ \limsup_{n\to\infty} [ | t_n - t_0 | + \norm{ x_n - x_0 } + \sqrt{n}\norm{ x_n - \mf x_n } ] = 0 < \liminf_{ n \to \infty } ( r_n + \mf r_n ) = \limsup_{ n \to \infty } ( r_n + \mf r_n ) \leq \sup_{n\in\N} (|r_n| + | \mf r_n | ) < \infty $
	and 
		$ \Forall ( n \in \N, z, \mf z \in \R^d ) \colon -5 ( \norm{ z }^2 + \norm{ \mf z }^2 ) \leq \langle z, A_n z\rangle + \langle \mf z, \mf A_n \mf z \rangle \leq 5 \norm{ z - \mf z }^2 $
	we have that 	
		\begin{equation}
 		\begin{split}
  		& 
  		\limsup_{n\to\infty} 
  		\big[ G( t_n, x_n, r_n, n( x_n - \mf x_n ), n A_n )
   		+ H( t_n, \mf x_n, \mf r_n, n( \mf x_n - x_n ), n \mf A_n ) \big] 
  		\\
  		& =
  		\limsup_{n\to\infty} 
  		\big[ G( t_n, x_n, r_n, n( x_n - \mf x_n ), n A_n )
   		- G( t_n, \mf x_n, -\mf r_n, n( x_n - \mf x_n ), -n \mf A_n ) \big] 
  		\leq 0 . 
 		\end{split}
 		\end{equation} 
	\cref{lem:viscosity_domination} (applied with 
 		$ k \is 2 $, 
 		$ u_1 \is u $, 
 		$ u_2 \is -v $, 
 		$ G_1 \is G $, 
 		$ G_2 \is H $ 
 	in the notation of \cref{lem:viscosity_domination}) hence ensures that for all 
 		$ t \in (0,T] $, 
 		$ x \in \mc O $ 
 	we have that 
 		$ u(t,x) - v(t,x) \leq 0 $. 
	The assumption that $ u, v \in C( [0,T]\times\mc O, \R ) $ therefore ensures that for all 
		$ t \in [0,T] $, 
		$ x \in \mc O $ 
	we have that 
 		$ u(t,x) \leq v(t,x) $. 
	This completes the proof of \cref{cor:comparison_viscosity}. 
\end{proof}

\begin{prop}
\label{prop:uniqueness_viscosity_semilinear}
	Let $ d, m \in \N $, 
    	$ L, T \in (0,\infty) $, 
 	let $ \langle\cdot,\cdot\rangle\colon \R^d\times\R^d \to \R $ be the standard Euclidean scalar product on $ \R^d $, 
 	let $ \norm{\cdot} \colon \R^d\to [0,\infty) $ be the standard Euclidean norm on $ \R^d $, 
 	let $ \HSnorm{\cdot} \colon \R^{d\times m} \to [0,\infty) $ be the Frobenius norm on $ \R^{d\times m} $, 
 	let $ \mc O \subseteq \R^d $ be a non-empty open set,
 	for every
    	$ r \in (0,\infty) $ 
    let $ O_r \subseteq \mc O $ satisfy 
    	$ O_r = \{ x \in \mc O \colon ( \norm{x}\leq r~\text{and}~\{y\in\R^d\colon \norm{y-x} < \nicefrac{1}{r}\} \subseteq \mc O ) \} $, 
    let $ \mu \in C([0,T]\times\mc O,\R^d) $, 
		$ \sigma \in C([0,T]\times\mc O,\R^{d\times m}) $, 
		$ f \in C([0,T]\times\mc O\times\R,\R) $, 
		$ g \in C(\mc O,\R) $, 
    	$ V \in C^{1,2}([0,T]\times\mc O,(0,\infty)) $, 
    assume for all 
		$ r \in (0,\infty) $ 
	that 
    	\begin{equation} \label{unique_viscosity_solution:local_Lipschitz_type_assumption}
    	\sup
     	\!
     	\left(
     	\left\{
     	\frac{    
     	\norm{ \mu(t,x) - \mu(t,y) } 
     	+ 
     	\HSnorm{ \sigma(t,x) - \sigma(t,y) }
    	}{
    	\norm{ x - y }}\colon 
		t \in [0,T],
		x,y \in O_r, 
		x \neq y
		\right\}
     	\cup
     	\{0\}
     	\right)
     	< \infty, 
    	\end{equation}
	assume for all 
    	$ t \in [0,T] $, 
    	$ x \in \mc O $, 
    	$ v,w \in \R $ 
    that 
    	$ ( f( t, x, v ) - f( t, x, w ) )( v - w ) \leq L | v - w |^2 $, 
		$ \limsup_{r\to\infty} [ \sup_{s\in [0,T]} \sup_{y\in \mc O\setminus O_r} ( \frac{|f(s,y,0)|}{V(s,y)} ) ] = 0 $, 
	and 
	    \begin{equation} \label{unique_viscosity_solution:inequality_for_V}
	    (\tfrac{\partial}{\partial t}V)(t,x) 
	    + 
	    \tfrac12
	    \operatorname{Trace}\!\big( 
	        \sigma(t,x)[\sigma(t,x)]^{*}
	        (\operatorname{Hess}_x V)(t,x)
	    \big) 
	    + 
	    \langle \mu(t,x), (\nabla_x V)(t,x) 
	    \rangle
	    \leq 0,
	    \end{equation}
	and let 
		$ u_1, u_2 \in \{ {\bf u} \in C([0,T]\times\cO,\R) \colon \limsup_{r\to \infty} 
		[ \sup_{t\in [0,T]} 
		\allowbreak
		\sup_{x\in\cO\setminus O_r} ( \frac{|{\bf u}(t,x)|}{V(t,x)} ) ]  
		= 0 \} $
	satisfy for all $ i \in \{1,2\} $ that $ u_i $ is a viscosity solution of 
		\begin{multline} \label{unique_viscosity_solution:viscosity_solution_ass}
		(\tfrac{\partial}{\partial t} u_i)(t,x) 
		+ 
		\tfrac12 
		\operatorname{Trace}\!\big( 
		\sigma(t,x)[\sigma(t,x)]^{*}
		(\operatorname{Hess}_x u_i)(t,x)
		\big) 
		+ 
		\langle \mu(t,x), 
		(\nabla_x u_i)(t,x) 
		\rangle
		\\
		+ 
		f(t,x,u_i(t,x)) 
		= 
		0	    
		\end{multline}
	with $u_i(T,x) = g(x)$ for $(t,x) \in (0,T)\times\cO$ (cf.~\cref{def:viscosity_solution}). Then we have for all 
		$ t \in [0,T] $, 
		$ x \in \mc O $ 
	that 
		$ u_1(t,x) = u_2(t,x) $. 
\end{prop}

\begin{proof}[Proof of \cref{prop:uniqueness_viscosity_semilinear}]
	Throughout this proof let 
		$ \llbracket \cdot \rrbracket \colon \big( \bigcup_{ a,b=1 }^{\infty} \R^{a\times b} \big) \to [ 0,\infty ) $
	satisfy for all 
		$ a,b \in \N $, 
		$ A = (A_{i,j})_{(i,j)\in\{1,2,\ldots,a\}\times\{1,2,\ldots,b\}} \in \R^{a\times b} $ 
	that 
		\begin{equation} \label{prop:uniqueness_viscosity_semilinear:definition_frobenius_norms} 
		\llbracket A \rrbracket = \left[ \sum_{i=1}^a\sum_{j=1}^b |A_{i,j}|^2 \right]^{\nicefrac12},
		\end{equation}  	
	let 
		$ \V \colon [0,T]\times\mc O \to (0,\infty) $
	satisfy for all 
		$ t \in [0,T] $,
		$ x \in \mc O $ 
	that 
		$ \V(t,x) = e^{-Lt} V(t,x) $,  
	let $ v_i \colon [0,T]\times \mc O \to\R $, $ i \in \{1,2\} $, 
	satisfy for all 
 		$ i \in \{1,2\} $,
 		$ t \in [0,T] $, 
 		$ x \in \mc O $ 
 	that 
 		$ v_i(t,x) = \frac{u_i(t,x)}{\mathbb{V}(t,x)} $, 
 	let 
 		$ G \colon (0,T) \times \mc O \times \R \times \R^d \times \Sym_{d} \to \R $ 
 	satisfy for all 
 		$ t \in (0,T) $, 
 		$ x \in \mc O $, 
 		$ r \in \R $, 
 		$ p \in \R^d $, 
 		$ A \in \Sym_d $ 
 	that 
 		\begin{equation} \label{uniqueness_viscosity_semilinear:definition_of_G}
  		G(t,x,r,p,A) 
  		= 
  		\tfrac12\operatorname{Trace}\!\big( 
   		\sigma(t,x)[\sigma(t,x)]^{*} A
  		\big) 
  		+ 
  		\langle \mu(t,x), p \rangle
  		+ 
  		f(t,x,r), 
 		\end{equation}
	and let 
 		$ H\colon (0,T)\times \mc O \times \R \times \R^d\times \Sym_{d} \to \R $
	satisfy for all 
 		$ t \in (0,T) $, 
 		$ x \in \mc O $, 
 		$ r \in \R $, 
 		$ p \in \R^d $, 
 		$ A \in \Sym_{d} $
 	that 
		\begin{multline} \label{unique_viscosity_solution:definition_of_H}
		H(t,x,r,p,A) 
		= 
		\tfrac{r}{\mathbb{V}(t,x)}
		(\tfrac{\partial}{\partial t}\V)(t,x)
		+ 
		\tfrac{1}{\mathbb{V}(t,x)}
		G\big( 
		t, 
		x, 
		r\mathbb{V}(t,x), 
		\mathbb{V}(t,x) p 
		+ 
		r(\nabla_x \mathbb{V})(t,x) , \\
		\mathbb{V}(t,x) A 
		+ 
		p [(\nabla_x \mathbb{V})(t,x)]^{*} 
		+ 
		(\nabla_x \mathbb{V})(t,x)p^{*} 
		+ 
		r(\operatorname{Hess}_x \mathbb{V})(t,x)
		\big).   
		\end{multline}
	Observe that \eqref{prop:uniqueness_viscosity_semilinear:definition_frobenius_norms} proves that for all 
		$ A \in \R^{d\times m} $ 
	we have that 
		$ \llbracket A \rrbracket = \HSnorm{ A } $. 
	Next note that \eqref{unique_viscosity_solution:inequality_for_V} implies  that for all 
		$ t \in [0,T] $, 
		$ x \in \mc O $ 
	we have that 
		$ \V \in C^{1,2}( [0,T] \times \mc O, (0,\infty) ) $ 
	and 
		\begin{equation}
		\label{unique_viscosity_solution:Z_supersolution}
		(\tfrac{\partial}{\partial t}\V)(t,x) 
		+ 
		\tfrac12 
		\operatorname{Trace}\!\big( 
		\sigma(t,x)[\sigma(t,x)]^{*}
		(\operatorname{Hess}_x\mathbb{V})(t,x)
		\big) 
		+ 
		\langle 
		\mu(t,x), (\nabla_x\mathbb{V})(t,x) 
		\rangle
		+ 
		L \mathbb{V}(t,x) 
		\leq 0 . 
		\end{equation}
	Moreover, observe that \eqref{uniqueness_viscosity_semilinear:definition_of_G} ensures that $ G \in C( (0,T) \times \mc O \times \R \times \R^d \times \Sym_d, \R ) $ is degenerate elliptic. 
	In addition, note that \eqref{unique_viscosity_solution:definition_of_H} proves that  $ H \in C( (0,T) \times \mc O \times \R \times \R^d \times \Sym_d, \R ) $ is degenerate elliptic. 
 	Next observe that the assumption that for all 
 		$ i \in \{1,2\} $, 
 		$ x \in \mc O $  
 	we have that 
 		$ u_i(T,x) = g(x) $
 	shows that for all  
	 	$ x \in \mc O $ 
	we have that 
 		\begin{equation} \label{unique_viscosity_solution:initial_inequality}
 		v_1(T,x) \leq v_2(T,x) \leq v_1(T,x).
 		\end{equation}
 	In addition, note that the assumption that $ \limsup_{r\to\infty} [ \sup_{t \in [0,T]} \sup_{x\in\mc O \setminus O_r} (\frac{|u_1(t,x)|+|u_2(t,x)|}{V(t,x)} ) ] = 0 $ implies that 
		\begin{equation} \label{unique_viscosity_solution:boundary_condition}
 		\begin{split}
  		& \limsup_{r\to\infty} 
  		\left[ 
   		\sup_{ t \in [0,T] } 
   		\sup_{ x \in \mc O \setminus O_r } 
			| v_1(t,x) - v_2(t,x) |
 	 	\right]
 		=
  		0.  
 		\end{split}
 		\end{equation}
 	Furthermore, observe that \eqref{unique_viscosity_solution:viscosity_solution_ass} and \eqref{unique_viscosity_solution:definition_of_H} ensure that for all 
 		$ i \in \{ 1, 2 \} $ 
	we have that $v_i$ is a viscosity solution of 
		\begin{equation}
		(\tfrac{\partial}{\partial t} v_i)(t,x) 
		+ 
		H\big(t,x,v_i(t,x),(\nabla_x v_i)(t,x), (\operatorname{Hess}_x v_i)(t,x)\big) = 0
		\end{equation}
	for $ (t,x) \in (0,T) \times \mc O $ (cf., for example, Hairer et al.~\cite[Lemma 4.12]{HaHuJe2017_LossOfRegularityKolmogorov}).
	We intend to prove that $ u_1 = u_2 $ by applying \cref{cor:comparison_viscosity} to obtain that $v_1 \leq v_2$ and $v_2\leq v_1$. 
	Next let $ e_1, e_2, \ldots, e_m \in \R^m $
	satisfy 
		$ e_1 = (1,0,\ldots,0) $, $e_2 = (0,1,0,\ldots,0) $, $\ldots$, $ e_m = (0,\ldots,0,1) $, 
	let 
		$ t_n \in (0,T) $, $ n \in \N_0 $, 
	satisfy $ \limsup_{ n \to \infty } |t_n-t_0| = 0 $,
	and let 
		$ ( x_n, r_n, A_n ) \in \mc O \times \R \times \Sym_d $, $ n \in \N_0 $, 
	and $ ( \mf x_n, \mf r_n, \mf A_n ) \in \mc O \times \R \times \Sym_d $, $ n \in \N_0 $, 
	satisfy  
		$ \limsup_{n\to\infty} [ | t_n - t_0 | + \Norm{ x_n - x_0 } + \sqrt{n} \Norm{ x_n - \mf x_n } ] = 0 < r_0 = \liminf_{n\to\infty} ( r_n - \mf r_n ) = \limsup_{n\to\infty} ( r_n - \mf r_n ) \leq \sup_{n\in\N} ( | r_n | + | \mf r_n | ) < \infty $
 	and 
 		$ \Forall ( n \in \N, z, \mf z \in \R^d ) \colon 
 		\langle z, A_n z \rangle - \langle \mf z, \mf A_n \mf z \rangle \leq 5 \norm{ z - \mf z }^2 $.  
	Observe that \eqref{unique_viscosity_solution:local_Lipschitz_type_assumption} and the fact that $ \limsup_{n\to\infty} [|t_n-t_0| + \norm{x_n-x_0} + \sqrt{n} \norm{ x_n - \mf x_n } ] = 0 $ ensure that 
 		\begin{equation} 
 		\limsup_{n\to\infty} \Big[ n \HSnorm{ \sigma(t_n,x_n) - \sigma(t_n,\mf x_n) }^2 \Big] = 0. 
 		\end{equation}
 	This, the fact that for all 
		$ B \in \Sym_d $, 
		$ C \in \R^{d\times m} $
	we have that 
		$ \operatorname{Trace}\!\left( CC^{\ast}B \right) = \sum_{i=1}^m \langle Ce_i, BCe_i \rangle $,  
	and the assumption that for all 
		$ n \in \N $, 
		$ z,\mf z \in \R^d $ 
	we have that 
		$ \langle z, A_n z \rangle - \langle \mf z, \mf A_n \mf z \rangle \leq 5 \norm{ z - \mf z }^2 $
	prove that 
		\begin{equation} \label{unique_viscosity_solution:trace_term:1}
		\begin{split}
		& 
   		\limsup_{n\to\infty} 
  		\left[
   		\tfrac12 
   		\operatorname{Trace}\!\left( 
   		\tfrac{\sigma(t_n,x_n)[\sigma(t_n,x_n)]^{*}}{\mathbb{V}(t_n,x_n)} \mathbb{V}(t_n, x_n) n A_n 
    	- \tfrac{\sigma(t_n,\mf x_n)[\sigma(t_n,\mf x_n)]^{*}}{\mathbb{V}(t_n,\mf x_n)}\mathbb{V}(t_n,\mf x_n) n \mf A_n    \right)
	    \right]
  		\\[1ex]
  		& 
  		= 
  		\limsup_{n\to\infty}
  		\left[ 
  		\tfrac{n}{2}
  		\operatorname{Trace}\!\left( 
  		\sigma(t_n,x_n)[\sigma(t_n,x_n)]^{*}A_n 
  		-
  		\sigma(t_n,\mf x_n)
  		[\sigma(t_n,\mf x_n)]^{*}\mf A_n
  		\right) 
  		\right] 
  		\\
		& 
  		=
  		\limsup_{n\to\infty} 
  		\bigg[
    	\tfrac{n}{2} 
  		\sum_{i=1}^m
    	\left(\langle 
    	\sigma(t_n,x_n)e_i
    	, 
    	A_n 
    	\sigma(t_n,x_n)e_i
    	\rangle
    	- 
    	\langle
    	\sigma(t_n,\mf x_n)e_i, 
    	\mf A_n 
    	\sigma(t_n,\mf x_n)e_i\rangle
  		\right) 
  		\bigg]
  		\\
  		& 
		\leq 
		\limsup_{n\to\infty} 
		\left[ 
		\sum_{i=1}^m 
		  \tfrac52 n \norm{
	 	\sigma(t_n,x_n)e_i
	 	-
	 	\sigma(t_n,\mf x_n)e_i 
		}^2 \right]
		= 
		\tfrac52 
		\limsup_{n\to\infty} \left[ 
		n \HSnorm{
		\sigma(t_n,x_n)
		- 
	  	\sigma(t_n,\mf x_n) 
	  	}^2 \right] 
	  	\\
	  	& = 0 . 
		\end{split}  
		\end{equation}
	Next observe that \eqref{unique_viscosity_solution:local_Lipschitz_type_assumption} and the fact that $ \V \in C^{1,2}([0,T]\times\mc O, (0,\infty) ) $ ensure that for all compact $ \mc K \subseteq \mc O $ we have that there exists 
		$ \mf c \in \R $
	such that for all 
		$ s \in [0,T] $, 
		$ y_1,y_2 \in \mc K $ 
	we have that 
		\begin{multline} 
		\left\llbracket \frac{\sigma(s,y_1)\left[\sigma(s,y_1)\right]^{*}}{\mathbb{V}(s,y_1)} - \frac{\sigma(s,y_2)\left[\sigma(s,y_2)\right]^{*}}{\mathbb{V}(s,y_2)} 
		\right\rrbracket
		+ 
		\Norm{ (\nabla_x \V)(s,y_1) - (\nabla_x \V)(s,y_2) }
		\leq \mf c \Norm{ y_1 - y_2 }.   
		\end{multline} 
	The fact that 
 		$ \limsup_{n\to\infty} [|t_n-t_0| + \norm{x_n-x_0}] = 0 $ 
	and the fact that 
		$ \limsup_{n\to\infty} [\sqrt{n}\norm{x_n-\mf x_n}] = 0 $
	hence guarantee that 
		\begin{equation}
 		\begin{split}
 		& \limsup_{n\to\infty} \left[n\norm{x_n-\mf x_n}
 		\left\llbracket \tfrac{\sigma(t_n,x_n)[\sigma(t_n,x_n)]^{*}}{\mathbb{V}(t_n,x_n)}
 		-\tfrac{\sigma(t_n,\mf x_n)[\sigma(t_n,\mf x_n)]^{*}}{\mathbb{V}(t_n,\mf x_n)} \right\rrbracket
 		\right]
 		\\[1ex]
 		&   
 		= 0 =
 		\limsup_{n\to\infty} \left[n\norm{x_n-\mf x_n}
  		\norm{(\nabla_x \mathbb{V})(t_n,x_n)-
 		(\nabla_x \mathbb{V})(t_n,\mf x_n)}
 		\right] . 
 		\end{split}
 		\end{equation} 
	Combining this with the fact that for all 
	 	$ B \in \Sym_d $,
	 	$ v,w \in \R^d $ 
	we have that 
		$ \operatorname{Trace} (Bvw^{*}) 
		= \operatorname{Trace} (w^{*}Bv) 
		= w^{*}Bv 
		= \langle w, Bv \rangle 
		= \langle Bw,v \rangle 
		= \langle v, Bw \rangle 
		= v^{*}Bw 
		= \operatorname{Trace}(v^{*}Bw) 
		= \operatorname{Trace}(Bwv^{*}) $ 
	yields that
 		\begin{align} \label{unique_viscosity_solution:trace_term:2}
		& \nonumber
   		\limsup_{n\to\infty} 
  		\bigg[
   		\tfrac12 
   		\operatorname{Trace}\!\bigg( 
   		\tfrac{\sigma(t_n,x_n)[\sigma(t_n,x_n)]^{*}}{\mathbb{V}(t_n,x_n)} 
   		\big( n ( x_n - \mf x_n ) [(\nabla_x \mathbb{V})(t_n,x_n)]^{*} + (\nabla_x \mathbb{V})(t_n,x_n) n ( x_n - \mf x_n )^{*}  \big) 
		\\ \nonumber
   		& 
  		\quad  
   		-
   		\tfrac{\sigma(t_n,\mf x_n)[\sigma(t_n,\mf x_n)]^{*}}{\mathbb{V}(t_n,\mf x_n)}
   		\big( 
   		n ( x_n - \mf x_n ) 
   		[(\nabla_x \mathbb{V})(t_n,\mf x_n)]^{*} 
   		+ 
   		(\nabla_x \mathbb{V})(t_n, \mf x_n)
   		n ( x_n - \mf x_n )^{*}
   		\big) 
   		\bigg)
	    \bigg]
  		\\ \nonumber
  		&  
  		=
  		\limsup_{n\to\infty}
  		\bigg[
   		\left\langle
    	\tfrac{\sigma(t_n,x_n)[\sigma(t_n,x_n)]^{*}}{\mathbb{V}(t_n,x_n)} 
   		 n ( x_n - \mf x_n ),
   		(\nabla_x \mathbb{V})(t_n,x_n)
   		\right\rangle
   		\\ \nonumber   
   		& \quad 
   		-
   		\left\langle
   		\tfrac{\sigma(t_n,\mf x_n)[\sigma(t_n,\mf x_n)]^{*}}{\mathbb{V}(t_n,\mf x_n)} 
    	n ( x_n - \mf x_n ),  
   		(\nabla_x \mathbb{V})(t_n,\mf x_n)
		\right\rangle
  		\bigg]
  		\\
  		&  
  		= 
  		\limsup_{n\to\infty} 
  		\bigg[ 
  		 \left\langle
  		  \left( 
  			  \tfrac{\sigma(t_n,x_n)[\sigma(t_n,x_n)]^{*}}{\mathbb{V}(t_n,x_n)} 
  			 -
  		  \tfrac{\sigma(t_n,\mf x_n)[\sigma(t_n,\mf x_n)]^{*}}{
  		  \mathbb{V}(t_n,\mf x_n)}
  		 \right)
  		 n ( x_n - \mf x_n ),
  		 (\nabla_x \mathbb{V})(t_n,x_n)
  		 \right\rangle
  		 \\ \nonumber
  		 & \quad 
  	 	+ 
  	 	\left\langle
   		\tfrac{\sigma(t_n,\mf x_n)[\sigma(t_n,\mf x_n)]^{*}}{\mathbb{V}(t_n,\mf x_n)} n (x_n - \mf x_n),
   		(\nabla_x \mathbb{V})(t_n,x_n)
   		 - 
   		(\nabla_x \mathbb{V})(t_n,\mf x_n)
   		\right\rangle
  		\bigg] 
  		\\ \nonumber
  		& 
  		\leq 
  		\limsup_{n\to\infty}
  		\bigg[ 
   		\left\llbracket
   		\tfrac{\sigma(t_n,x_n)[\sigma(t_n,x_n)]^{*}}
   		{\mathbb{V}(t_n,x_n)} 
   		- 
   		\tfrac{\sigma(t_n,\mf x_n)[\sigma(t_n,\mf x_n)]^{*}}
   		{\mathbb{V}(t_n,\mf x_n)}
   		\right\rrbracket
		n \norm{ x_n - \mf x_n }
		\norm{ (\nabla_x \mathbb{V})(t_n,x_n) } 
   		\\ \nonumber
   		& \quad
   		+ 
		\left\llbracket 
		\tfrac{\sigma(t_n,\mf x_n)[\sigma(t_n,\mf x_n)]^{*}}{\mathbb{V}(t_n,\mf x_n)} 
   		\right\rrbracket
		n \norm{ x_n - \mf x_n } 
   		\norm{ 
    	(\nabla_x \mathbb{V})(t_n,x_n) - 
    	(\nabla_x \mathbb{V})(t_n,\mf x_n)
  		} 
  		\bigg]
 		\\ \nonumber
 		& = 0 . 
		\end{align}
	Moreover, note that the fact that 
		$ (0,T) \times \mc O \ni (s,y) \mapsto \tfrac{\sigma(s,y)[\sigma(s,y)]^{*}}{\mathbb{V}(s,y)} (\operatorname{Hess}_x \mathbb{V})(s,y) \in \R^{d\times d}	$ is continuous, 
	the fact that 
		$ \limsup_{n\to\infty} [ \sqrt{n} \norm{x_n-\mf x_n} ] = 0 $, 
	and the fact that 
		$ \limsup_{n\to\infty} [ |t_n-t_0| + \norm{x_n-x_0} ] = 0 $ 
	guarantee that 
		\begin{equation}		
		\begin{split}
 		0 & = 
 		\limsup_{n\to\infty} 
 		\left|
 \operatorname{Trace}\!\left(
  \tfrac{\sigma(t_n,x_n)[\sigma(t_n,x_n)]^{*}}{\mathbb{V}(t_n,x_n)}(\operatorname{Hess}_x \mathbb{V})(t_n,x_n)
  -
    \tfrac{\sigma(t_0,x_0)[\sigma(t_0,x_0)]^{*}}{\mathbb{V}(t_0,x_0)}(\operatorname{Hess}_x \mathbb{V})(t_0,x_0)
 \right)
 \right|  
 \\
 & 
 = 
 \limsup_{n\to\infty} 
 \left| 
 \operatorname{Trace}\!\left(
 \tfrac{\sigma(t_n,\mf x_n)[\sigma(t_n,\mf x_n)]^{*}}{\mathbb{V}(t_n,\mf x_n)}(\operatorname{Hess}_x \mathbb{V})(t_n,\mf x_n)
 -  \tfrac{\sigma(t_0,x_0)[\sigma(t_0,x_0)]^{*}}{\mathbb{V}(t_0,x_0)}(\operatorname{Hess}_x \mathbb{V})(t_0,x_0)
 \right)
 \right|. 
\end{split}
\end{equation}
	This and the fact that 
		$ 0 < r_0 
		= \liminf_{n\to\infty} (r_n-\mf r_n)
		= \limsup_{n\to\infty} (r_n-\mf r_n) 
		\leq \sup_{n\in\N} (|r_n|+|\mf r_n|) < \infty $ 
	ensure that 
 		\begin{align} \label{unique_viscosity_solution:trace_term:3}
		\nonumber 
		& \limsup_{n\to\infty}
  		\left[
   		\tfrac12 
   		\operatorname{Trace}\!\big( 
   		\tfrac{\sigma(t_n,x_n)[\sigma(t_n,x_n)]^{*}}{\mathbb{V}(t_n,x_n)} 
   		r_n (\operatorname{Hess}_x \mathbb{V})(t_n,x_n) 
   		- 
   		\tfrac{\sigma(t_n,\mf x_n)\left[\sigma(t_n,\mf x_n)\right]^{*}}{\mathbb{V}(t_n,\mf x_n)} 
   		\mf r_n (\operatorname{Hess}_x \mathbb{V})(t_n,\mf x_n) 
   		\big) 
  		\right]
  		\\
  		& \,
  		= 
  		\tfrac12
  		\limsup_{n\to\infty} 
  		\bigg[ (r_n - \mf r_n) \operatorname{Trace}\!\left( \tfrac{\sigma(t_n,x_n)[\sigma(t_n,x_n)]^{*}}{\mathbb{V}(t_n,x_n)} 
		( \operatorname{Hess}_x \mathbb{V})(t_n,x_n)
   		\right) 
   		\\ \nonumber 
   		& 
   		\quad + 
   		\mf r_n 
   		\operatorname{Trace}\!\left( 
    	\tfrac{\sigma(t_n,x_n)[\sigma(t_n,x_n)]^{*}}{\mathbb{V}(t_n,x_n)} 
    	( \operatorname{Hess}_x \mathbb{V})(t_n,x_n) 
   		- 
    	\tfrac{\sigma(t_n,\mf x_n)[\sigma(t_n,\mf x_n)]^{*}}{\mathbb{V}(t_n,\mf x_n)} 
   		( \operatorname{Hess}_x \mathbb{V})(t_n,\mf x_n)
   		\right) 
  		\bigg] 
  		\\ \nonumber
  		& \, 
  		=
  		\tfrac{r_0}{2 \mathbb{V}(t_0,x_0)}
  		\operatorname{Trace}\!\left( 
   		\sigma(t_0,x_0)
   		[\sigma(t_0,x_0)]^{*} 
   		(\operatorname{Hess}_x \mathbb{V})(t_0,x_0)
  		\right) . 
  		\end{align}
 	Combining this with \eqref{unique_viscosity_solution:trace_term:1} and \eqref{unique_viscosity_solution:trace_term:2} demonstrates that
 		\begin{equation} \label{unique_viscosity_solution:trace_term}
 		\begin{split}
 		& \limsup_{n\to\infty} 
 		\bigg[ \tfrac12 \operatorname{Trace}\!\bigg( \tfrac{\sigma(t_n,x_n)[\sigma(t_n,x_n)]^{*}}{\mathbb{V}(t_n,x_n)} \big( 
		\mathbb{V}(t_n, x_n) n A_n + n (x_n - \mf x_n) \left[(\nabla_x \mathbb{V})(t_n,x_n)\right]^{*} \\
		& \quad  + 
		(\nabla_x \mathbb{V})(t_n,x_n) n (x_n - \mf x_n)^{*} 
		+ r_n (\operatorname{Hess}_x \mathbb{V})(t_n,x_n)
		\big) \bigg)    
		\\
		& \quad - 
		\tfrac12 \operatorname{Trace}\!\bigg( 
	    \tfrac{\sigma(t_n,\mf x_n)[\sigma(t_n,\mf x_n)]^{*}}{\mathbb{V}(t_n,\mf x_n)}
	    \big( 
	    \mathbb{V}( t_n, \mf x_n ) n \mf A_n 
	    + 
	    n (x_n - \mf x_n) 
	    \left[(\nabla_x \mathbb{V})(t_n,\mf x_n)\right]^{*}
	    \\
	    & \quad + 
		(\nabla_x \mathbb{V})(t_n,\mf x_n) n (x_n - \mf x_n)^{*} 
		+ 
		\mf r_n (\operatorname{Hess}_x \mathbb{V})(t_n,\mf x_n) \big) \bigg) \bigg] 
		\\
		& 
		\leq 
		\tfrac{r_0}{2 \mathbb{V}(t_0, x_0)}
		  \operatorname{Trace}\!\big( 
		   \sigma(t_0,x_0)
		   [\sigma(t_0,x_0)]^{*} 
		   (\operatorname{Hess}_x \mathbb{V})(t_0,x_0)
		  \big). 
	  	\end{split}
		\end{equation}
	Moreover, observe that the fact that 
 		$ (0,T) \times \mc O \ni (s,y) \mapsto 	\frac{1}{\mathbb{V}(s,y)}(\frac{\partial }{\partial t}\mathbb{V})(s,y)\in\R $
 	is continuous and the fact that 
 		$ 0 < r_0 = \liminf_{n\to\infty} (r_n-\mf r_n) = \limsup_{n\to\infty} (r_n-\mf r_n) \leq \sup_{n\in\N} (|r_n|+|\mf r_n|) < \infty $
	show that 
		\begin{equation} \label{unique_viscosity_solution:r_dV_term}
		\begin{split}
		& \limsup_{n\to\infty} 
		\left[ 
		\tfrac{r_n}{\mathbb{V}(t_n,x_n)}
		 \left(\tfrac{\partial }{\partial t} \mathbb{V} \right)\!(t_n,x_n) - \tfrac{\mf r_n}{\mathbb{V}(t_n,\mf x_n)} 
		 \left(\tfrac{\partial }{\partial t} \mathbb{V} \right)\!(t_n,\mf x_n) 
		 \right]
 		\\
 		& 
 		= \limsup_{n\to\infty} 
 		\left[ 
 		\tfrac{r_n-\mf r_n}{\mathbb{V}(t_n,x_n)} 
		 \left(\tfrac{\partial}{\partial t} \mathbb{V}\right)\!(t_n,x_n)
 		+ 
 		\mf r_n
  		\left( 
  		\tfrac{\left(\frac{\partial}{\partial t} \mathbb{V}\right)(t_n,x_n)}{\mathbb{V}(t_n,x_n)} 
	  	- 
  		\tfrac{\left(\frac{\partial}{\partial t} \mathbb{V}\right)(t_n,\mf x_n)}{\mathbb{V}(t_n,\mf x_n)} 
  		\right) 
  		\right]
  		= 
  		\tfrac{r_0}{\mathbb{V}(t_0,x_0)} 
  		\left(\tfrac{\partial}{\partial t} \mathbb{V} \right)\!(t_0,x_0) . 
 		\end{split}
		\end{equation}
	Next note that \eqref{unique_viscosity_solution:local_Lipschitz_type_assumption}, 
	the fact that $ \limsup_{n\to\infty} [|t_n-t_0| + \norm{x_n-x_0}] = 0 $, 
 	and the fact that $ \limsup_{n\to\infty} [\sqrt{n}\norm{x_n-\mf x_n}] = 0	$ 
 	imply that
 		\begin{equation} 
  		\limsup_{n\to\infty} 
  		\Big[ n \Norm{\mu(t_n,x_n)-\mu(t_n,\mf x_n)}
   		 \Norm{x_n-\mf x_n} \Big] = 0.
  		\end{equation}  
 	This, the fact that 
 		$ (0,T) \times \mc O \ni (s,y) \mapsto \langle \frac{\mu(s,y)}{\mathbb{V}(s,y)}, (\nabla_x \mathbb{V})(s,y)	\rangle \in\R $	 
	is continuous, and 
	the fact that 
		$ 0 < r_0 
		= \liminf_{n\to\infty} (r_n-\mf r_n)
		= \limsup_{n\to\infty} (r_n-\mf r_n) 
		\leq \sup_{n\in\N} (|r_n|+|\mf r_n|) < \infty $ 
	yield that
		 \begin{equation} \label{unique_viscosity_solution:mu_term}
		 \begin{split}
		  & \limsup_{n\to\infty}
		  \bigg[
		  \tfrac{1}{\mathbb{V}(t_n,x_n)}
		  \left\langle 
		   \mu(t_n,x_n), \mathbb{V}(t_n,x_n)
		   n (x_n - \mf x_n )
		   + 
		   r_n(\nabla_x \mathbb{V})(t_n,x_n) 
		  \right\rangle
		  \\
		  & \quad - 
		  \tfrac{1}{\mathbb{V}(t_n,\mf x_n)}
		  \left\langle 
		   \mu(t_n,\mf x_n), \mathbb{V}(t_n, \mf x_n) n (x_n - \mf x_n ) 
		   + 
		   \mf r_n (\nabla_x \mathbb{V})(t_n, \mf x_n)
		  \right\rangle
		  \bigg] 
		  \\
		  & 
		  = 
		  \limsup_{n\to\infty} 
		  \bigg[
		   \langle 
		    \mu(t_n,x_n) - \mu(t_n,\mf x_n), n(x_n-\mf x_n) 
		   \rangle
		   \\
		   & \quad 
		   + 
		   r_n
		   \left\langle 
		   \tfrac{\mu(t_n,x_n)}{\mathbb{V}(t_n,x_n)},(\nabla_x \mathbb{V})(t_n,x_n)
		   \right\rangle
		   -
		   \mf r_n
		   \left\langle 
		   \tfrac{\mu(t_n,\mf x_n)}{\mathbb{V}(t_n,
		   \mf x_n)}, 
		   (\nabla_x \mathbb{V})(t_n,\mf x_n)
		   \right\rangle
		  \bigg]
		  \\
		  & 
		  \leq \limsup_{n\to\infty} 
		  \big[ 
		   \norm{ 
		    \mu(t_n,x_n)
		    - 
		    \mu(t_n,\mf x_n)
		   } 
		   n \norm{ x_n - \mf x_n }
		  \big] 
		  \\
		  & \quad 
		  + 
		  \limsup_{n\to\infty} 
		  \left[ 
		  (r_n-\mf r_n)
		   \left\langle 
		    \tfrac{\mu(t_n,x_n)}{\mathbb{V}(t_n,x_n)},(\nabla_x \mathbb{V})(t_n,x_n)
		   \right\rangle
		  \right]
		  \\
		  & \quad
		  + 
		  \limsup_{n\to\infty} 
		  \left[
		  \mf r_n
		  \left(
		   \left\langle 
		    \tfrac{\mu(t_n,x_n)}{\mathbb{V}(t_n,x_n)},
		    (\nabla_x \mathbb{V})(t_n,x_n)
		   \right\rangle
		   - 
		   \left\langle 
		    \tfrac{\mu(t_n,\mf x_n)}{\mathbb{V}(t_n,\mf x_n)},
		    (\nabla_x \mathbb{V})(t_n,\mf x_n)
		   \right\rangle
		   \right) 
		  \right]
		  \\[1ex]
		  & 
		  =  
		  \tfrac{r_0}{\mathbb{V}(t_0,x_0)} \langle 
		  \mu(t_0,x_0), 
		  (\nabla_x \mathbb{V})(t_0,x_0)
		  \rangle
		  . 
		  \end{split}
		 \end{equation}
	Furthermore, note that the assumption that $ f \in C( [0,T]\times\mc O \times\R,\R) $ proves that for all compact $ \mc K  \subseteq [0,T] \times \mc O \times \R $ we have that 
		\begin{multline} 
		\limsup_{ (0,\infty) \ni \varepsilon \to 0 } \left[ 
		\sup\!\left(\left\{ | f(s_1,y_1,a_1) - f(s_2,y_2,a_2) | \colon 
		\left( 
		\begin{array}{c} 
		(s_1,y_1,a_1),(s_2,y_2,a_2) \in \mc K, \\
		|a_1-a_2| + |s_1-s_2| \leq \varepsilon, \\
		\norm{y_1-y_2} \leq \varepsilon
		\end{array} \right) \right\} \cup \{0\}\right) \right] 
		\\
		= 0.  
		\end{multline}
	This and the assumption that for all 
		$ s \in [0,T] $, 
		$ y \in \mc O $, 
		$ a,b \in \R $ 
	we have that 
		$ ( f( s, y, a ) - f( s, y, b ) ) ( a - b ) \leq L | a - b |^2 $ 
	imply that 
		\begin{equation}\label{unique_viscosity_solution:f_term}
		\begin{split}
		& \limsup_{n\to\infty}
		\left[
		\tfrac{f(t_n, x_n, r_n \mathbb{V}(t_n,x_n))}{\mathbb{V}(t_n,x_n)} 
   		- 
   		\tfrac{f(t_n, \mf x_n, \mf r_n \mathbb{V}(t_n,\mf x_n))}{\mathbb{V}(t_n,\mf x_n)}  \right] 
   		\\
   		&
   		= \limsup_{n\to\infty} 
  		\left[ 
		\tfrac{f(t_n,x_n,r_n \mathbb{V}(t_n,x_n))}{\mathbb{V}(t_n,x_n)}   
   		- 
   		\tfrac{f(t_n,\mf x_n,r_n \mathbb{V}(t_n,\mf x_n))}{\mathbb{V}(t_n,\mf x_n)}
   		+ 
   		\tfrac{f(t_n,\mf x_n,r_n \mathbb{V}(t_n,\mf x_n))}{\mathbb{V}(t_n,\mf x_n)}
   		- 
   		\tfrac{f(t_n,\mf x_n,\mf r_n \mathbb{V}(t_n,\mf x_n))}{\mathbb{V}(t_n,\mf x_n)}
  		\right]  
  		\\
  		& 
  		\leq 
  		\limsup_{n\to\infty} \left[ \tfrac{f(t_n,x_n,r_n \mathbb{V}(t_n,x_n))}{\mathbb{V}(t_n,x_n)}   
   		- 
   		\tfrac{f(t_n,\mf x_n,r_n \mathbb{V}(t_n,\mf x_n))}{\mathbb{V}(t_n,\mf x_n)} \right]
  		+ 
  		\limsup_{n\to\infty}
  		\left[ \tfrac{L (r_n \mathbb{V}(t_n,\mf x_n)-\mf r_n \mathbb{V}(t_n,\mf x_n) )}{\mathbb{V}(t_n,\mf x_n)} \right] 
  		\\[1ex]
  		& 
  		= \limsup_{n\to\infty} \left[ L ( r_n - \mf r_n ) \right]
  		= L r_0. 
 		\end{split}
 		\end{equation}
 	Combining 
 		\eqref{unique_viscosity_solution:definition_of_H},
 	 	\eqref{unique_viscosity_solution:Z_supersolution},
 	 	\eqref{unique_viscosity_solution:trace_term}, 
 	 	\eqref{unique_viscosity_solution:r_dV_term}, 
	and 
 	 	\eqref{unique_viscosity_solution:mu_term} 
 	hence demonstrates that 
 		\begin{equation} 
 		\begin{split}
		& \limsup_{n\to\infty} 
  		\left[
   		H( t_n, x_n, r_n, n(x_n-\mf x_n) , nA_n ) 
   		- 
   		H( t_n, \mf x_n, \mf r_n, n ( x_n - \mf x_n ), n\mf A_n ) 
  		\right] 
  		\\
  		& \leq 
  		\tfrac{r_0}{\mathbb{V}(t_0,x_0)}
  		\bigg[
   		(\tfrac{\partial }{\partial t}\V)(t_0,x_0) 
  	 	+ 
   		\tfrac12
  		\operatorname{Trace}\!\left( 
    	\sigma(t_0,x_0)[\sigma(t_0,x_0)]^{*}(\operatorname{Hess}_x \mathbb{V})(t_0,x_0)
   		\right) 
   		\\
   		& \quad 
   		+ 
   		\langle 
   		\mu(t_0,x_0), 
   		(\nabla_x \mathbb{V})(t_0,x_0)
   		\rangle
		+ 
   		L \mathbb{V}(t_0,x_0)
  		\bigg]
  		\leq 
  		0 .  
  		\end{split}
 		\end{equation}
	This, \eqref{unique_viscosity_solution:initial_inequality}, \eqref{unique_viscosity_solution:boundary_condition}, and \cref{cor:comparison_viscosity} guarantee that $v_1 \leq v_2$ and $v_2 \leq v_1$. 
 	Therefore, we obtain that $v_1 = v_2$. 
 	This establishes $u_1 = u_2$. 
 	This completes the proof of  \cref{prop:uniqueness_viscosity_semilinear}. 
\end{proof}

\subsection{Existence results for solutions of SDEs} 
\label{subsec:sde_existence}

\begin{prop} \label{lem:krylov_existence}
	Let $ d,m \in \N $, 
		$ T \in (0,\infty) $, 
	let $ \langle\cdot,\cdot\rangle \colon \R^d\times\R^d \to \R $ be the standard Euclidean scalar product on $ \R^d $, 
	let $ \norm{\cdot}\colon\R^d\to [0,\infty)$ be the standard Euclidean norm on $ \R^d $, 
	let $ \HSnorm{\cdot} \colon \R^{d\times m} \to [0,\infty) $ be the Frobenius norm on $ \R^{d\times m} $, 
	let $ \mc O \subseteq \R^d $ be a non-empty open set, 
	for every 
		$ r \in (0,\infty) $ 
	let	$ O_r \subseteq \mc O $ satisfy $ O_r = \{x\in\mc O\colon ( \norm{x}\leq r~\text{and}~\{y\in\R^d\colon\norm{y-x}<\nicefrac{1}{r}\} \subseteq \mc O ) \}$,
	let $ \mu \in C([0,T]\times\mc O,\R^d) $, 
		$ \sigma \in C([0,T]\times\mc O,\R^{d\times m}) $ 
	satisfy for all 
		$ r \in (0,\infty) $ 
	that 
		\begin{equation} 
		\sup
		\!\left(
		\left\{ 
		\frac{
			\norm{\mu(t,x) - \mu(t,y)} 
			+ 
			\HSnorm{\sigma(t,x) - \sigma(t,y)}    
		}
		{ \norm{ x - y } } \colon 
		t\in [0,T], 
		x,y\in O_r, 
		x\neq y
		\right\}
		\cup 
		\{ 0 \}
		\right) < \infty, 
		\end{equation}
	let 
		$ V \in C^{1,2}([0,T]\times\mc O,(0,\infty))$ 
	satisfy 
		$ \limsup_{r\to\infty} \left[ \inf_{t\in [0,T]} \inf_{x\in\cO\setminus O_r} V(t,x) \right] = \infty	$, 
	assume for all 
		$ t \in [0,T] $, 
		$ x \in \mc O $ 
	that 
		\begin{equation}\label{krylov_existence:supersolution}
		(\tfrac{\partial}{\partial t}V)(t,x) 
		+ 
		\tfrac{1}{2}
		\operatorname{Trace}\!\left( 
		\sigma(t,x)[\sigma(t,x)]^{*}(\operatorname{Hess}_x V)(t,x) \right) 
		+ 
		\langle \mu(t,x),(\nabla_x V)(t,x) \rangle 
		\leq 0, 
		\end{equation}
	let $ ( \Omega,\cF,\P,(\F_t)_{t\in [0,T]} )$ be a stochastic basis, 
	let $ W \colon [0,T]\times\Omega\to\R^m $ be a standard $(\F_t)_{t\in [0,T]}$-Brownian motion, 
	and let $ \xi \in \mc O $. 
	Then there exists an up to indistinguishability unique $(\F_t)_{t\in [0,T]}$-adapted stochastic process 
		$ X = (X_t)_{t\in [0,T]} \colon [0,T]\times\Omega\to\mc O $ 
	with continuous sample paths such that for all 
		$ t \in [0,T] $ 
	we have $\P$-a.s.~that 
		\begin{equation} \label{krylov_existence:claim}
		X_t = \xi + \int_0^t \mu( s, X_s )\,ds + \int_0^t \sigma( s, X_s ) \,dW_s. 
		\end{equation} 
\end{prop}

\begin{proof}[Proof of \cref{lem:krylov_existence}]
	Throughout this proof let 
		$ \mathfrak{m}_{n} \in C([0,T]\times\mc O,\R^d) $, $ n \in \N $, 
	and $ \mathfrak{s}_{n} \in C([0,T]\times\mc O,\R^{d\times m}) $, $ n \in \N $,
	satisfy that 
	\begin{enumerate}[(A)]
		\item \label{krylov_existence:lipschitz_property_mun_sigman}
		we have for all 
			$ n \in \N $ 
		that
			\begin{equation}
			\sup_{t\in [0,T]}
			\sup_{x\in \cO} 
			\sup_{y\in \cO\setminus\{x\}} 
			\left[ 
			\frac{\norm{\mathfrak{m}_n(t,x)-\mathfrak{m}_n(t,y)} 
				+ 
				\HSnorm{\mathfrak{s}_n(t,x)-\mathfrak{s}_{n}(t,y)}}{\norm{x-y}}
			\right] 
			< \infty,  
			\end{equation}
		\item 
		\label{krylov_existence:coinciding_with_mu_and_sigma}
		we have for all 
			$ n \in \N $, 
			$ t \in [0,T] $, 
			$ x \in \mc O $ 
		that 
			\begin{equation}
			\mathbbm{1}_{\{V\leq n\}}(t,x)\left[ 
			\norm{\mathfrak{m}_{n}(t,x)-\mu(t,x)} + \HSnorm{\mathfrak{s}_{n}(t,x) - \sigma(t,x)}
			\right] 
			= 
			0,  
			\end{equation}
		and
		\item 
		\label{krylov_existence:vanishing_far_away}
		we have for all 
			$ n \in \N $, 
			$ t \in [0,T] $, 
			$ x \in \mc O $
		that 
		\begin{equation}
		\mathbbm{1}_{\{V\geq n+1\}}(t,x)\left[ 
		\norm{\mathfrak{m}_{n}(t,x)} 
		+ 
		\HSnorm{\mathfrak{s}_{n}(t,x)}
		\right] 
		= 
		0.  
		\end{equation}
	\end{enumerate}
	Note that Items~\eqref{krylov_existence:lipschitz_property_mun_sigman} and \eqref{krylov_existence:vanishing_far_away} 
	ensure that there exist $(\F_t)_{t\in [0,T]}$-adapted stochastic processes $\mathfrak{X}^{(n)}=(\mathfrak{X}^{(n)}_t)_{t\in [0,T]}\colon [0,T]\times\Omega\to\cO$, $ n \in \N $,  with continuous sample paths satisfying that for all 
		$ n \in \N $, 
		$ t \in [0,T] $ 
	we have $\P$-a.s.~that 
		\begin{equation} \label{krylov_existence:sde_for_Xn}
		\mathfrak{X}^{(n)}_t 
		= 
		\xi + \int_0^t \mathfrak{m}_{n}( s, \mathfrak{X}^{(n)}_s )\,ds + \int_0^t \mathfrak{s}_{n}( s, \mathfrak{X}^{(n)}_s )\,ds
		\end{equation}
	(cf., e.g, Karatzas \& Shreve~\cite[Theorem 5.2.9]{KaSh1991_BrownianMotionAndStochasticCalculus} and \cite[Item~(ii) in Lemma 3.4]{StochasticFixedPointEquations}). 
	Next let $ \tau_n \colon \Omega \to [0,T] $, $ n \in \N $, 
	satisfy for all 
		$ n \in \N $ 
	that 
		$ \tau_n = \inf( \{t\in [0,T]\colon V(t,\mathfrak{X}^{(n)}_t) \geq n\} \cup \{T\} ) $. 
	Moreover, observe that  Item~\eqref{krylov_existence:coinciding_with_mu_and_sigma} ensures that for all 
		$ m \in \N $, 
		$ n \in \N \cap [m,\infty) $
	we have that 
		$ \P(\Forall t\in [0,T]\colon \mathbbm{1}_{\{t\leq \tau_m\}} \mathfrak{X}^{(n)}_t = \mathbbm{1}_{\{t\leq \tau_m\}} \mathfrak{X}^{(m)}_t ) = 1 $ (cf., e.g, \cite[Lemma 3.5]{StochasticFixedPointEquations}). 
	Combining this with \eqref{krylov_existence:sde_for_Xn} and Item~\eqref{krylov_existence:coinciding_with_mu_and_sigma} proves that for all 
		$ n \in \N $,
		$ t \in [0,T] $ 
	we have $\P$-a.s.~that 
		\begin{equation}
		\begin{split}
		\mathfrak{X}^{(n)}_{\min\{\tau_n,t\}} 
		& = 
		\xi 
		+ 
		\int_0^{\min\{\tau_n,t\}} 
		\mathfrak{m}_n( s, \mathfrak{X}^{(n)}_s )\,ds 
		+ 
		\int_0^{\min\{\tau_n,t\}}
		\mathfrak{s}_{n}( s, \mathfrak{X}^{(n)}_s)\,dW_s 
		\\
		& = \xi 
		+ 
		\int_0^{\min\{\tau_n,t\}} 
		\mu(s,\mathfrak{X}^{(n)}_s)\,ds 
		+ 
		\int_0^{\min\{\tau_n,t\}}
		\sigma(s,\mathfrak{X}^{(n)}_s)\,dW_s 
		\\
		& = 
		\xi + 
		\int_0^t 
		\mathbbm{1}_{\{s\leq\tau_n\}}
		\mu(s,\mathfrak{X}^{(n)}_{\min\{\tau_n,s\}})\,ds
		+ 
		\int_0^t 
		\mathbbm{1}_{\{s\leq\tau_n\}}
		\sigma(s,\mathfrak{X}^{(n)}_{\min\{\tau_n,s\}})\,dW_s .
		\end{split}
		\end{equation}
	It\^o's formula hence guarantees that for all 
		$ n \in \N $, 
		$ t \in [0,T] $ 
	we have $\P$-a.s.~that 
		\begin{equation} 
		\begin{split} 
		& 
		V(\min\{\tau_n,t\}, \mathfrak{X}^{(n)}_{\min\{\tau_n,t\}}) 
		= 
		V(0,\xi) 
		+ 
		\int_0^{\min\{\tau_n,t\}} \left\langle (\nabla_x V)(s,\mathfrak{X}^{(n)}_{\min\{\tau_n,s\}}), \sigma(s,\mf X^{(n)}_{\min\{\tau_n,s\}})\,dW_s \right\rangle 
		\\
		& + 
		\int_0^{\min\{\tau_n,t\}} (\tfrac{\partial}{\partial t}V)(s,\mf X^{(n)}_{\min\{\tau_n,s\}}) \,ds 
		+
		\int_0^{\min\{\tau_n,t\}} \left\langle \mu(s,\mathfrak{X}^{(n)}_{\min\{\tau_n,s\}}), (\nabla_x V)(s,\mf X^{(n)}_{\min\{\tau_n,s\}}) \right\rangle \,ds 
		\\
		& + 
		\int_0^{\min\{\tau_n,t\}} 
		\tfrac12 \operatorname{Trace}\!\big( \sigma(s,\mf X^{(n)}_{\min\{\tau_n,s\}})[\sigma(s,\mf X^{(n)}_{\min\{\tau_n,s\}})]^{*} (\operatorname{Hess}_x V)(s,\mf X^{(n)}_{\min\{\tau_n,s\}}) \big) \,ds . 
		\end{split} 
		\end{equation} 
	This and \eqref{krylov_existence:supersolution} show that for all 
		$ n \in \N $, 
		$ t \in [0,T] $ 
	we have $ \P $-a.s.~that 
		\begin{equation}
		V(\min\{\tau_n,t\}, \mathfrak{X}^{(n)}_{\min\{\tau_n,t\}})
		\leq  
		V(0,\xi)
		+ 
		\int_0^{\min\{\tau_n,t\}} 
		\langle
		(\nabla_x V)(s, \mathfrak{X}^{(n)}_s), \sigma(s, \mathfrak{X}^{(n)}_s)
		\,dW_s 
		\rangle. 
		\end{equation}
	Hence, we obtain for all 
		$ n \in \N $, 
		$ t \in [0,T] $  
	that 
		\begin{equation} 
		\EXPP{ V(\min\{\tau_n,t\}, \mathfrak{X}^{(n)}_{\min\{\tau_n,t\}}) } \leq V(0,\xi).  
		\end{equation} 
	This implies for all 
		$ n \in \N $ 
	that 
		\begin{equation} 
		\EXPP{ V(\tau_n, \mf X^{(n)}_{\tau_n}) } \leq V(0,\xi). 
		\end{equation} 
	Markov's inequality and the fact that $ \mf X^{(n)} \colon [0,T] \times \Omega \to \mc O $, $ n \in \N $, are stochastic processes with continuous sample paths hence ensure that for all
		$ n \in \N $ 
	we have that 
		\begin{equation}
		\P(\tau_n < T) 
		\leq  
		\P\!\left(
		V(\tau_n,\mathfrak{X}^{(n)}_{\tau_n}) \geq n 
		\right) 
		\leq 
		\frac{1}{n}
		\EXPP{V(\tau_n,\mathfrak{X}^{(n)}_{\tau_n})}
		\leq 
		\frac{V(0,\xi)}{n}. 
		\end{equation}
	Therefore, we obtain that 
		\begin{equation}
		\sum_{n=1}^{\infty} 
		\P(\tau_{n^2}<T) 
		\leq 
		V(0,\xi) \left[ \sum_{n=1}^{\infty} \frac{1}{n^2} \right]
		< \infty. 
		\end{equation}
	The Borel-Cantelli lemma hence yields that $ \P(\Exists n\in\N\colon \tau_n=T)=1 $. 
	This demonstrates that there exists an $(\F_t)_{t\in [0,T]}$-adapted stochastic process $ X \colon [0,T]\times \Omega \to \mc O $ 	with continuous sample paths satisfying that 
	$ \liminf_{n\to\infty} \P(\Forall t\in [0,T]\colon X_t = \mathfrak{X}^{(n)}_t) = 1 $.
	Item~\eqref{krylov_existence:coinciding_with_mu_and_sigma} 
	hence yields that for all 
		$ t \in [0,T] $
	we have that 
		\begin{equation}
		\limsup_{n\to\infty} 
		\Exp{\min\!\left\{
			1,\int_0^t 
			\HSnorm{
				\mathfrak{s}_{n}(s,\mathfrak{X}^{(n)}_s) 
				-
				\sigma(s,X_s)}^2 \,ds 
			\right\}}
		= 0.  
	\end{equation}
	This, the fact that for all 
		$ t \in [0,T] $ 
	we have $\P$-a.s.~that 
	\begin{equation}
	\limsup_{n\to\infty} 
	\norm{
		\int_0^t \mathfrak{m}_n(s,\mathfrak{X}^{(n)}_s)\,ds 
		- 
		\int_0^t \mu(s,X_s)\,ds
	}
	= 0, 
	\end{equation} 
	and \eqref{krylov_existence:sde_for_Xn} guarantee that for all 
		$ t \in [0,T] $
	we have $\P$-a.s.~that 
		\begin{equation} \label{krylov_existence:existence}
		\begin{split}
		X_t 
		& = 
		\xi
		+ 
		\int_0^t 
		\mu(s,X_s)\,ds 
		+ 
		\int_0^t 
		\sigma(s,X_s)\,dW_s. 
		\end{split}
		\end{equation}
	This and, e.g., Karatzas \& Shreve~\cite[Theorem 5.2.5]{KaSh1991_BrownianMotionAndStochasticCalculus}  establish \eqref{krylov_existence:claim}. 
	This completes the proof of  \cref{lem:krylov_existence}. 
\end{proof}

\subsection{Existence results for viscosity solutions of semilinear Kolmogorov PDEs} 
\label{subsec:viscosity_semilinear}

\begin{theorem} \label{thm:existence_of_fixpoint}
	Let $ d,m \in \N $, 
	    $ L, T \in (0,\infty) $,
	let $ \langle\cdot,\cdot\rangle\colon\R^d\times\R^d\to\R $ be the standard Euclidean scalar product on $ \R^d $, 
	let $ \norm{\cdot}\colon\R^d\to [0,\infty) $ be the standard Euclidean norm on $ \R^d $, 
	let $ \HSnorm{\cdot}\colon\R^{d\times m}\to [0,\infty) $ be the Frobenius norm on $ \R^{d\times m} $,     
	let $ \mc O \subseteq \R^d $ be a non-empty open set,
	for every 
	    $ r \in (0,\infty) $ 
	let 
	    $ O_r \subseteq \mc O $ 
	satisfy 
	    $ O_r = \{x\in\mc O \colon (\norm{x}\leq r~\text{and}~\{y\in\R^d\colon \norm{y-x} < \nicefrac{1}{r}\}\subseteq \mc O) \} $,
	let $ \mu \in C([0,T]\times \mc O,\R^d) $, 
	    $ \sigma \in C([0,T]\times\mc O,\R^{d\times m})$, 
		$ f \in C( [0,T] \times \mc O \times \R , \R) $, 
		$ g \in C(\mc O, \R) $, 
		$ V \in C^{1,2}([0,T]\times\mc O,(0,\infty)) $, 
	assume for all 
	    $ r \in (0,\infty) $ 
	that 
	    \begin{equation} 
		\sup\!\left(
	     \left\{
	      \frac{
	      \norm{ \mu(t,x) - \mu(t,y) } 
	      + 
	      \HSnorm{
	        \sigma(t,x) - \sigma(t,y) }}
	      {\norm{ x - y }}
	      \colon t\in [0,T], x,y\in O_r, x\neq y
	     \right\}
	     \cup \{0\}
	     \right)
	     < \infty, 
	    \end{equation}
	assume 
		$ \sup_{r\in (0,\infty)} 
		[\inf_{t\in [0,T]}\inf_{x\in \cO \setminus O_r} V(t,x)]
		= \infty $
	and 
		$ \inf_{r\in (0,\infty)} 
		[\sup_{t\in [0,T]} 
		\sup_{x\in\cO \setminus O_r} 
		(\frac{|f(t,x,0)|}{V( t , x )}
		+
		\frac{|g(x)|}{V(T,x)} ) ] \allowbreak =0 $,     
	assume for all 
	    $ t \in [0, T] $, 
	    $ x \in \mc O $, 
	    $ v, w \in \R $ 
	that
		$ | f(t,x, v) - f(t, x, w) | \leq L | v - w | $
	and 
		\begin{equation}
		(\tfrac{\partial }{\partial t}V)(t,x) 
		+
		\tfrac12 
		\operatorname{Trace}\!\left( 
		\sigma(t,x)[\sigma(t,x)]^{*}
		(\operatorname{Hess}_x V)(t,x) 
		\right) 
		+ 
		\langle 
		\mu(t,x),(\nabla_x V)(t,x)
		\rangle
		\leq 
		0, 
		\end{equation}
	let $ ( \Omega, \mathcal{F}, \P, (\mathbb{F}_t)_{t\in [0,T]} ) $ be a stochastic basis,  
	and let $ W \colon [0,T] \times \Omega \to \R^m $ be a standard $(\mathbb{F}_t)_{t\in [0,T]}$-Brownian motion. 
	Then 
		\begin{enumerate}[(i)]
		\item \label{existence_of_fixpoint:item2}
		there exists a unique viscosity solution 
			$ u \in \{ {\bf u} \in C([0,T] \times \mc O, \R) \colon \limsup_{r\to\infty} [\sup_{t\in [0,T]}\sup_{x\in\cO\setminus O_r} \allowbreak (\frac{| {\bf u}(t,x)|}{V(t,x)})] = 0  \} $ 
		of 
			\begin{multline}
			(\tfrac{\partial}{\partial t}u)(t,x) 
			+
			\tfrac{1}{2} 
			\operatorname{Trace}\! \big( 
			\sigma(t, x)[\sigma(t, x)]^{\ast}(\operatorname{Hess}_x u )(t,x)
			\big) 
			+ 
			\langle 
			\mu(t,x), (\nabla_x u)(t,x)
			\rangle
		 	\\ 
			+
			f(t, x, u(t, x))
			=
			0
			\end{multline} 
		with $u(T,x) = g(x)$ for $ (t,x) \in (0,T)\times\mc O $,  
				\item \label{existence_of_fixpoint:item1}
		for every 
			$ t \in [0,T] $, 
			$ x \in \mc O $ 
		there exists an up to indistinguishability unique $(\mathbb{F}_s)_{s\in [t,T]}$-adapted stochastic process 	$X^{t,x}=(X^{t,x}_{s})_{s\in [t,T]}\colon [t,T]\times\Omega\to\cO$ with continuous sample paths satisfying that for all 
			$ s \in [t,T] $ 
		we have $\P$-a.s.~that 
			\begin{equation}\label{existence_of_fixpoint:ass4}
			X^{t,x}_{s} = x + \int_t^s \mu(r, X^{t,x}_{r})\, dr + \int_t^s \sigma(r, 	X^{t,x}_{r}) \, dW_r,
			\end{equation}
		\item \label{existence_of_fixpoint:item3}
		there exists a unique  
	    	$ v \in C(0,T]\times\mc O,\R) $ 
		which satisfies for all 
			$ t \in [0,T] $, 
			$ x \in \mc O $ 
		that
	    	$ \limsup_{r\to\infty} [ \sup_{s\in [0,T]} \sup_{y\in \mc O \setminus O_r} ( \frac{|v(s,y)|}{V(s, y)} ) ] = 0 $, 
			$ \EXP{|g(X^{t,x}_T)| + \int_t^T |f(s,X^{t,x}_s,v(s,X^{t,x}_s))|\,ds} < \infty $, 
    	and 
			\begin{equation}
	  		v(t, x)
			=
  			\Exp{
    		g ( X^{t,x}_{T} )
    		+
    		\int_{t}^T
    		  f( s, X^{t,x}_{s},  v(s, X^{t,x}_{s}) )
    		\, ds
  			}\!, 
			\end{equation}
		and 
		\item\label{existence_of_fixpoint:item4} 
		we have for all 
	    	$ t \in [0,T] $, 
	    	$ x \in \mc O $ 
    	that 
    		$ u(t,x) = v(t,x) $ 
	\end{enumerate}
	(cf.~\cref{def:viscosity_solution}). 
\end{theorem}

\begin{proof}[Proof of \cref{thm:existence_of_fixpoint}]
	First, observe that \cref{lem:krylov_existence} (applied with 
		$ d \is d $, 
		$ m \is m $, 
		$ T \is T-t $, 
		$ \mc O \is \mc O $, 
		$ \mu \is ( [0,T-t]\times\mc O \ni (s,x) \mapsto \mu(t+s,x) \in \R^d ) $, 
		$ \sigma \is ([0,T-t]\times\mc O \ni (s,x) \mapsto \sigma(t+s,x) \in \R^{d\times m} ) $, 
		$ V \is ([0,T-t]\times\mc O \ni (s,x) \mapsto V(t+s,x) \in (0,\infty) ) $, 
		$ ( \Omega, \mc F, \P ) \is ( \Omega, \mc F, \P ) $, 
		$ (\F_s)_{s\in [0,T]} \is (\F_{s+t})_{s\in [0,T-t]} $, 
		$ (W_s)_{s\in [0,T]} \is (W_{s+t}-W_t)_{s\in [0,T-t]} $
	for $ t \in [0,T] $ in the notation of \cref{lem:krylov_existence}) establishes Item~\eqref{existence_of_fixpoint:item1}. 
	Next we prove Item~\eqref{existence_of_fixpoint:item3}. Note that Item~\eqref{existence_of_fixpoint:item1} ensures that there exists a unique $ v \in C([0,T] \times \mc O, \R) $ which satisfies for all 
    	$ t \in [0,T] $, 
    	$ x \in \mc O $ 
    that   
		$ \limsup_{r\to\infty} [\sup_{s\in [0,T]}\allowbreak\sup_{y\in \cO\setminus O_r} (\frac{|v(s,y)|}{V(s, y)}) ] = 0 $, 
		$ \EXP{|g(X^{t,x}_T)|+\int_t^T |f(s,X^{t,x}_s,v(s,X^{t,x}_s))|\,ds} < \infty $, 
    and   
	    \begin{equation}\label{existence_of_fixpoint:formula_for_v}
    	v(t,x) = \Exp{ g( X^{t,x}_T ) + \int_t^T f( s , X^{t,x}_s, v(s, X^{t,x}_s ) ) \,ds } 
	    \end{equation}
	(cf., e.g.,  \cite[Theorem 3.8]{StochasticFixedPointEquations}).
    This establishes Item~\eqref{existence_of_fixpoint:item3}.
    In the next step we prove Items~\eqref{existence_of_fixpoint:item2} and \eqref{existence_of_fixpoint:item4}. 
	For this let  
    	$ h \colon [0,T] \times \mc O \to \R $ 
    satisfy for all
    	$ t \in [0,T] $, 
    	$ x \in \mc O $ 
    that 
    	$ h(t,x) = f(t,x,v(t,x)) $. 
	Observe that 
    	$ h \in C([0,T]\times\mc O, \R) $ 
    and  
		\begin{equation} 
 		\begin{split}
  		& 
  		\limsup_{r\to\infty} \left[ \sup_{t \in [0,T]} \sup_{x \in \mc O \setminus O_r} \left( \frac{|h(t,x)|}{V(t,x)} \right) \right] 
  		= \limsup_{r\to\infty} \left[ \sup_{t \in [0,T]} \sup_{x \in \mc O \setminus O_r} \left( \frac{|f(t,x,v(t,x))|}{V(t,x)} \right) \right] \\
  		&  
  		\leq \limsup_{r\to\infty} \left[ \sup_{t \in [0,T]} \sup_{x \in \mc O \setminus O_r} \left( \frac{|f(t,x,0)|+|f(t,x,v(t,x))-f(t,x,0)|}{V(t,x)} \right) \right] \\
  		&   
		\leq \limsup_{r\to\infty} \left[ \sup_{t \in [0,T]} \sup_{x \in \mc O\setminus O_r} \left( \frac{|f(t,x,0)| + L |v(t,x)|}{V(t,x)}   \right) \right] = 0. 
		\end{split}
 		\end{equation} 
 	\cref{prop:viscosity_solution_lyapunov}, Item~\eqref{existence_of_fixpoint:item1}, and \eqref{existence_of_fixpoint:formula_for_v} hence guarantee that $v$ is a viscosity solution of 
		\begin{equation}\label{existence_of_fixpoint:auxiliary_eqn}
		 (\tfrac{\partial}{\partial t}v)(t,x) 
		 + 
		 \tfrac12 
		 \operatorname{Trace}\!\big( 
		   \sigma(t,x)[\sigma(t,x)]^{*}
		   (\operatorname{Hess}_x v)(t,x)
		 \big) 
		 + 
		 \langle \mu(t,x), (\nabla_x v)
		 (t,x) 
		 \rangle
		 + 
		 h(t,x) 
		 = 
		 0 
		\end{equation}
	for $(t,x)\in (0,T)\times\cO$. 
	This implies that for all 
    	$ t \in (0,T) $, 
    	$ x \in \mc O $, 
    	$ \phi \in C^{1,2}((0,T)\times\mc O,\R)$ 
    with $ \phi \geq v $ and $ \phi(t,x) = v(t,x) $ 
 	we have that 
		\begin{align} \label{viscosity_solution_semilinear:viscosity_subsolution}
		\nonumber
		& (\tfrac{\partial}{\partial t}\phi)(t,x) 
		+ 
		\tfrac12 
		\operatorname{Trace}\!\left( 
		 \sigma(t,x)[\sigma(t,x)]^{*}
		 (\operatorname{Hess}_x \phi)(t,x) 
		\right) 
		+ 
		\langle 
		 \mu(t,x),(\nabla_x \phi)(t,x)
		\rangle
		+ 
		f(t,x,\phi(t,x)) 
		\\
		& 
		=
		(\tfrac{\partial}{\partial t}\phi)(t,x) 
		+ 
		\tfrac12 
		\operatorname{Trace}\!\left( 
		 \sigma(t,x)[\sigma(t,x)]^{*}
		 (\operatorname{Hess}_x \phi)(t,x) 
		\right) 
		+ 
		\langle 
		 \mu(t,x),(\nabla_x \phi)(t,x)
		\rangle
		+ 
		h(t,x)
		\\ \nonumber
		&
		\geq 0.  
		\end{align}
	In addition, note that \eqref{existence_of_fixpoint:auxiliary_eqn} proves that for all 
  		$ t \in (0,T) $, 
  		$ x \in \mc O $,
  		$ \phi \in C^{1,2}((0,T)\times\mc O,\R) $ 
  	with $ \phi \leq v $ and $ \phi(t,x) = v(t,x) $ 
  	we have that 
  		\begin{align} \label{viscosity_solution_semilinear:viscosity_supersolution}
		\nonumber  
		& 
		(\tfrac{\partial}{\partial t}\phi)(t,x) 
		+ 
		\tfrac12 
		\operatorname{Trace}\!\left( 
		 \sigma(t,x)[\sigma(t,x)]^{*}
		 (\operatorname{Hess}_x \phi)(t,x) 
		\right) 
		+ 
		\langle 
		 \mu(t,x),(\nabla_x \phi)(t,x)
		\rangle
		+ 
		f(t,x,\phi(t,x))
		\\
		&
		=
		(\tfrac{\partial}{\partial t}\phi)(t,x) 
		+ 
		\tfrac12 
		\operatorname{Trace}\!\left( 
		 \sigma(t,x)[\sigma(t,x)]^{*}
		 (\operatorname{Hess}_x \phi)(t,x) 
		\right) 
		+ 
		\langle 
		  \mu(t,x),(\nabla_x \phi)(t,x)
		\rangle
		+ 
		h(t,x)
		\\ \nonumber
		& \leq 0 .
		\end{align}
	Combining this with \eqref{viscosity_solution_semilinear:viscosity_subsolution} shows that 
    	$v$
 	is a viscosity solution of 
		\begin{multline}
		 \label{existence_of_fixpoint:equation_for_v}
		  (\tfrac{\partial}{\partial t}v)(t,x) 
		  + 
		  \tfrac12 
		  \operatorname{Trace}\!\left( 
		    \sigma(t,x)[\sigma(t,x)]^{*}
		    (\operatorname{Hess}_x v)(t,x)
		  \right) 
		  + 
		  \langle \mu(t,x), (\nabla_x v)
		  (t,x) 
		  \rangle
		  \\
		  + 
		  f(t,x,v(t,x)) 
		  = 
		  0  
		\end{multline}
	 for $(t,x) \in (0,T)\times\cO$.
	 Combining this and the fact that 
	 	$ v \in \{ {\bf u} \in C([0,T]\times\mc O,\R)\colon \limsup_{r\to\infty}\allowbreak[\sup_{t\in [0,T]}\sup_{x\in\mc O\setminus O_r}(\frac{|{\bf u}(t,x)|}{V(t,x)})] = 0 \}$ 
	 with
	 \cref{prop:uniqueness_viscosity_semilinear} (applied with 
	 	$ u_1 \is v $
	 in the notation of \cref{prop:uniqueness_viscosity_semilinear}) establishes Items~\eqref{existence_of_fixpoint:item2} and \eqref{existence_of_fixpoint:item4}. 
 	This completes the proof of  \cref{thm:existence_of_fixpoint}.
\end{proof}

\begin{cor}\label{cor:existence_of_fixpoint_product_lyapunov}
	Let $ d,m \in \N $,
	 	$ T \in (0,\infty) $, 
	 	$ L,\rho \in \R $,
 	let $ \langle \cdot, \cdot \rangle \colon \R^d\times\R^d \to \R $ be the standard Euclidean scalar product on $ \R^d $, 
 	let $ \norm{\cdot}\colon \R^d\to [0,\infty) $ be the standard Euclidean norm on $ \R^d $, 
 	let $ \HSnorm{\cdot} \colon \R^{d\times m}\to [0,\infty) $ be the Frobenius norm on $\R^{d\times m}$, 
 	let $ \mc O\subseteq \R^d$ be a non-empty open set, 
 	for every 
 		$ r \in (0,\infty) $
 	let $ O_r\subseteq \mc O $ satisfy $ O_r = \{x\in \mc O\colon ( \norm{x}\leq r~\text{and}~\{y\in\R^d\colon\norm{y-x}<\nicefrac{1}{r}\}\subseteq\mc O ) \} $, 
 	let $ \mu\in C([0,T]\times\mc O,\R^d) $, 
 		$ \sigma\in C([0,T]\times\mc O,\R^{d\times m}) $, 
 		$ f \in C([0,T]\times\mc O\times\R,\R) $, 
 		$ g \in C(\mc O,\R) $,   
 		$ V \in C^2(\mc O,(0,\infty)) $,  		
 	assume for all 
 		$r\in (0,\infty)$ 
 	that 
 		\begin{equation}
 		 \sup\!\left( 
 		  \left\{
 		   \frac{\norm{\mu(t,x)-\mu(t,y)}
 		   	+ 
 		   	\HSnorm{\sigma(t,x)-\sigma(t,y)}}{\norm{x-y}}
 		   	\colon 
    		t\in [0,T], 
    		x,y\in O_r,
    		x\neq y
 		  \right\}
 		  \cup \{0\}
 		 \right) 
 		 < 
 		 \infty, 
 		\end{equation}
	assume that
 		$ \sup_{r\in (0,\infty)} [ \inf_{ x \in \mc O\setminus O_r} V(x)] = \infty $
 	and 
	 	$ \inf_{r\in (0,\infty)} [ \sup_{t\in [0,T]} \sup_{x\in \cO\setminus O_r} (\frac{|f(t,x,0)|+|g(x)|}{V(x)}) ] = 0 $, 
 	assume for all
 		$ t \in [0,T] $, 
 		$ x \in \mc O $, 
 		$ v,w \in \R $ 
 	that 
 		$ |f(t,x,v)-f(t,x,w)| \leq L|v-w| $
	and 
 		\begin{equation} \label{existence_of_fixpoint_product_lyapunov:V_supersolution}
 		\tfrac12
 		\operatorname{Trace}\!\left(
 		\sigma(t,x)[\sigma(t,x)]^{*}(\operatorname{Hess} V)(x)
 		\right)
 		+ 
 		\langle \mu(t,x),(\nabla V)(x)\rangle 
 		\leq 
 		\rho V(x), 
 		\end{equation}
 	let $ (\Omega,\cF,\P,(\F_t)_{t\in [0,T]}) $ be a stochastic basis, 
	and let $ W \colon [0,T]\times\Omega \to \R^m $ be a standard $(\F_t)_{t\in [0,T]}$-Brownian motion. 
	Then 
		\begin{enumerate}[(i)]
  		\item \label{existence_of_fixpoint_product_lyapunov:item2}
  		there exists a unique viscosity solution $u\in \{  {\bf u} \in C([0,T]\times\cO,\R) \colon \limsup_{r \to \infty } [\sup_{s\in [0,T]}\sup_{y\in\cO\setminus O_r} \allowbreak (\frac{|{\bf u}(s,y)|}{V(y)})] = 0 \} $ of
		    \begin{multline}
		    (\tfrac{\partial}{\partial t}u)(t,x) 
	    	+ 
	    	\tfrac12 
	    	\operatorname{Trace}\!\left( 
	    	\sigma(t,x)[\sigma(t,x)]^{*}
	    	(\operatorname{Hess}_x u)(t,x)\right) 
	    	+ 
	    	\langle 
	    	\mu(t,x),(\nabla_x u)(t,x) 
	    	\rangle 
	    	\\
	    	+ 
	    	f(t,x,u(t,x)) = 0
			\end{multline} 
		with $u(T,x) = g(x)$ for $(t,x) \in (0,T)\times\cO$,  
  			\item \label{existence_of_fixpoint_product_lyapunov:item1}
  		for every 
	  		$ t \in [0,T] $, 
	  		$ x \in \mc O $ 
  		there exists an up to indistinguishability unique $(\F_s)_{s\in [t,T]}$-adapted stochastic process $X^{t,x} = (X^{t,x}_s)_{s\in [t,T]}\colon [t,T]\times\Omega \to \cO$ with continuous sample paths satisfying that for all 
  			$ s \in [t,T] $
  		we have $\P$-a.s.~that 
	  		\begin{equation}
	  		X^{t,x}_s = x + \int_t^s \mu(r,X^{t,x}_r)\,dr + 
	  		\int_t^s \sigma(r,X^{t,x}_r)\,dW_r, 
	  		\end{equation}
  		\item \label{existence_of_fixpoint_product_lyapunov:item3}
		there exists a unique $v\in C([0,T]\times\cO,\R)$ which satisfies for all 
  			$ t \in [0,T] $, 
  			$ x \in \mc O $ 
  		that 
	  		$ \limsup_{r\to\infty} [ \sup_{s\in [t,T]} \sup_{y\in\mc O \setminus O_r} ( \frac{|v(s,y)|}{V(y)} ) ] = 0 $, 
		  	$ \EXP{|g(X^{t,x}_T)|+\int_t^T |f(s,X^{t,x}_s,v(s,X^{t,x}_s))|\,ds} < \infty $,   
	 	and 
			\begin{equation}
			v(t,x) = \Exp{ g( X^{t,x}_T ) + \int_t^T f( s, X^{t,x}_s, v(s,X^{t,x}_s) ) \,ds }\!, 
	  		\end{equation}  
  		and 
  	\item \label{existence_of_fixpoint_product_lyapunov:item4}
	we have for all 
  		$t\in [0,T]$, 
  		$x\in \cO$ 
  	that 
  		$u(t,x) = v(t,x)$ 
 	\end{enumerate}
	(cf.~\cref{def:viscosity_solution}). 
\end{cor}

\begin{proof}[Proof of \cref{cor:existence_of_fixpoint_product_lyapunov}]
	Throughout this proof let 
		$ \V\colon [0,T]\times \mc O \to (0,\infty) $ 
	satisfy for all 
 		$ t \in [0,T] $, 
 		$ x \in \mc O $ 
	that 
		$ \V(t,x) = e^{-\rho t}V(x) $. 
	Observe that \eqref{existence_of_fixpoint_product_lyapunov:V_supersolution} ensures that for all 
 		$ t\in [0,T] $, 
 		$ x\in \mc O $ 
 	we have that 
 		\begin{equation}
  		(\tfrac{\partial}{\partial t}\V)(t,x) 
  		+ 
  		\tfrac12 
  		\operatorname{Trace}\!\left( 
  		\sigma(t,x)[\sigma(t,x)]^{*}(\operatorname{Hess}_x \V)(t,x)
  		\right)
  		+ 
  		\langle \mu(t,x), (\nabla_x\V)(t,x) \rangle 
  		\leq 0  
 		\end{equation}
 	(cf., e.g, \cite[Lemma 3.2]{StochasticFixedPointEquations}). 
	Moreover, note that the hypothesis that $ \sup_{ r\in (0,\infty) } [\inf_{x\in\mc O\setminus O_r} V(x)] = \infty$ assures that 
		\begin{equation}
		\sup_{ r\in (0,\infty) } 
		\left[ 
		\inf_{t\in [0,T]}
		\inf_{x\in \cO\setminus O_r} 
		\V(t,x)
		\right] 
		= \infty. 
		\end{equation}
	In addition, observe that the hypothesis that 
		$ \inf_{r\in (0,\infty)}\big[ \sup_{t\in [0,T]} \sup_{x\in \mc O \setminus O_r} \big( \frac{|f(t,x,0)| + |g(x)|}{V(x)} \big) \big] = 0 $
	guarantees that 
		\begin{equation} 
		\inf_{r\in (0,\infty)} \left[ 
   		\sup_{t\in [0,T]}
   		\sup_{x\in \cO\setminus O_r} 
   		\left( 
   		\frac{|f(t,x,0)|}{\V(t,x)} 
   		+ 
   		\frac{|g(x)|}{\V(T,x)}
   		\right)
		\right] = 0. 
		\end{equation}
	\cref{thm:existence_of_fixpoint} hence establishes Items  \eqref{existence_of_fixpoint_product_lyapunov:item1}--\eqref{existence_of_fixpoint_product_lyapunov:item4}. 
 	This completes the proof of  \cref{cor:existence_of_fixpoint_product_lyapunov}.
\end{proof}

\begin{cor} \label{existence_of_fixpoint_polynomial_growth}
	Let $ d,m \in \N $,
	    $ L,T \in (0,\infty) $, 
		$ \mu \in C([0,T]\times \R^d,\R^d) $, 
		$ \sigma \in C([0,T]\times\R^d,\R^{d\times m}) $, 
		$ \mf C \in C([0,\infty),[0,\infty)) $,  
	let	$ f \in C( [0,T] \times \R^d \times \R , \R) $, 
		$ g \in C(\R^d, \R) $ 
	be at most polynomially growing, 
	let $ \langle\cdot,\cdot\rangle \colon \R^d\times\R^d\to\R $ be the standard Euclidean scalar product on $\R^d$, 
	let $ \norm{\cdot}\colon\R^d\to [0,\infty)$ be the standard Euclidean norm on $\R^d$, 
	let $ \HSnorm{\cdot}\colon\R^{d\times m}\to [0,\infty)$ be the Frobenius norm on $\R^{d\times m}$, 
	assume for all 
		$ t \in [0, T] $, 
		$ x,y \in \R^d $, 
		$ v, w \in \R $ 
	that
		$ \norm{ \mu(t,x) - \mu(t,y) } + \HSnorm{ \sigma(t,x) - \sigma(t,y) } \leq \mf C(\norm{x}+\norm{y}) \Norm{ x - y } $, 
		$ \langle x,\mu(t,x)\rangle \leq  L ( 1 + \norm{x}^2 ) $,
		$ \HSnorm{\sigma(t,x)} \leq L ( 1+\norm{x} ) $, and 
		$ | f(t,x, v) - f(t, x, w) | \leq L | v - w | $,
	let $ ( \Omega, \mathcal{F}, \P, (\mathbb{F}_t)_{t\in [0,T]} ) $ be a stochastic basis, 
	and let $W \colon [0,T] \times \Omega \to \R^m$ be a standard $(\mathbb{F}_t)_{t\in [0,T]}$-Brownian motion. 
	Then 
		\begin{enumerate}[(i)]
		\item \label{existence_of_fixpoint_polynomial_growth:item2}
		there exists a unique at most polynomially growing viscosity solution $ u \in C([0,T] \times \R^d, \R) $ of 
			\begin{multline}
			\label{existence_of_fixpoint_polynomial_growth:claim2}
				(\tfrac{\partial}{\partial t}u)(t,x) 
				+
				\tfrac{1}{2} 
				\operatorname{Trace}\!\left( 
					\sigma(t, x)[\sigma(t, x)]^{\ast}(\operatorname{Hess}_x u )(t,x)
				\right) 
				+
			    \langle 
			    \mu(t,x), (\nabla_x u)(t,x)\rangle
			     \\
			    \quad \,
			    +
				f(t, x, u(t, x))
			=
				0
			\end{multline}
		with $ u(T,x) = g(x) $ for $ (t,x) \in (0,T) \times \R^d $,
		\item \label{existence_of_fixpoint_polynomial_growth:item1}
		for every 
			$t\in [0,T]$, 
			$x\in\R^d$ 
		there exists an up to indistinguishability unique $(\mathbb{F}_s)_{s\in [t,T]}$-adapted stochastic process
			$X^{t,x}=(X^{t,x}_s)_{s\in [t,T]}\colon [t,T]\times\Omega\to\R^d$ 
		with continuous sample paths satisfying that for all
			$s\in [t,T]$ 
		we have $\P$-a.s.~that 
			\begin{equation}
			\label{existence_of_fixpoint_polynomial_growth:claim1}
			X^{t,x}_s
			= 
			x + \int_t^s \mu(r, X^{t,x}_r)\, dr + \int_t^s \sigma(r, X^{t,x}_r) \, dW_r,
			\end{equation}
		\item \label{existence_of_fixpoint_polynomial_growth:item3}
		there exists a unique at most polynomially growing $ v \in C([0,T]\times\R^d,\R) $ which satisfies for all 
			$t\in [0,T]$, 
			$x\in\R^d$ 
		that 
			$\EXP{|g(X^{t,x}_T)|+\int_t^T |f(s,X^{t,x}_s,v(s,X^{t,x}_s))|\,ds}<\infty $ 
		and 
			\begin{equation}
			v(t, x) = \Exp{ g( X^{t,x}_{T}) + \int_{t}^T f( s, X^{t,x}_s,  v(s, X^{t,x}_s) ) \, ds }\!, 
			\end{equation}
		and 
		\item \label{existence_of_fixpoint_polynomial_growth:item4} 
		we have for all 
			$ t \in [0,T] $, 
			$ x \in \R^d $ 
		that $ u(t,x) = v(t,x) $ 
	\end{enumerate}
	(cf.~\cref{def:viscosity_solution}). 
\end{cor}

\begin{proof}[Proof of \cref{existence_of_fixpoint_polynomial_growth}]
 	Throughout this proof let 
 		$ V_q\colon\R^d \to (0,\infty) $, $ q\in (0,\infty) $,
 	satisfy for all
 		$ q \in ( 0, \infty ) $, 
 		$ x \in \R^d $ 
 	that 
 		\begin{equation} \label{existence_of_fixpoint_polynomial_growth:definition_of_V_q_functions}
 		V_q(x) = [1 + \Norm{x}^2]^{\nicefrac{q}{2}} . 
 		\end{equation} 
 	Observe that the assumption that $ f $ is at most polynomially growing and the assumption that $ g $ is at most polynomially growing ensure that there exists
 		$ p \in (0,\infty) $ 
 	which satisfies that 
		\begin{equation} 
		\sup_{t\in [0,T]} \sup_{x\in \R^d} \left( \frac{ |g(x)| + |f(t,x,0)| }{V_p(x)} \right)  < \infty . 
		\end{equation}
 	Hence, we obtain for all 
    	$ q \in (p,\infty) $ 
	that 
    	\begin{equation} \label{existence_of_fixpoint_polynomial_growth:growth_condition_for_f_and_g}
    	\limsup_{ r \to \infty } 
    	\left[ 
    	\sup_{ t \in [0,T] }
    	\sup_{ x \in \R^d, \norm{x} > r } 
    	\left( \frac{ |f(t,x,0)| + |g(x)| }
    	{ V_q(x) }
    	\right)
    	\right] 
    	= 0.
    	\end{equation}
	Moreover, note that \eqref{existence_of_fixpoint_polynomial_growth:definition_of_V_q_functions}, the assumption that for all 
		$ t \in [0,T] $, 
		$ x \in \R^d $ 
	we have that 
		$ \langle x, \mu(t,x) \rangle \leq L ( 1 + \norm{x}^2 ) $, 
	and the assumption that for all 
		$ t \in [0,T] $, 
		$ x \in \R^d $ 
	we have that  
		$ \HSnorm{ \sigma(t,x ) } \leq L ( 1 + \norm{ x } ) $
	guarantee that there exist
 		$ \rho_q \in [0,\infty) $, $ q \in \R $, 
	which satisfy for all 
		$ q \in (p,\infty) $, 
 		$ t \in [0,T] $, 
 		$ x \in \R^d $ 
	that 
		\begin{equation} \label{existence_of_fixpoint_polynomial_growth:supersolution_condition}
  		\tfrac12 
  		\operatorname{Trace}\!\left( 
   		\sigma(t,x)[\sigma(t,x)]^{*}
   		(\operatorname{Hess} V_q)(x)
  		\right)
  		+ 
  		\langle \mu(t,x),(\nabla V_q)(x)\rangle
  		\leq \rho_q V_q(x)  
 		\end{equation}
	(cf., e.g., \cite[Lemma 3.3]{StochasticFixedPointEquations}).
 	In addition, observe that \eqref{existence_of_fixpoint_polynomial_growth:definition_of_V_q_functions} demonstrates for all 
 		$ q \in (0,\infty) $ 
 	that 
		$ \liminf_{r\to\infty} [ \inf_{x\in \R^d, \norm{x} > r} V_q(x) ] = \infty $. 
	Item~\eqref{existence_of_fixpoint_product_lyapunov:item1} in \cref{cor:existence_of_fixpoint_product_lyapunov} (applied with 
		$ \rho \is \rho_{2p} $, 
		$ \mc O \is \R^d $, 
		$ V \is V_{2p} $ 
	in the notation of \cref{cor:existence_of_fixpoint_product_lyapunov}), 
	\eqref{existence_of_fixpoint_polynomial_growth:growth_condition_for_f_and_g}, 
	and \eqref{existence_of_fixpoint_polynomial_growth:supersolution_condition} 
	therefore ensure that for every 
		$ t \in [0,T] $, 
		$ x \in \R^d $ 
	there exists an up to indistinguishability unique $(\mathbb{F}_s)_{s\in [t,T]}$-adapted stochastic process $ X^{t,x}=(X^{t,x}_s)_{s\in [t,T]} \colon [t,T]\times\Omega\to\R^d$ with continuous sample paths satisfying that for all 
		$ s \in [t,T] $ 
	we have $\P$-a.s.~that 
 		\begin{equation}
  		X^{t,x}_s = x + \int_t^s \mu( r, X^{t,x}_r ) \,dr + \int_t^s \sigma( r, X^{t,x}_r ) \,dW_r.
  		\end{equation}
	This establishes Item~\eqref{existence_of_fixpoint_polynomial_growth:item1}. 
	Next we prove Item~\eqref{existence_of_fixpoint_polynomial_growth:item2}.  Note that Item~\eqref{existence_of_fixpoint_product_lyapunov:item2} in \cref{cor:existence_of_fixpoint_product_lyapunov}  (applied with 
		$ \rho \is \rho_{2p} $, 
		$ \mc O \is \R^d $, 
		$ V \is V_{2p} $ 
	in the notation of \cref{cor:existence_of_fixpoint_product_lyapunov}), \eqref{existence_of_fixpoint_polynomial_growth:growth_condition_for_f_and_g}, 
	and \eqref{existence_of_fixpoint_polynomial_growth:supersolution_condition} 
	prove that there exists a unique viscosity solution $ u \in \{ {\bf u} \in C([0,T]\times\R^d,\R) \colon \limsup_{r\to\infty}[\sup_{t\in [0,T]} \sup_{x\in\mc O\setminus O_r} (\frac{|{\bf u}(t,x)|}{V_{2p}(t,x)}) ] \allowbreak = 0  \} $ of 
		\begin{multline} \label{existence_of_fixpoint_polynomial_growth:viscosity_solution}
		 (\tfrac{\partial}{\partial t}u)(t,x) 
		 + 
		 \tfrac{1}{2} 
		 \operatorname{Trace}\!\left( 
		 \sigma(t, x)[\sigma(t, x)]^{\ast}(\operatorname{Hess}_x u )(t,x)
		 \right) 
		 +
		 \langle 
		 \mu(t,x), (\nabla_x u)(t,x)\rangle
		 \\
		 +
		 f(t, x, u(t, x))
		 =
		 0
		\end{multline}
	with $u(T,x)=g(x)$ for $(t,x)\in (0,T)\times\R^d$. 
	Next let $ v \in C([0,T]\times\R^d,\R) $ be an at most polynomially growing viscosity solution of 
		\begin{multline}
		(\tfrac{\partial}{\partial t}v)(t,x) 
	  	+ 
	  	\tfrac12 
	  	\operatorname{Trace}\!\left(
	  	\sigma(t,x)[\sigma(t,x)]^{*}(\operatorname{Hess}_x v)(t,x)\right)
	  	+ 
	  	\langle 
	  	\mu(t,x), (\nabla_x v)(t,x)
	  	\rangle 
	  	\\
	  	+ 
	  	f(t,x,v(t,x))
	  	= 
	  	0
	 	\end{multline}
	with $v(T,x) = g(x)$ for $(t,x)\in (0,T)\times\R^d$. 
	Note that the fact that $ v $ is at most polynomially growing guarantees that there exists $\alpha\in [2p,\infty)$ which satisfies that 
		\begin{equation} 
		\limsup_{r\to\infty} \left[ \sup_{t\in [0,T]} \sup_{ x\in \R^d, \norm{x}>r } \left( \frac{ |v(t,x)| }{ V_{\alpha}(x) }	\right) \right] = 0. 
		\end{equation}
	Item~\eqref{existence_of_fixpoint_product_lyapunov:item2} in  \cref{cor:existence_of_fixpoint_product_lyapunov} (applied with 
		$ \rho \is \rho_{\alpha} $, 
		$ \mc O \is \R^d $, 
		$ V \is V_{\alpha} $ 
	in the notation of \cref{cor:existence_of_fixpoint_product_lyapunov}) and \eqref{existence_of_fixpoint_polynomial_growth:viscosity_solution} hence ensure that $ u = v $. 
	This establishes Item~\eqref{existence_of_fixpoint_polynomial_growth:item2}. 
	In the next step we prove Items~\eqref{existence_of_fixpoint_polynomial_growth:item3} and~\eqref{existence_of_fixpoint_polynomial_growth:item4}.
	Observe that Item~\eqref{existence_of_fixpoint_product_lyapunov:item3} in  \cref{cor:existence_of_fixpoint_product_lyapunov} guarantees for all 
		$ t \in [0,T] $, 
		$ x \in \R^d $
	that 
		$ \EXP{|g(X^{t,x}_T)| + \int_t^T | f( s, X^{t,x}_s, u( s, X^{t,x}_s ) ) |\,ds } < \infty $ 
	and  
		\begin{equation} \label{existence_of_fixpoint_polynomial_growth:sfpe}
		u(t,x) 
		= 
		\Exp{ g( X^{t,x}_{T} ) + \int_t^T f( s, X^{t,x}_s, u(s, X^{t,x}_s) ) \,ds }\!. 
		\end{equation}
	Next let $ w \in C([0,T]\times\R^d,\R) $ be an at most polynomially growing function satisfying for all 
		$ t \in [0,T] $, 
		$ x \in \R^d $ 
	that 
		\begin{equation}
		 w(t,x) = \Exp{ g(  X^{t,x}_{T} ) + \int_t^T f( s,  X^{t,x}_s, w(s,  X^{t,x}_s) ) \,ds }\!. 
		\end{equation}
	Observe that the fact that $ w $ is at most polynomially growing yields that there exists $\beta\in [\alpha,\infty)$ which satisfies that 
		\begin{equation} 
		\limsup_{r\to\infty} 
		\left[ 
		\sup_{t\in [0,T]}
		\sup_{ x\in \R^d, \norm{x}>r }
		\left( 
		 \frac{ |w(t,x)| }{ V_{\beta}(x) }
		\right) 
		\right] = 0. 
		\end{equation}
	Combining Items~\eqref{existence_of_fixpoint_product_lyapunov:item3} and~\eqref{existence_of_fixpoint_product_lyapunov:item4} in  \cref{cor:existence_of_fixpoint_product_lyapunov} (applied with 
		$ \rho \is \rho_{\beta} $, 
		$ \mc O \is \R^d $, 
		$ V \is V_{\beta} $
	in the notation of \cref{cor:existence_of_fixpoint_product_lyapunov}) with \eqref{existence_of_fixpoint_polynomial_growth:viscosity_solution} and \eqref{existence_of_fixpoint_polynomial_growth:sfpe} hence demonstrates that $ u = v = w $. 
	This establishes Items~\eqref{existence_of_fixpoint_polynomial_growth:item3} and~\eqref{existence_of_fixpoint_polynomial_growth:item4}. 
 	This completes the proof of  \cref{existence_of_fixpoint_polynomial_growth}.
\end{proof}

\begin{lemma}\label{lem:lyapunov_function_for_heat_type_equation}
 	Let $ d,m \in \N $, 
 		$ \varepsilon,T\in (0,\infty)$,  	    
 	    $ \alpha,c \in [0,\infty) $, 
 	let $ \langle \cdot,\cdot \rangle \colon \R^d \times \R^d \to \R $ be the standard Euclidean scalar product on $ \R^d $, 
 	let $ \norm{\cdot}\colon\R^d\to [0,\infty)$ be the standard Euclidean norm on $\R^d$,
 	let $ B \colon [0,T]\times\R^d \to \R^{d\times m} $ satisfy 
 		$ \sup \{ \langle \xi, B(t,x)[B(t,x)]^{*} \xi \rangle \colon t\in [0,T], x,\xi \in \R^d, \norm{\xi}=1 \} = c \leq  \alpha $, 
 	and let $ V \colon [0,T]\times\R^d\to (0,\infty) $ satisfy for all 
 		$ t \in [0,T] $,
 		$ x \in \R^d $
 	that 
 		\begin{equation} \label{lyapunov_function_for_heat_equation:time_shifted_backwards_heat_kernel}
 		V(t,x) = \tfrac{1}{\left[ 2 \pi (\alpha t+\varepsilon)\right]^{d\slash 2}}
		\exp\!\left( 
    	\tfrac{\norm{x}^2}{2(\alpha t+\varepsilon)} 
  		\right).    
 		\end{equation}
 	Then 
		\begin{enumerate}[(i)]
 		\item \label{lyapunov_function_for_heat_equation:item1}
 			we have that 
 				$V\in C^{\infty}([0,T]\times\R^d, (0,\infty))$ 
 			and 
 		\item \label{lyapunov_function_for_heat_equation:item2}
 			we have for all 
 				$ t \in [0,T] $, 
 				$ x \in \R^d $ 
			that 
				\begin{equation}
				(\tfrac{\partial}{\partial t}V)(t,x) + \tfrac12 \operatorname{Trace}\!\big( B(t,x) [B(t,x)]^{*} (\operatorname{Hess}_x V)(t,x) \big) \leq 0 . 
 				\end{equation}
		\end{enumerate}
\end{lemma}

\begin{proof}[Proof of \cref{lem:lyapunov_function_for_heat_type_equation}]
	Throughout this proof let 
		$ \delta_{i,j} \in \R $, $ i,j \in \N $, 
	satisfy for all 
		$ i,j \in \N $ with $ i < j $ 
	that 
		$ \delta_{i,i} = 1 $ 
	and 
		$ \delta_{i,j} = \delta_{j,i} = 0 $
	and let 
    	$ b_{i,j} \colon [0,T] \times \R^d \to \R $, $ i \in \{1,2,\ldots,d\} $, $ j \in \{1,2,\ldots,m \} $, 
    satisfy for all 
		$ t \in [0,T] $, 
		$ x \in \R^d $ 
	that 
		\begin{equation} 
		B(t,x) = 
		\begin{pmatrix} 
		b_{1,1}(t,x) & b_{1,2}(t,x) & \ldots & b_{1,m}(t,x) \\
		b_{2,1}(t,x) & b_{2,2}(t,x) & \ldots & b_{2,m}(t,x) \\
		\vdots & \vdots & \ddots & \vdots \\
		b_{d,1}(t,x) & b_{d,2}(t,x) & \ldots & b_{d,m}(t,x) 
		\end{pmatrix}. 
		\end{equation}  
	Observe that \eqref{lyapunov_function_for_heat_equation:time_shifted_backwards_heat_kernel} and the chain rule establish Item~\eqref{lyapunov_function_for_heat_equation:item1}. 
	Moreover, note that \eqref{lyapunov_function_for_heat_equation:time_shifted_backwards_heat_kernel} ensures that for all 
		$ i \in \{ 1,2,\ldots,d \} $, 
		$ j \in \{ 1,2,\ldots,m \} $, 
	    $ t \in [0,T] $, 
	    $ x=(x_1,x_2,\ldots,x_d) \in \R^d $ 
	we have that 
		\begin{equation}
		\begin{gathered}
		(\tfrac{\partial}{\partial t}V)(t,x) = \alpha \left[ -\tfrac{d}{2(\alpha t+\varepsilon)}-\tfrac{\norm{x}^2}{2(\alpha t+\varepsilon)^2} \right] V(t,x), \qquad 
		(\tfrac{\partial}{\partial x_i} V)(t,x) = \tfrac{x_i}{\alpha t+\varepsilon} V(t,x), 
		\\
		\text{and}\qquad
		(\tfrac{\partial^2}{\partial x_i\partial x_j} V)(t,x) = \left[ \tfrac{x_i x_j}{(\alpha t + \varepsilon)^2} + \tfrac{\delta_{i,j}}{\alpha t + \varepsilon} \right] V(t,x). 
		\end{gathered}
		\end{equation}
	Hence, we obtain that for all 
		$ t \in [0,T] $, 
		$ x=(x_1,x_2,\ldots,x_d) \in \R^d $ 
	we have that 
		\begin{equation} \label{lyapunov_function_for_heat_equation:calculation}
		\begin{split} 
		& (\tfrac{\partial}{\partial t}V)(t,x) + \tfrac12 \operatorname{Trace}\!\big( B(t,x) [B(t,x)]^{*} (\operatorname{Hess}_x V)(t,x) \big) 
		\\
		& = 
		(\tfrac{\partial}{\partial t}V)(t,x) + \tfrac12 \left[ \smallsum\limits_{i,j=1}^{d} \smallsum\limits_{k=1}^{m}  b_{i,k}(t,x) b_{j,k}(t,x) (\tfrac{\partial^2}{\partial x_i\partial x_j} V) (t,x) \right]
		\\
		& = 
		\alpha \left[ -\tfrac{d}{2(\alpha t+\varepsilon)}-\tfrac{\norm{x}^2}{2(\alpha t+\varepsilon)^2} \right] V(t,x) 
		+ 
		\tfrac12 \left[ \smallsum\limits_{i,j=1}^{d} \smallsum\limits_{k=1}^{m} b_{i,k}(t,x) b_{j,k}(t,x)\left(
		\tfrac{x_i x_j}{(\alpha t + \varepsilon)^2} 
		+ 
		\tfrac{\delta_{i,j}}{\alpha t + \varepsilon} 
		\right) 
		\right]
		V(t,x) 
		\\
		& = 
		\alpha \left[ -\tfrac{d}{2(\alpha t+\varepsilon)}-\tfrac{\norm{x}^2}{2(\alpha t+\varepsilon)^2} \right] V(t,x) 
		+ 
		\tfrac12 \left[ \tfrac{ \langle x, B(t,x)[B(t,x)]^{*} x \rangle }{(\alpha t + \varepsilon)^2} + \tfrac{\operatorname{Trace}(B(t,x)[B(t,x)]^{*})}{\alpha t + \varepsilon} \right]\!.
		\end{split}
		\end{equation}
	Next note that the assumption that for all 
		$ t\in [0,T] $, 
		$ x,\xi \in \R^d $ 
	we have that 
		$ \langle \xi, B(t,x)[B(t,x)]^{*} \xi \rangle \leq c \Norm{\xi}^2 $ 
	implies that for all 
		$ t \in [0,T] $, 
		$ x \in \R^d $ 
	we have that 
		$ \operatorname{Trace}(B(t,x)[B(t,x)]^{*}) \leq c d $. 
	Combining this with \eqref{lyapunov_function_for_heat_equation:calculation} and the assumption that $ c \leq \alpha $ ensures that for all 
		$ t \in [0,T] $, 
		$ x \in \R^d $ 
	we have that 
		\begin{equation}
		\begin{split}
		& 
		(\tfrac{\partial}{\partial t}V)(t,x) + \tfrac12 \operatorname{Trace}\!\big( B(t,x) [B(t,x)]^{*} (\operatorname{Hess}_x V)(t,x) \big) 
		\\ &
		\leq 
		-\alpha \left[ \tfrac{d}{2(\alpha t+\varepsilon)} + \tfrac{\norm{x}^2}{2(\alpha t+\varepsilon)^2} \right] V(t,x) 
		+ 
		\tfrac12 \left[ \tfrac{c\norm{x}^2}{(\alpha t+\varepsilon)^2} + \tfrac{cd}{\alpha t+\varepsilon} \right] V(t,x)
		\\
		& = 
		(-\alpha+c) \left[ \tfrac{d}{2(\alpha t+\varepsilon)} + \tfrac{\norm{x}^2}{2(\alpha t+\varepsilon)^2} \right] V(t,x) 
		\leq 0. 
		\end{split} 
		\end{equation} 	
	This establishes Item~\eqref{lyapunov_function_for_heat_equation:item2}. 
	This completes the proof of  \cref{lem:lyapunov_function_for_heat_type_equation}.
\end{proof}

\begin{cor}\label{cor:existence_bounded_heat_equation_type}
	Let $ d,m \in \N $, 
		$ B \in \R^{d\times m} $, 
		$ a,T \in (0,\infty) $,
		$ c,L \in \R $,
    	$ f \in C( [0,T] \times \R^d \times \R , \R) $, 
    	$ g \in C(\R^d, \R) $, 
    let $ \langle\cdot,\cdot\rangle \colon \R^d\times\R^d \to \R $ be the standard Euclidean scalar product on $ \R^d $, 
    let $ \norm{\cdot}\colon\R^d\to [0,\infty) $ be the standard Euclidean norm on $\R^d$,
	assume for all 
		$ t \in [0,T] $, 
		$ x \in \R^d $, 
		$ v,w \in \R $ 
	that
		$ c = \sup \{ \langle y, B B^{*} y \rangle \colon y \in \R^d, \Norm{y}=1 \} < \frac{1}{2aT} $, 
		$ | f(t,x,0) | + | g(x) | \leq L \exp(a\Norm{x}^2) $, 
	and 
		$ | f(t,x,v) - f(t,x,w) | \leq L | v - w | $, 
	let $ (\Omega,\cF,\P) $ be a probability space, 
	and let $ W \colon [0,T] \times \Omega \to \R^m $ be a standard Brownian motion. 
	Then
		\begin{enumerate}[(i)]
		\item \label{existence_bounded_heat_equation_type:item2}
		there exists a unique viscosity solution 
			$ u \in \big( \bigcup_{ b \in \R } \{ {\bf u} \in C([0,T]\times\R^d,\R) \colon 
			\sup\{ \frac{|{\bf u}(t,x)|}{\exp(b\Norm{x}^2)} \colon t\in [0,T], x \in \R^d \} < \infty \} \big) $
		of 
			\begin{equation} \label{existence_bounded_heat_equation_type:claim2}
			(\tfrac{\partial}{\partial t}u)(t,x) 
			+
			\tfrac{1}{2} \operatorname{Trace}\!\big( B B^{*} (\operatorname{Hess}_x u)(t,x) \big) 
			+
			f(t, x, u(t, x))
			=
			0
			\end{equation}
		with $ u(T,x) = g(x) $ for $ (t,x) \in (0,T) \times \R^d $,  	
	\item \label{existence_bounded_heat_equation_type:item3}
	there exists a unique 
		$ v \in \big( \bigcup_{ \varepsilon \in (0,\infty) } \{ {\bf u} \in C([0,T]\times\R^d,\R) \colon \sup \{ |{\bf u}(t,x)| \exp(-\frac{\Norm{x}^2}{2(ct+\varepsilon)}) \colon t\in [0,T], x \in \R^d \} < \infty \} \big) $
	which satisfies for all 
		$ t \in [0,T] $, 
		$ x \in \R^d $		
	that 
		$ \EXP{|g(x+B W_{T-t})| + \int_t^T |f(s,x+BW_{s-t},v(s,x+BW_{s-t}))|\,ds}<\infty $
	and 
		\begin{equation}
		\begin{split} 
		& v(t, x) =	\Exp{ g( x + BW_{T - t} ) + \int_{t}^T f( s, x + BW_{s-t},  v(s, x + BW_{s-t}) ) \, ds }\!,  
		\end{split} 
		\end{equation}
	and 
	\item\label{existence_bounded_heat_equation_type:item4} 
	we have for all 
		$t\in [0,T]$, 
		$x\in\R^d$ 
	that 
		$u(t,x)=v(t,x)$ 
	\end{enumerate}
	(cf.~\cref{def:viscosity_solution}). 
\end{cor}

\begin{proof}[Proof of \cref{cor:existence_bounded_heat_equation_type}]
	Throughout this let $ V_{\varepsilon} \colon [0,T]\times\R^d\to (0,\infty) $, $ \varepsilon\in (0,\infty) $, 
	satisfy for all
		$ \varepsilon \in (0,\infty) $, 
		$ t \in [0,T] $,
		$ x \in \R^d $
	that 
    	\begin{equation} \label{existence_bounded_heat_equation_type:definition_of_V_epsilon_function}
    	V_{\varepsilon}(t,x) 
     	= 
		\tfrac{ 1 }{ [2\pi( c t+\varepsilon )]^{ d \slash 2 } }
     	\exp\!\left( 
        \tfrac{\norm{x}^2}{2( c t+\varepsilon )}
     	\right)
    	\end{equation}
 	and let $ \mc V_{ \varepsilon } \subseteq C([0,T]\times\R^d,\R) $, $ \varepsilon \in (0,\infty) $, satisfy for all 
 		$ \varepsilon \in (0,\infty) $ 
	that
		\begin{equation} \label{existence_bounded_heat_equation_type:definition_of_V_epsilon_set}
     	\mc V_{ \varepsilon } 
     	= 
     	\left\{ 
      	{\bf u} \in C([0,T]\times\R^d,\R) 
      	\colon 
      	\left(
      	\limsup_{r\to\infty} 
      	\left[ 
      	 \sup_{t\in [0,T]}
      	 \sup_{x\in \R^d,\norm{x}>r} 
      	 \left( 
        	\frac{|{\bf u}(t,x)|}{V_{\varepsilon}(t,x)}
       	\right) 
      	\right] 
      	= 0 
      	\right)
     	\right\}\!. 
   		\end{equation}
   	Observe that the assumption that for all 
		$ t \in [0,T] $, 
		$ x \in \R^d $ 
	we have that 
		$ |f(t,x,0)| + | g(x) | \leq L \exp(a\Norm{x}^2) $ 
	ensures that for all 
		$ \varepsilon \in (0,\infty) $, 
		$ t \in [0,T] $, 
		$ x \in \R^d $
	we have that 
		\begin{equation} 
		\begin{split} 
		\frac{| f(t,x,0) | + | g(x) |}{V_{\varepsilon}(t,x)}
		& \leq 
		L \, [ 2\pi ( c t + \varepsilon ) ]^{\nicefrac{d}{2}}\, \exp\!\left( a \Norm{x}^2 - \tfrac{\Norm{x}^2}{2( c t + \varepsilon)}\right) 
		\\
		& \leq 
		L \, [ 2\pi ( c T + \varepsilon ) ]^{\nicefrac{d}{2}}\, \exp\!\left( \big( a  - \tfrac{1}{2( c t + \varepsilon)} \big)  \Norm{x}^2  \right)
		\\[1ex]
		& \leq 
		L \, [ 2\pi ( c T + \varepsilon ) ]^{\nicefrac{d}{2}}\, \exp\!\left( \big( a  - \tfrac{1}{2( c T + \varepsilon)} \big)  \Norm{x}^2  \right)\!.
		\end{split} 
		\end{equation} 
	Hence, we obtain that for all 
		$ \varepsilon \in (0,\frac{1}{2 a} - c T) $ 
	we have that 
		\begin{equation} \label{existence_bounded_heat_equation_type:checking_growth_ass}
		\begin{split} 
		& 
		\limsup_{ r \to \infty} 
		\left[ 
		\sup_{t \in [0,T]} 
		\sup_{x \in \R^d, \norm{x}> r} 
		\left(  
		\frac{|f(t,x,0)|}{V_{\varepsilon}(t,x)} + \frac{|g(x)|}{V_{\varepsilon}(T,x)} 
		\right)  
		\right] 
		\\
		& \leq 
		2L \, [ 2\pi ( c T + \varepsilon ) ]^{\nicefrac{d}{2}} 
		\left[ \limsup_{ r \to \infty} \left[ \exp\!\Big( \big( a  - \tfrac{1}{2( c T + \varepsilon)} \big) r^2 \Big) \right] \right] = 0 . 
		\end{split} 
		\end{equation} 
	Moreover, note that \eqref{existence_bounded_heat_equation_type:definition_of_V_epsilon_function} demonstrates that for all 
		$ \varepsilon \in (0,\infty) $, 
		$ t \in [0,T] $, 
		$ x \in \R^d $ 
	we have that 
		\begin{equation} \label{cor:existence_bounded_heat_equation_type:lower_bound_for_V_epsilon_function}
		V_{\varepsilon}(t,x) 
		\geq 
		\tfrac{1}{[2\pi( c T+\varepsilon)]^{d \slash 2}} \exp\!\left( \tfrac{\norm{x}^2}{2( c T+\varepsilon)} \right)\!.  
		\end{equation} 
	Hence, we obtain for all 
		$ \varepsilon \in (0,\infty) $ 
	that 
		\begin{equation} \label{existence_bounded_heat_equation_type:growth_of_V}
		\liminf_{ r \to \infty } \left[ \inf_{t\in [0,T]} \inf_{x\in\R^d,\norm{x}>r} V_{\varepsilon}(t,x) \right] = \infty. 
		\end{equation} 
	In addition, observe that \cref{lem:lyapunov_function_for_heat_type_equation} guarantees that for all 
 		$ \varepsilon \in (0,\infty)$, 
    	$ t \in [0,T] $, 
    	$ x \in \R^d $ 
	we have that 
		\begin{equation} \label{existence_bounded_heat_equation_type:V_epsilon_supersolution}
	  	( \tfrac{\partial}{\partial t} V_{\varepsilon} )(t,x) + \tfrac12 \operatorname{Trace}\!\big( B B^{*} (\operatorname{Hess}_x V_{\varepsilon})(t,x) \big) \leq 0. 
	 	\end{equation}
	Combining this with \eqref{existence_bounded_heat_equation_type:checking_growth_ass}, \eqref{existence_bounded_heat_equation_type:growth_of_V}, and Item~\eqref{existence_of_fixpoint:item2} in \cref{thm:existence_of_fixpoint} (applied with 
		$ \mc O \is \R^d $, 
		$ \mu \is ([0,T]\times\R^d \ni (t,x) \mapsto 0 \in \R^d ) $, 
		$ \sigma \is ([0,T]\times\R^d \ni (t,x) \mapsto B \in \R^{d\times m} ) $, 
		$ V \is V_{\varepsilon} $
	for $ \varepsilon \in ( 0, \frac{1}{2a} - c T ) $ 		
	in the notation of \cref{thm:existence_of_fixpoint}) proves that for every 
		$ \varepsilon \in ( 0, \frac{1}{2a} - c T) $ 
 	there exists a unique viscosity solution
 		$ U_{ \varepsilon } \in \mc V_{ \varepsilon } $ 
	of		
		\begin{equation} \label{existence_bounded_heat_equation_type:viscosity_solution_in_V_epsilon}
		(\tfrac{\partial }{\partial t}U_{ \varepsilon })(t,x) 
		+ 
		\tfrac12 \operatorname{Trace}\!\big( B B^{*} (\operatorname{Hess}_x U_{ \varepsilon })(t,x) \big) 
		+ 
		f(t,x,U_{ \varepsilon }(t,x)) 
		= 0
		\end{equation}
 	with $ U_{ \varepsilon }(T,x) = g(x) $ for $ (t,x) \in (0,T)\times\R^d $. 
	This, \eqref{existence_bounded_heat_equation_type:definition_of_V_epsilon_set}, and \eqref{cor:existence_bounded_heat_equation_type:lower_bound_for_V_epsilon_function} ensure that for all 
 		$ \varepsilon \in ( 0, \frac{1}{2a} - c T ) $ 
 	we have that 
	 	\begin{equation}
	 	\label{existence_bounded_heat_equation_type:u_viscosity_solution}
	 	\sup_{t\in [0,T]} \sup_{x\in \R^d} \left[ \frac{|U_{ \varepsilon }(t,x)|}{\exp\!\left( \frac{\Norm{x}^2}{2\varepsilon} \right)} \right] 
	 	\leq 
	 	\frac{1}{[2\pi\varepsilon]^{ d \slash 2 }} 
	 	\sup_{t\in [0,T]} \sup_{x\in \R^d} \left[ \frac{|U_{ \varepsilon }(t,x)|}{V_{ \varepsilon }(t,x)} \right] 	 	
	 	< \infty.
	 	\end{equation} 
	Moreover, observe that \eqref{existence_bounded_heat_equation_type:definition_of_V_epsilon_function} yields for all 
		$ t \in [0,T] $, 
		$ x \in \R^d $, 
		$ \delta,\varepsilon \in ( 0, \infty ) $ 
	with 
		$ \delta \leq \varepsilon $ 
	that 
		\begin{equation} \label{existence_bounded_heat_equation_type:comparing_V_epsilon_and_V_delta}
		\frac{ V_{\varepsilon}(t,x) }{ V_{\delta}(t,x) } 
		= \left[ \frac{ 2\pi (ct+\delta) }{ 2 \pi (ct+\varepsilon)} \right]^{ d \slash 2 } 
		\exp\!\left( \norm{x}^2 \big( \tfrac{1}{ 2(ct+\varepsilon) } - \tfrac{ 1 }{ 2(ct+\delta) } \big) \right)
		\leq \left[ \frac{ cT + \delta }{ \varepsilon } \right]^{\nicefrac{d}{2}}. 
		\end{equation}  
	This implies for all 
		$ \delta,\varepsilon \in ( 0, \frac{1}{2a} - c T ) $ 
	with 
		$ \delta \leq \varepsilon $ 
	that 
		$ \mc V_{\varepsilon} \subseteq \mc V_{\delta} $. 
 	Combining this with \eqref{existence_bounded_heat_equation_type:viscosity_solution_in_V_epsilon} demonstrates that for all 
 		$ \delta,\varepsilon \in ( 0, \frac{1}{2a} - c T ) $ 
 	we have that 
 		$ U_{\varepsilon} = U_{\delta} $. 
	This proves that there exists a unique $ u \in C([0,T]\times\R^d,\R) $ which satisfies for all 
 		$ \varepsilon \in ( 0, \frac{1}{2a} - c T ) $, 
 		$ t \in [0,T] $, 
 		$ x \in \R^d $ 
 	that 
 		\begin{equation} \label{existence_bounded_heat_equation_type:existence_of_u_function}
 		u(t,x) = U_{\varepsilon}(t,x).  
 		\end{equation} 
 	Note that \eqref{existence_bounded_heat_equation_type:viscosity_solution_in_V_epsilon},  \eqref{existence_bounded_heat_equation_type:u_viscosity_solution}, and \eqref{existence_bounded_heat_equation_type:existence_of_u_function} ensure that 
 		$ u \in \big( \bigcup_{ b \in \R } \{ {\bf u} \in C([0,T]\times\R^d,\R) \colon \sup\{ \frac{|{\bf u}(t,x)|}{\exp(b\Norm{x}^2)} \colon \allowbreak t\in [0,T], x \in \R^d \} < \infty \} \big) $ 
 	is a viscosity solution of 
 		\begin{equation} \label{existence_bounded_heat_equation_type:existence_of_viscosity_solution}
 		(\tfrac{\partial }{\partial t}u)(t,x) + \tfrac12 \operatorname{Trace}\!\big( B B^{*} (\operatorname{Hess}_x u)(t,x) \big) + f(t,x,u(t,x)) = 0
 		\end{equation} 
 	with $ u(T,x) = g(x) $ for $ (t,x) \in (0,T)\times\R^d $. 
 	Next let 
 		$ v \in \big( \bigcup_{ b \in \R } \{ {\bf u} \in C([0,T]\times\R^d,\R) \colon \sup\{ \frac{|{\bf u}(t,x)|}{\exp(b\Norm{x}^2)} \colon \allowbreak t\in [0,T], x \in \R^d \} < \infty \} \big) $ 
 	be a viscosity solution of 
	 	\begin{equation} 
	 	(\tfrac{\partial}{\partial t} v)(t,x) + \tfrac12 \operatorname{Trace}\!\left(BB^{*}(\operatorname{Hess}_x v)(t,x)\right) 
	 	+ f(t,x,v(t,x)) = 0
	 	\end{equation}
 	with $ v(T,x) = g(x) $ for $ (t,x)\in (0,T)\times\R^d $, 
 	let $ b \in (0,\infty) $ satisfy $ \sup_{t\in [0,T]} \sup_{x\in \R^d} (\frac{|v(t,x)|}{\exp(b\|x\|^2)}) < \infty $, 
 	and let $ \mc T \subseteq [0,T] $ satisfy 
		\begin{equation} \label{existence_bounded_heat_equation_type:definition_of_T_set}
		\mc T = \{ t \in [0,T]\colon v|_{[t,T]\times\R^d} = u|_{[t,T]\times \R^d} \}. 
		\end{equation}
 	Observe that \eqref{existence_bounded_heat_equation_type:definition_of_T_set} and the fact that for all 
 		$ x \in \R^d $ 
 	we have that 
 		$ u(T,x) = g(x) = v(T,x) $ 
 	imply that $ T \in \mc T $. 
 	Moreover, note that the fact that $ u,v \in C([0,T]\times\R^d,\R) $ assures that $ \mc T $ is closed. 
	In addition, observe that \eqref{existence_bounded_heat_equation_type:definition_of_V_epsilon_function} and \eqref{existence_bounded_heat_equation_type:definition_of_T_set} ensure that for all 
	 	$ t \in [0,T] $,
	 	$ s \in [0,\min\{t,\frac{1}{4bc}\}] $, 
	 	$ x \in \R^d $, 
	 	$ \varepsilon \in ( 0, \nicefrac{1}{4b} ) $ 
 	we have that 
		\begin{equation}
		\begin{split} 
		\tfrac{| v(\max\{ t-\frac{1}{4bc}, 0 \} + s, x) |}{ V_{\varepsilon}(s,x) } 
		&= 
		[2\pi (cs + \varepsilon)]^{\nicefrac{d}{2}} 
		\left[ \tfrac{|v(\max\{ t-\frac{1}{4bc}, 0 \} + s, x)|}{\exp(b\Norm{x}^2)} \right] \exp\!\left( (b-\tfrac{1}{2(cs+\varepsilon)})\norm{x}^2 \right)
		\\
		& 
		\leq 
		[2\pi (cT+\varepsilon)]^{\nicefrac{d}{2}} 
		\left[ \tfrac{|v(\max\{ t-\frac{1}{4bc}, 0 \} + s,  x)|}{\exp(b\Norm{x}^2)} \right] \exp\!\left( (b-\tfrac{1}{2\varepsilon+\nicefrac{1}{2b}})\norm{x}^2 \right)\!.
		\end{split} 
		\end{equation}
	Hence, we obtain for all 
		$ t \in [0,T] $,
		$ \varepsilon \in ( 0, \nicefrac{1}{4b} ) $ 
	that 
		\begin{equation}
		\limsup_{r\to\infty} \left[ \sup_{ s \in [0,\min\{t,\frac{1}{4bc}\}] } \sup_{ x \in \R^d, \norm{x}>r } \left( \tfrac{|v(\max\{ t-\frac{1}{4bc}, 0 \} + s,x)|}{ V_{\varepsilon}(s,x) } \right) \right] = 0. 
		\end{equation}
	Item~\eqref{existence_of_fixpoint:item2} in \cref{thm:existence_of_fixpoint} (applied with 
		$ T \is \min\{ t, \frac{1}{4bc} \} $, 
		$ \mc O \is \R^d $, 
		$ \mu \is ( [0, \min\{ t, \frac{1}{4bc} \} ]\times\R^d \ni (s,x) \mapsto 0 \in \R^d ) $, 
		$ \sigma \is ( [0, \min\{ t, \frac{1}{4bc} \} ]\times\R^d \ni (s,x)\mapsto B \in \R^{d\times m} ) $, 
		$ f \is ( [0, \min\{ t, \frac{1}{4bc} \} ]\times\R^d\times\R \ni (s,x,v) \mapsto f( \max\{ t-\frac{1}{4bc}, 0 \} + s,x,v ) \in \R ) $, 
		$ g \is ( \R^d \ni x \mapsto u(t,x) \in \R ) $, 
		$ V \is V_{\varepsilon} $
	for $ t \in \mc T \cap (0,T] $, 
		$ \varepsilon \in (0,\frac{1}{4b})\cap (0,\frac{1}{2a}-cT)$ 
	in the notation of \cref{thm:existence_of_fixpoint}), the fact that for all 
		$ \varepsilon \in ( 0, \frac{1}{2a}-cT )$ 
	we have that 
		$ U_{\varepsilon} = u $, 
	and \eqref{existence_bounded_heat_equation_type:viscosity_solution_in_V_epsilon} therefore demonstrate that for all 
	 	$ t \in \mc T \cap (0,T] $, 
		$ \varepsilon \in (0,\frac{1}{4b})\cap (0,\frac{1}{2a}-cT)$ 
	we have that 
 		$ v|_{[\max\{t-\frac{1}{4b},0\},t]\times\R^d} = U_{\varepsilon}|_{[\max\{t-\frac{1}{4b},0\},t]\times\R^d} $. 
 	This and \eqref{existence_bounded_heat_equation_type:existence_of_u_function} ensure that for all 
 		$ t \in \mc T \cap (0,T] $ 
	we have that 
		$ v|_{[\max\{t-\frac{1}{4b},0\},t]\times \R^d} = u|_{[\max\{t-\frac{1}{4b},0\},t]\times \R^d} $. 
	Hence, we obtain for all
		$ t \in \mc T \cap (0,T] $ 
	that $[\max\{t-\frac{1}{4b},0\},t] \subseteq \mc T$. 
	This implies that $ \mc T \in \{ A \subseteq [0,T] \colon ( \Forall a \in A \colon \Exists \varepsilon \in (0,\infty) \colon (a-\varepsilon,a+\varepsilon) \cap [0,T] \subseteq A )  \}$. 
	Combining this with the fact that $ \mc T $ is non-empty, the fact that $ \mc T $ is closed, and the fact that $ [0,T] $ is connected ensures that $ \mc T = [0,T] $. 
	This and \eqref{existence_bounded_heat_equation_type:existence_of_viscosity_solution} establish Item~\eqref{existence_bounded_heat_equation_type:item2}. 
	Next we prove Items~\eqref{existence_bounded_heat_equation_type:item3} and \eqref{existence_bounded_heat_equation_type:item4}.  
 	Observe that 
 	\eqref{existence_bounded_heat_equation_type:definition_of_V_epsilon_function},
 	\eqref{existence_bounded_heat_equation_type:definition_of_V_epsilon_set},
 	\eqref{existence_bounded_heat_equation_type:checking_growth_ass}, 
 	\eqref{existence_bounded_heat_equation_type:growth_of_V},
 	\eqref{existence_bounded_heat_equation_type:V_epsilon_supersolution},
 	\eqref{existence_bounded_heat_equation_type:viscosity_solution_in_V_epsilon},  \eqref{existence_bounded_heat_equation_type:existence_of_u_function}, 
 	and Items~\eqref{existence_of_fixpoint:item3} and \eqref{existence_of_fixpoint:item4} in \cref{thm:existence_of_fixpoint} (applied with 
 		$ \mc O \is \R^d $, 
 		$ \mu \is ( [0,T]\times\R^d \ni (t,x) \mapsto 0 \in \R^d ) $, 
 		$ \sigma \is ( [0,T]\times\R^d \ni (t,x) \mapsto B \in \R^{d\times m} ) $, 
 		$ V \is V_{\varepsilon} $
 	for $ \varepsilon \in (0,\frac{1}{2a}-cT) $
 	in the notation of \cref{thm:existence_of_fixpoint})  guarantee that for all 
 		$ t \in [0,T] $,
 		$ x \in \R^d $
	we have that 
 		\begin{equation} 
  		u(t,x) = \Exp{ g(x+B(W_T-W_t)) + \int_t^T f(s, x+B(W_s-W_t),u(s,x+B(W_s-W_t))) \,ds }\!. \end{equation}
	The assumption that $ W $ is a standard Brownian motion and Fubini's theorem therefore ensure that for all 
		$ t \in [0,T] $,
		$ x \in \R^d $
	we have that 
		\begin{equation} \label{existence_bounded_heat_equation_type:fixpoint_equation_for_u}
		u(t,x) = \Exp{ g(x+BW_{T-t}) + \int_t^T f(s, x+BW_{s-t},u(s,x+BW_{s-t}) ) \,ds }\!. \end{equation}
	Next let 
		$ w \in \big( \bigcup_{ \varepsilon \in (0,\infty) } \{ {\bf u} \in C([0,T]\times\R^d,\R) \colon \sup \{ |{\bf u}(t,x)| \exp(-\frac{\Norm{x}^2}{2(ct+\varepsilon)}) \colon t\in [0,T], x \in \R^d \} < \infty \} \big) $ 
	satisfy for all 
		$ t \in [0,T] $, 
		$ x \in \R^d $ 
	that
		$ \EXP{|g(x+B W_{T-t})| + \int_t^T |f(s,x+BW_{s-t},w(s,x+BW_{s-t}))|\,ds}<\infty $ 
	and 
		\begin{equation} \label{existence_bounded_heat_equation_type:fixpoint_equation_for_w}
		\begin{split} 
		& w(t, x) =	\Exp{ g( x + BW_{T - t} ) + \int_{t}^T f( s, x + BW_{s-t},  w(s, x + BW_{s-t}) ) \, ds }
		\end{split} 
		\end{equation}
	and let $ \eta \in (0,\infty) $ satisfy  
		$ \sup_{t\in [0,T]} \sup_{x\in\R^d} ( |w(t,x)|\exp(-\frac{\norm{x}^2}{2(ct+\eta)}) ) < \infty $. 
	Observe that \eqref{existence_bounded_heat_equation_type:comparing_V_epsilon_and_V_delta} demonstrates that for all 
		$ \varepsilon \in (0,\eta) \cap (0,\frac{1}{2a}-cT) $ 
	we have that 
		$ w \in \mc V_{\varepsilon} $. 
	Combining this,  \eqref{existence_bounded_heat_equation_type:fixpoint_equation_for_u}, and 
	\eqref{existence_bounded_heat_equation_type:fixpoint_equation_for_w} 
	with 
	the fact that $ W $ is a standard Brownian motion,
	the fact that for all 
		$ \varepsilon \in (0,\frac{1}{2a}-cT) $ 
	we have that 
		$ u \in \mc V_{\varepsilon} $, 
	and Item~\eqref{existence_of_fixpoint:item3} in \cref{thm:existence_of_fixpoint} proves that $ u = w $. 	
	This establishes Items~\eqref{existence_bounded_heat_equation_type:item3} and  \eqref{existence_bounded_heat_equation_type:item4}. 
 	This completes the proof of  \cref{cor:existence_bounded_heat_equation_type}.
\end{proof}

\subsection*{Acknowledgements} 

The third author acknowledges funding by the Deutsche Forschungsgemeinschaft (DFG, German Research Foundation) under Germany's Excellence Strategy EXC 2044-390685587, Mathematics Muenster: Dynamics-Geometry-Structure. 

\bibliographystyle{acm}
\bibliography{PDE_approximation_bibfile}
\end{document}